\documentclass{article}
\usepackage{graphicx} 
\usepackage{setspace}
\usepackage{amsthm,amsmath,amssymb}
\usepackage{graphicx}
\usepackage{mathrsfs}
\usepackage{tikz}
\usepackage{authblk}
\usetikzlibrary{cd}
\usepackage{tikz-cd}
\usepackage{geometry}
\usepackage{subcaption}
\usepackage{verbatim}
\usepackage{amsfonts}
\usepackage{float}
\usepackage[utf8]{inputenc}
\usepackage{amssymb}
\usepackage{amsmath}
\usepackage{amsfonts}
\numberwithin{equation}{section}
\usepackage{amsmath,amscd}
\everymath{\displaystyle}
\usepackage{indentfirst} 
\usepackage{enumerate}
\usepackage[colorlinks]{hyperref}
\usepackage{bm}
\usepackage{mathtools}
\usepackage{algorithm}
\usepackage{algorithmic}
\usepackage{mathrsfs}
\usepackage{esint}
\usepackage[toc,page,title,titletoc,header]{appendix}
\newtheorem{theorem}{Theorem}[section]
\newtheorem{prop}[theorem]{Proposition}
\newtheorem{remark}[theorem]{Remark}
\newtheorem{definition}[theorem]{Definition}
\newtheorem{lemma}[theorem]{Lemma}

\newtheorem{cor}[theorem]{Corollary}

\newcommand{\ad}{{\rm ad}}

\newcommand{\diag}{{\rm diag}}
\newcommand{\I}{{\mathrm{i}}} 

\everymath{\displaystyle}
\usepackage{graphicx} 
\title{Asymptotic and monodromy problems for higher-order Painlevé III equations}

\author{Zikang Wang and Xiaomeng Xu}
\date{}

\begin{document}

\maketitle

\begin{abstract}
In this paper, we study the isomonodromy deformation equations for the $n\times n$ system of first order meromorphic linear ordinary differential equations with two second order poles. We analyze the asymptotic behaviour of the solutions at a boundary point of the isomonodromic deformation space, and derive a parameterization of the solutions via asymptotic parameters. We then derive the explicit formula for the Stokes matrices and connection matrix of the associated linear system in terms of the asymptotic parameters. In the end, we apply the results to the study of the $tt^{*}$ equations.
\end{abstract}
\section{Introduction}

In this paper, we study the isomonodromy deformation for the $n\times n$ system of first order meromorphic linear ordinary differential equations with
two second order poles (at $\xi=0$ and $\xi=\infty$)
\begin{align}\label{introeq}
    \frac{d F}{d \xi}=\left(U+\frac{A}{\xi}+\frac{GVG^{-1}}{\xi^2}\right)\cdot F(\xi),
\end{align}
where $A\in\frak{gl}_n$, $G\in {\rm GL}_n$, and $U=\diag(u_1,\ldots,u_n)$, $V=\diag(v_1,\ldots,v_n)$ are diagoal matrices with $u_i\ne u_j$ and $v_i\ne v_j$ for $i\ne j$. 

The isomonodromy equations are a $(u_1,...,u_n,v_1,...,v_n)$-time dependent Hamiltonian system on the symplectic space $T^*{\rm GL}_n$.
In terms of the pair of matrices $(G,A)\in T^*{\rm GL}_n\cong {\rm GL}_n \times \frak{gl}_n$, the systems can be written as a completely integrable nonlinear partial differential equation for $n\times n$
matrix valued function function $A(\mathbf{u}, \mathbf{v})$ and $n\times n$
invertible matrix valued function $G(\mathbf{u}, \mathbf{v})$, with $2n$ complex variables $\mathbf{u}=(u_1,\ldots,u_n)$ and $\mathbf{v}=(v_1,...,v_n)$,
\begin{align}\label{eq:iso eq of two irr A-u}
    \frac{\partial A}{\partial u_k} &=[\mathrm{ad}_U^{-1}\mathrm{ad}_{E_k}A,A]+{\rm ad}_{E_k}(GVG^{-1}),\\
    \frac{\partial A}{\partial v_k} &= \mathrm{ad}_U(GE_kG^{-1}),\\
        \frac{\partial G}{\partial u_k} &= (\mathrm{ad}_U^{-1}\mathrm{ad}_{E_k}A)\cdot G,\\ \label{eq:iso eq of two irr G-v}
    \frac{\partial G}{\partial v_k} &= G\cdot (\mathrm{ad}_V^{-1}\mathrm{ad}_{E_k}(G^{-1}AG)).
\end{align}
Here $E_k={\rm diag} (0, \ldots ,
\underset{k \text{-th}}{1} 
, \ldots, 0)$, and for any $n\times n$ matrix $A$, $\ad_u^{-1} A$ is a $n\times n$ matrix with the $(i,j)$-entry
\begin{align}
\label{Nota:adu-1}
(\ad_U^{-1} A)_{ij}:=
\left\{\begin{array}{ll}
    \frac{1}{u_i - u_j} A_{i j} & ; i \neq j\\
    0 & ; i = j
  \end{array}\right..
\end{align}

We study this system for two primary reasons. On the one hand, the system lies at the crossroad of many subjects. It is closely related to the $tt^*$ equation \cite{cecotti1991topological, dubrovin1993tteq}, the double symplectic groupoid \cite{lu1989groupoides, boalch2007quasi}, tensor product rule in representation theory and so on. On the other hand, the solutions of the system has strong Painlev\'e property: following Miwa \cite{Miwa1981}, 
the solutions $A(\mathbf{u}, \mathbf{v})$ and $G(\mathbf{u}, \mathbf{v})$ of \eqref{eq:iso eq of two irr A-u}-\eqref{eq:iso eq of two irr G-v}
are multi-valued meromorphic functions 
of $u_1,\ldots,u_n$ and $v_1,...,v_n$, where
the branching occurs 
when $u$ or $v$ moves along a loop around
\begin{equation}\label{Deltauv}
\Delta_{\mathbf{u}, \mathbf{v}}=\{(u_1,...,u_n,v_1,...,v_n)\in \mathbb{C}^{2n}~|~u_i = u_j \text{ or } v_i = v_j, \,\text{for some } i\ne j \}.
\end{equation}
Thus, according to the original idea of Painlev\'{e}, 
they can be a new class of special functions. 

In particular, when the matrix order $n=2$, the above isomonodromy equations are equivalent to the Painlevé III equation \cite{Painleve1902}. 
We refer the reader to the book of Fokas, Its, Kapaev and Novokshenov \cite{FIKN2006} for a thorough introduction to the history and developments of the study of Painlev\'{e} equations. As stressed in \cite{FIKN2006,Its-Novokshenov1986}, the solutions of Painlev\'{e} equations are seen as nonlinear special functions, because they play the same role in nonlinear mathematical physics as that of classical special functions, like Airy functions, Bessel functions, etc., in linear physics. And it is the parametrization of Painlev\'{e} transcendents by their asymptotic behaviour at critical points and the solution to the monodromy problem that 
make Painlev\'{e} transcendents as efficient in applications as linear special functions.

For $n>2$, following \cite{dubrovin1993tteq}, the above isomonodromy equations can be considered as a high-order generalisation of the Painlevé III.
In this paper, we generalize the results for the asymptotics and monodromy problems for the Painlevé III (see \cite{Jimbo1982}, \cite{FIKN2006}) to the general $n$ case. 

\subsection{Asymptotics of solutions of isomonodromy deformation equations}
Our first result is to establish the asymptotic behavior of almost all the multivalued meromorphic solutions of the isomonodromy deformation system \eqref{eq:iso eq of two irr A-u}-\eqref{eq:iso eq of two irr G-v} at an iterated limit, and to identify the asymptotic parameters by elements of the cotangent bundle $T^*{\rm GL}_n$ satisfying the following boundary condition.

\vspace{2mm}

{\bf Boundary Condition:} An element $( G_0,\widehat{A}_0)\in T^*{\rm GL}_n\cong {\rm GL}_n\times \frak{gl}_n$ is said to satisfy boundary condition, if 
\begin{align}\label{eq:eigenvalue cond of hatA tilA}
        \left|\mathrm{Re}(\widehat{\lambda}^{(k)}_{i}-\widehat{\lambda}^{(k)}_{j})\right|<1,\quad \left|\mathrm{Re}(\widetilde{\lambda}^{(k)}_{i}-\widetilde{\lambda}^{(k)}_{j})\right|<1,\   k=1,2,...,n;\ 1\leq i,j\leq k.
\end{align}
Here $\{\widehat{\lambda}_i^{(k)}\}_{i=1,...,k}$ and $\{\widetilde{\lambda}_i^{(k)}\}_{i=1,...,k}$ are the sets of eigenvalues of upper left $k\times k$ submatrices of the matrices $\widehat{A}_0$ and $\widetilde{A}_0:=-G_0^{-1}\widehat{A}_0G_0$ respectively.

\vspace{2mm}
To formulate our results, we introduce the operator $\delta_k A$, which extracts the upper-left $k\times k$ submatrix and diagonal elements of an $n\times n$ matrix $A$:
\begin{align}\label{deltak}
  (\delta_k A)_{i j} 
  & :=
  \left\{\begin{array}{ll}
    A_{i j} & ; 1 \leqslant i, j \leqslant k \quad
    \text{or}\quad i = j\\
    0 & ; \text{otherwise}
  \end{array}\right..\\
  \delta A&:=\delta_1A \text{ the diagonal part of } A
\end{align} 
We introduce a new coordinate system $(\mathbf{z},t,\mathbf{w})=(z_0,z_1,...,z_{n-1},t,w_{n-1},...,w_2,w_0)$:
\begin{align}\label{coor1}
    &z_0=u_1,z_1=u_2-u_1,z_2=\frac{u_3-u_1}{u_2-u_1},...,z_{n-1}=\frac{u_n-u_1}{u_{n-1}-u_1},\\ \label{coor2}
    &t=(v_n-v_1)(u_n-u_1),\\ \label{coor3}
   &w_{n-1}=\frac{v_{n}-v_1}{v_{n-1}-v_1},w_{n-2}=\frac{v_{n-1}-v_1}{v_{n-2}-v_1},...,w_{2}=\frac{v_3-v_1}{v_2-v_1},\ w_0=v_1.
\end{align}
We take the convention $w_1:=\frac{t}{z_1\cdots z_{n-1}w_{n-1}\cdots w_2}$. (The form of the above isomonodromy equations under this new  coordinate system is given by equations \eqref{iso for z begin}–\eqref{iso for z end} in the main text.) 

\begin{theorem}\label{mainthm}
For any $( G_0,\widehat{A}_0)\in {\rm GL}_n\times \frak{gl}_n$ satisfying the boundary condition, there exists a unique multi-valued meromorphic solution $A(\mathbf{z},t,\mathbf{w};\widehat{A}_0,G_0),G(\mathbf{z},t,\mathbf{w};\widehat{A}_0,G_0)$ of the isomonodromy equations \eqref{iso for z begin}–\eqref{iso for z end}, and a sequence of $n \times n$ matrix-valued functions $\widehat{A}_k(z_1,...,z_k)$, $\widetilde{A}_k(w_1,\ldots,w_k)$ for $k=1,\ldots,n-1$, and $\widehat{G}$, such that 

\begin{align*}
   &\lim_{t\rightarrow 0} A(\mathbf{z},t,\mathbf{w})=\widehat{A}_{n-1}(\mathbf{z}),\\
   &\lim_{t\rightarrow 0} t^{-\widehat{A}_{n-1}}G(\mathbf{z},t,\mathbf{w})\cdot w_1^{\delta(G^{-1}AG)}=\widehat{G}(\mathbf{z},\mathbf{w}), \\
    &\widetilde{A}_{n-1}(\mathbf{w})=-\widehat{G}^{-1}\widehat{A}_{n-1}\widehat{G},
\end{align*}
and
\begin{align*}
&\lim_{z_{k}\rightarrow\infty} \delta_{k}\widehat{A}_k=\delta_k\widehat{A}_{k-1}, \ \lim_{z_k\rightarrow\infty} z_k^{\delta_k\widehat{A}_{k-1}}\cdot\widehat{A}_{k}\cdot z_k^{-\delta_k\widehat{A}_{k-1}}=\widehat{A}_{k-1},\ k=2,...,n-1,\ \widehat{A}_1=z_1^{-\delta\widehat{A}_0}\cdot\widehat{A}_0\cdot z_1^{\delta\widehat{A}_0};\\
      &\lim_{w_{k}\rightarrow\infty} \delta_{k}\widetilde{A}_k=\delta_k\widetilde{A}_{k-1}, \ \lim_{w_k\rightarrow\infty} w_k^{\delta_k\widetilde{A}_{k-1}}\cdot\widetilde{A}_{k}\cdot w_k^{-\delta_k\widetilde{A}_{k-1}}=\widetilde{A}_{k-1},\ k=2,...,n-1,\ \widetilde{A}_1=\widetilde{A}_0;\\
      &\lim_{w_2\rightarrow\infty}\cdots\lim_{w_{n-1}\rightarrow\infty}\lim_{z_{2}\rightarrow\infty}\cdots\lim_{z_{n-1}\rightarrow\infty}\overrightarrow{\prod_{k=1}^{n-1}}\left(z_{k}^{\delta_{k}\widehat{A}_{k-1}}z_{k}^{-\widehat{A}_{k}}\right)\cdot\left(z_1\cdots z_{n-1}w_{n-1}\cdots w_2\right)^{\widehat{A}_{n-1}}\widehat{G}\cdot\overleftarrow{\prod_{k=2}^{n-1}}\left(w_{k}^{\widetilde{A}_{k}}w_{k}^{-\delta_{k}\widetilde{A}_{k-1}}\right)= G_0.
\end{align*}
Here in the last identity, we let $z_{n-1},...,z_2$ and $w_{n-1},...,w_2$
  tend to infinity successively.
We call $(\widehat{A}_0, G_0)$ the boundary value of the corresponding solution at the iterated limit point 
\begin{equation}
t\rightarrow 0, z_{n-1}\rightarrow\infty,... ,z_{2}\rightarrow \infty,w_{n-1}\rightarrow\infty, ..., w_{2}\rightarrow \infty.  
\end{equation}

Furthermore, the set $\mathfrak{Sol}_{Shr}$ of solutions $A(\widehat{A}_0,G_0),G(\widehat{A}_0,G_0)$, parameterized by $( G_0,\widehat{A}_0)\in {\rm GL}_n\times \frak{gl}_n$ satisfying the boundary condition, includes almost all the solutions of the isomonodromy equation.
\end{theorem}
It should first be noted that the matrix functions $\widehat{A}_{n-1}(\mathbf{z})$ and $\widetilde{A}_{n-1}(\mathbf{w})$, which depend only on the variables $\mathbf{z}$ and $\mathbf{w}$ respectively, satisfy the isomonodromy equations for meromorphic linear ODEs with one irregular singularity (equations \eqref{eq:iso eq of one irr in z}).
Our construction and characterization of the solutions proceed in three principal steps:
\begin{enumerate}
    \item In {Theorem \ref{thm:asy of iso A from hatA}}, we construct the solutions $(A(\mathbf{z},t,\mathbf{w}), G(\mathbf{z},t,\mathbf{w}))$ from the intermediate pair $\widehat{A}(\mathbf{z})=\widehat{A}_{n-1}(\mathbf{z})$ and $  \widehat{G}(\mathbf{z},\mathbf{w})$, which satisfy specific eigenvalue conditions and corresponding equations.
    \item In {Proposition \ref{prop:from A0,G0 to hatA, hatG}}, we subsequently construct the intermediate pair  $(\widehat{A}(\mathbf{z}), \widehat{G}(\mathbf{w}))$ from the boundary values $(\widehat{A}_0, G_0)$ that satisfy the boundary condition.
    \item Finally, in {Section \ref{sec:almost every is shrinking}}, we propose a criterion on the monodromy data (Definition \ref{def:log-conf cond}) that selects an open and dense subset, and we prove that all solutions satisfying this criterion belong to the set $\mathfrak{Sol}_{\text{Shr}}$.
\end{enumerate}

\subsection{The monodromy problem for the meromorphic linear ODEs}

The monodromy data of the ordinary differential equation \eqref{introeq} is consisting of a pair of Stokes matrices $S_{d,\pm}^{(0)}(U,V,A,G)$ at $\xi=0$, a pair of Stokes matrices $S_{d,\pm}^{(\infty)}(U,V,A,G)$ at $\xi=\infty$, and the connection matrix $C_d(A,G)$ from $\xi=0$ to $\xi=\infty$ (associated to a chosen admissible direction $d$). See Section \ref{sec:intro for two irr sys} for the definitions. 

Following \cite{JMU1981I}, when $A(\mathbf{u}, \mathbf{v}),G(\mathbf{u}, \mathbf{v})$ satisfy the isomonodromy equations \eqref{eq:iso eq of two irr A-u}-\eqref{eq:iso eq of two irr G-v}, the Stokes matrices $S_{d,\pm}^{(0)}(A(\mathbf{u}, \mathbf{v}),G(\mathbf{u}, \mathbf{v}))$ and $S_{d,\pm}^{(\infty)}(A(\mathbf{u}, \mathbf{v}),G(\mathbf{u}, \mathbf{v}))$ are locally constant (independent of $\mathbf{u}$ and $\mathbf{v}$). 
Therefore, the Stokes matrices are the first integrals of the nonlinear equation \eqref{eq:iso eq of two irr A-u}-\eqref{eq:iso eq of two irr G-v}. Many global analytic properties of the nonlinear isomonodromy equation can be obtained through studying the Stokes matrices. Our second result is to derive an explicit formula of the monodromy data via the boundary values of $(A(\mathbf{u}, \mathbf{v}),G(\mathbf{u}, \mathbf{v}))$.

\begin{theorem}
\label{thm: introcatformula}
Let $A(\widehat{A}_0,G_0),G(\widehat{A}_0,G_0)$ be the solution of the isomonodromy equations \eqref{iso for z begin}–\eqref{iso for z end} with the boundary value $(\widehat{A}_0,G_0)\in T^*{\rm GL}_n$ as in Theorem \ref{mainthm}.  Select $U,V$ and the direction $d$ such that
\begin{align}\label{region1}
    &\mathrm{Im}(u_1e^{\mathrm{i}d})>...> \mathrm{Im}(u_ne^{\mathrm{i}d}),\quad \mathrm{Im}(v_1e^{\mathrm{i}d})<...< \mathrm{Im}(v_ne^{\mathrm{i}d}),\\ \label{region2}
    &-\pi<\left(d+\mathrm{arg}(u_{k+1}-u_{k})\right)<0,\quad 0<\left(d+\mathrm{arg}(v_{k+1}-v_{k})\right)<\pi,\ \text{for}\ \ k=1,..,n-1.
\end{align}
Then

$(1)$. The sub-diagonals of the Stokes matrices $S_{d,\pm}^{(\infty)}(U,V,A(\widehat{A}_0,G_0),G(\widehat{A}_0,G_0))$ 
are given by
\begin{align*}
(S_{d,+}^{(\infty)}(U,V,A,G))_{k,k+1} &=
-2\pi\mathrm{i}\\
&\quad \times 
\sum_{i=1}^k\frac{\prod_{l=1,l\ne i}^{k}
\Gamma(1+\widehat{\lambda}^{(k)}_i-\widehat{\lambda}^{(k)}_l)}{\prod_{l=1}^{k+1}
\Gamma(1+\widehat{\lambda}^{(k)}_i-\widehat{\lambda}^{(k+1)}_l)}\frac{\prod_{l=1,l\ne i}^{k}
\Gamma(\widehat{\lambda}^{(k)}_i-\widehat{\lambda}^{(k)}_l)}{\prod_{l=1}^{k-1}
\Gamma(1+\widehat{\lambda}^{(k)}_i-\widehat{\lambda}^{(k-1)}_l)}\cdot 
\det
(\widehat{\lambda}^{(k)}_i\mathrm{Id}-{\widehat{A}_0})
^{1,\ldots,k-1,k}_{1,\ldots,k-1,k+1},\\
(S_{d,-}^{(\infty)}(U,V,A,G))_{k+1,k}& = 
-2\pi\mathrm{i}\cdot 
{\rm e}^{\pi\mathrm{i}
\left((\widehat{A}_0)_{k+1, k+1}-(\widehat{A}_0)_{k, k}\right)}\\
&\quad \times 
\sum_{i=1}^k \frac{\prod_{l=1,l\ne i}^{k}\Gamma(1+\widehat{\lambda}^{(k)}_l-\widehat{\lambda}^{(k)}_i)}{\prod_{l=1}^{k+1}\Gamma(1+\widehat{\lambda}^{(k+1)}_l-\widehat{\lambda}^{(k)}_i)}\frac{\prod_{l=1,l\ne i}^{k}\Gamma(\widehat{\lambda}^{(k)}_l-\widehat{\lambda}^{(k)}_i)}{\prod_{l=1}^{k-1}\Gamma(1+\widehat{\lambda}^{(k-1)}_l-\widehat{\lambda}^{(k)}_i)}\cdot 
{\det
(\widehat{A}_0-{\widehat{\lambda}^{(k)}_i}\mathrm{Id})
^{1,\ldots,k-1,k+1}_{1,\ldots,k-1,k}}.
\end{align*}
where $k=1,\ldots,n-1$, $\mathrm{Id}$ is the $n\times n$ identity matrix and 
$M
^{a_1,\ldots,a_k}_{b_1,\ldots,b_k}$ is 
the $k\times k$ submatrix of $M$ formed 
by the $(a_1,\ldots,a_k)$ rows and 
$(b_1,\ldots,b_k)$ columns. 
Furthermore, the other entries are also given by explicit expressions.

$(2)$. The sub-diagonals of the Stokes matrices $S_{d,\pm}^{(0)}(U,V,A(\widehat{A}_0,G_0),G(\widehat{A}_0,G_0))$ 
are given by
\begin{align*}
S_{d,+}^{(0)}(U,V,A,G)_{k,k+1} &=
-2\pi\mathrm{i}\cdot {\rm e}^{-2\pi\mathrm{i}
{(\widetilde{A}_0)_{k, k}}} \\
&\quad \times 
\sum_{i=1}^k\frac{\prod_{l=1,l\ne i}^{k}
\Gamma(1+\widetilde{\lambda}^{(k)}_i-\widetilde{\lambda}^{(k)}_l)}{\prod_{l=1}^{k+1}
\Gamma(1+\widetilde{\lambda}^{(k)}_i-\widetilde{\lambda}^{(k+1)}_l)}\frac{\prod_{l=1,l\ne i}^{k}
\Gamma(\widetilde{\lambda}^{(k)}_i-\widetilde{\lambda}^{(k)}_l)}{\prod_{l=1}^{k-1}
\Gamma(1+\widetilde{\lambda}^{(k)}_i-\widetilde{\lambda}^{(k-1)}_l)}\cdot 
\det
(\widetilde{\lambda}^{(k)}_i\mathrm{Id}-{\widetilde{A}_0})
^{1,\ldots,k-1,k}_{1,\ldots,k-1,k+1},\\
S_{d,-}^{(\infty)}(U,V,A,G)_{k+1,k}& = 
-2\pi\mathrm{i}\\
&\quad \times 
\sum_{i=1}^k \frac{\prod_{l=1,l\ne i}^{k}\Gamma(1+\widetilde{\lambda}^{(k)}_l-\widetilde{\lambda}^{(k)}_i)}{\prod_{l=1}^{k+1}\Gamma(1+\widetilde{\lambda}^{(k+1)}_l-\widetilde{\lambda}^{(k)}_i)}\frac{\prod_{l=1,l\ne i}^{k}\Gamma(\widetilde{\lambda}^{(k)}_l-\widetilde{\lambda}^{(k)}_i)}{\prod_{l=1}^{k-1}\Gamma(1+\widetilde{\lambda}^{(k-1)}_l-\widetilde{\lambda}^{(k)}_i)}\cdot         
{\det
(\widetilde{A}_0-{\widetilde{\lambda}^{(k)}_i}\mathrm{Id})
^{1,\ldots,k-1,k+1}_{1,\ldots,k-1,k}}.
\end{align*}
where $k=1,\ldots,n-1$, $\widetilde{A}_0=-G_0^{-1}\widehat{A}_0G_0$, and the other entries are also given by explicit expressions.

$(3).$ The connection matrix is given by
\begin{align*}
C_d\left(U,V,A\left(\widehat{A}_0,G_0\right),G\left(\widehat{A}_0,G_0\right)\right)=e^{-\frac{\pi\mathrm{i}}{2}\delta\hat{A}_0}\cdot\left( \overrightarrow{\prod_{k = 1,...,n-1}} \widehat{C}_k \right)\cdot P^{-1}(\widehat{A}_0) G_0P(\widetilde{A}_0)\cdot \left( \overrightarrow{\prod_{k = 1,...,n-1}} \widetilde{C}_k \right)^{-1}\cdot e^{-\frac{\pi\mathrm{i}}{2}\delta\widetilde{A}_0},
    \end{align*} 
where the product $\overrightarrow{\prod}$ is taken with the index $i$ to the right of $j$ if $i>j$, and the entries of the $n\times n$ matrix $\widehat{C}_k$ are 
\begin{align}\nonumber \widehat{C}_{k,ij}=&\frac{ e^{\frac{\pi\I (\widehat{\lambda}^{(k+1)}_j-\widehat{\lambda}^{(k)}_i)}{2}}}{(\widehat{\lambda}^{(k+1)}_j-\widehat{\lambda}^{(k)}_i)}\frac{\prod_{v=1}^{k+1}\Gamma(1+{\widehat{\lambda}^{(k+1)}_j-\widehat{\lambda}^{(k+1)}_v})\prod_{v=1}^{k}\Gamma(1+{\widehat{\lambda}^{(k)}_i-\widehat{\lambda}^{(k)}_v})}{\prod_{v=1,v\ne i}^{k}\Gamma(1+{\widehat{\lambda}^{(k+1)}_j-\widehat{\lambda}^{(k)}_v})\prod_{v=1,v\ne j}^{k+1}\Gamma(1+{\widehat{\lambda}^{(k)}_i-\widehat{\lambda}^{(k+1)}_v})}\\ &\cdot \frac{(-1)^{k+i}\mathrm{det}^{1,...,k}_{1,...,k-1,k+1}\left(\widehat{\lambda}^{(k)}_i-\widehat{A}_0\right)}{\sqrt{\prod_{l=1,l\ne i}^k(\widehat{\lambda}^{(k)}_i-\widehat{\lambda}^{(k)}_l)\prod_{l=1}^{k-1}(\widehat{\lambda}^{(k)}_i-\widehat{\lambda}^{(k-1)}_l)}}\sqrt{\frac{\prod_{v=1}^{k}(\widehat{\lambda}^{(k+1)}_j-\widehat{\lambda}^{(k)}_v)}{\prod_{v=1, v\ne j}^{k+1}(\widehat{\lambda}^{(k+1)}_j-\widehat{\lambda}^{(k+1)}_v) }},
\end{align}
for $1\le j\le k+1, 1\le i\le k$, and
\begin{align*} (\widehat{C}_{k})_{ k+1,j}&=\frac{e^{\frac{\pi\I(\widehat{A}_{k+1,k+1}-\widehat{\lambda}^{(k+1)}_j)}{2}}\prod_{v=1}^{k+1}\Gamma(1+{\widehat{\lambda}^{(k+1)}_j-\widehat{\lambda}^{(k+1)}_v})}{\prod_{v=1}^{k}\Gamma(1+{\widehat{\lambda}^{(k+1)}_j-\widehat{\lambda}^{(k)}_v})}\sqrt{\frac{\prod_{v=1}^{k}(\widehat{\lambda}^{(k+1)}_j-\widehat{\lambda}^{(k)}_v)}{\prod_{v=1, v\ne j}^{k+1}(\widehat{\lambda}^{(k+1)}_j-\widehat{\lambda}^{(k+1)}_v) }}, \ \ \ \text{for} \ 1\le j\le k+1,\\
\widehat{C}_{k,ii}&=1, \ \ \ \ \ \ \text{for} \ k+1<i\le n,\\
\widehat{C}_{k,ij}&=0, \ \ \ \ \ \ \text{otherwise}.
\end{align*}
The $n\times n$ matrix $\widetilde{C}$ is given by the same formula as $\widehat{C}$ provided replacing  $\widehat{A}_0$ and $\widehat{\lambda}^{(k)}_i$'s by $\widetilde{A}_0$ and $\widetilde{\lambda}^{(k)}_i$'s respectively.
Here for any matrix $M$, $P(M)$ is the matrix with entries
\begin{align}
     {P}(M)_{ij} & =
        (-1)^{i+n}
        \frac
        {\det (M-\lambda^{(n)}_j I)^{1,\ldots,n\backslash n}_{1,\ldots,n\backslash i}}
        {\prod_{t = 1}^{n - 1} (\lambda^{(n - 1)}_t - \lambda^{(n)}_j)},\\
        ({P}(M)^{-1})_{ij} & =
        (-1)^{n+j}
        \frac
        {\det (M-\lambda^{(n)}_i I)^{1,\ldots,n\backslash j}_{1,\ldots,n\backslash n}}
        {\prod_{s\neq i} (\lambda^{(n)}_s - \lambda^{(n)}_i)},
\end{align}
that diagonalizes $M$, i.e., $P^{-1}(M)\cdot M\cdot P(M)=\mathrm{diag}(\lambda^{(n)}_1,...,\lambda^{(n)}_n)$.
\end{theorem}
\begin{remark}
    When we assume that the residue matrix $A$ has distinct eigenvalues, the above formula is well-defined. However, since both the Stokes matrices and the connection matrix are analytic with respect to the residue matrix $A$ and the matrix $G$, the formula remains valid via analytic continuation, even for cases where $A$ is non-diagonalizable or possesses repeated eigenvalues.
\end{remark}
\begin{remark}
    When the diagonal elements $u_i,v_i$
  do not satisfy the assumptions \eqref{region1}-\eqref{region2}, the corresponding Stokes matrices and connection matrices can still be computed explicitly, as long as the action of the braid group on the Stokes matrices (see \cite{Dubrovin1996}
 or \cite{Boalch2002}) induced by the variation of $u_i,v_i$ are taken into account.
 \end{remark}

 The expressions for these monodromy data are calculated at the iterated limit point, relying on the invariance  of the monodromy data under isomonodromic deformation. The achievement of the present paper is to express the monodromy data of the original equation in terms of the monodromy data of the limiting system derived from the first limit, $t \rightarrow 0$ (see Theorem \ref{prop:decomp of monodromy}). This requires a factorization property of the solutioins of the system \eqref{introeq}  as $t \rightarrow 0$ (see Theorem \ref{thm:decom of two irr F}). The final step involves determining the specific expressions for the monodromy data under the successive limits of $z_{n-1},...,z_2 \rightarrow \infty$ and $w_{n-1},...,w_2 \rightarrow \infty$, results which have been established in previous work \cite{xu2019closure1,TangXu}. The proof of Theorem \ref{thm: introcatformula} is presented immediately following Corollary \ref{cor: concrete monodromy data by A0 and G0}.

The explicit expression, of the monodromy data of linear systems with one second order pole and one simple pole, given in \cite{xu2019closure1} is closely related to the Gelfand-Tsetlin system on the dual of Lie algebra \cite{guillemin1983gelfand}, and 
was used to study the WKB approximation of the Stokes matrices \cite{alekseev2024wkb}, to give a transcendental realization of $\frak{gl}_n$-crystal structures \cite{xu2019closure1} and certain combinatorial operators of
Young tableaux \cite{Xu2025} arising from representation theory. 
The expression, of the monodromy data of linear systems with two 2nd order poles, given in Theorem \ref{thm: introcatformula} can be seen as a double or fusion construction. It is closely related to the Gelfand-Tsetlin type system on the cotangent bundle of the Lie group, and is expected to be related to the tensor product of crystals, and to give a transcendental interpretation of Littlewood-Richardson rule. Further exploration along this direction is planned for future work.

\subsection{Connection to the $tt^{*}$ equations and examples of solutions not in $\mathfrak{Sol}_{Shr}$}
In equation \eqref{introeq}, when we set $V=\overline{U}$ and let $A=-[U,q],\ G=m$, where $q(\mathbf{u},\bar{\mathbf{u}})$ is a symmetric off-diagonal matrix function, $m(\mathbf{u},\bar{\mathbf{u}})$ is an orthogonal, Hermitian matrix function, such as $ [U,q]=m[\overline{U},\overline{q}]\overline{m}$,
we obtain the following equation:
\begin{align}
     \frac{\partial}{\partial \lambda}\varphi&= (U-\lambda^{-1}[U,q]-\lambda^{-2}m\overline{U}m^{-1})\varphi.
\end{align}
The corresponding isomonodromy equations are the $tt^{*}$ equations (with similarity reduction) established in \cite{dubrovin1993tteq}. 

In the case $n=2$, the $tt^{*}$
  equation is equivalent to the sine-Gordon Painlevé III equation. Another important special case is the $tt^{*}$-Toda equation, where the variables $u_1,...,u_n$ are prescribed constant multiples of a single variable $x$.

In section \ref{sec:apply for sg piii}, \ref{sec:apply for tt toda}, we specialize our Theorems \ref{mainthm} and \ref{thm: introcatformula} to these two special cases, and show that our results are consistent with the known results of Painlevé III equation from \cite{FIKN2006} and the $tt^{*}$-Toda equations results from \cite{guest2015isomonodromy1,guest2015isomonodromy2,guest2023toda}. For the $tt^{*}$-Toda equations, a family of solutions was parameterized by a polytope in the space of certain asymptotic parameters. We further describe which part of the polytope corresponds to the solutions in the set $\mathfrak{Sol}_{Shr}$. By doing this, we also identity a family of known solutions of the isomonodromy equation that are not in $\mathfrak{Sol}_{Shr}$.

For the general $tt^{*}$-Toda equations, Theorems \ref{mainthm} allows us to derive the local behaviors of $q$ and $m$ as $t\rightarrow 0$, see Corollary \ref{prop:asym of gener tt}.

\subsection*{Acknowledgements}
\noindent
The authors would like to thank Martin Guest, Alexander Its, Qian Tang and Yuancheng Xie for useful discussion and comments. The authors are supported by the National Key Research and Development Program of China (No. 2021YFA1002000).

\section{Preliminaries on the meromorphic linear systems with one second order pole}
In this section, we recall the main results in \cite{TangXu, xu2019closure1}.
Let us consider the $n\times n$ linear system of meromorphic differential equation for a function 
$F(z,u_1,\ldots,u_n)\in {\rm GL}_n(\mathbb{C})$ 
\begin{subequations}
\begin{align}\label{introisoStokeseq1}
\frac{\partial F}{\partial z}&=\left( U +
\frac{\Phi(\mathbf{u})}{z}\right)\cdot F,\\
\label{introisoStokeseq2}
\frac{\partial F}{\partial u_k}&=\left(E_kz+{\rm ad}^{-1}_U{\rm ad}_{E_{k}}\Phi(\mathbf{u})\right)\cdot F, \ \text{for all} \ k=1,\ldots,n.
\end{align}
\end{subequations}
where $U=\mathrm{diag}\{u_1,u_2,...,u_n\},$ and the residue matrix $\Phi(\mathbf{u})=\Phi(u_1,...,u_n)$ is a solution of the equation
\begin{align}\label{eq:iso eq of one irr in u}
    \frac{\partial \Phi}{\partial u_k}=[\mathrm{ad}_U^{-1} \mathrm{ad}_{E_k}\Phi,\Phi],\ k=1,...,n.
\end{align}
One checks that \eqref{eq:iso eq of one irr in u} is the compatibility condition of the linear system. 

\subsection{Long time behaviour and parameterization of the generic solutions of the equation \eqref{eq:iso eq of one irr in u}}

Following Miwa \cite{Miwa1981}, 
the solutions $\Phi(\mathbf{u})$ of \eqref{eq:iso eq of one irr in u}
are multi-valued meromorphic functions 
of $u_1,\ldots,u_n$, where
the branching occurs 
when $(u_1,...,u_n)$ moves along a loop around 
the fat diagonal $\Delta$. 
The asymptotic behaviour at a critial point (long time hehaviour)\ and boundary condition of solutions of the equation \eqref{eq:iso eq of one irr in z} were given in \cite{TangXu}.

To state the result, { we introduce the coordinates $z_0=u_1,z_1=u_2-u_1,z_2=\frac{u_3-u_1}{u_2-u_1},...,z_{n-1}=\frac{u_n-u_1}{u_{n-1}-u_1}$, then the equations \eqref{eq:iso eq of one irr in u} become}
\begin{align}\label{eq:iso eq of one irr in z}
    \frac{\partial \Phi(\mathbf{z})}{\partial z_k}=[\mathrm{ad}_U^{-1} \mathrm{ad}_{\frac{\partial U}{\partial z_k}}\Phi,\Phi],\ k=0,...,n-1,
\end{align}
where \[\frac{\partial U}{\partial z_k}= \frac{1}{z_k}(u_{k+1}E_{k+1}+...+u_nE_n-z_0(E_{k+1}+...+E_n)).\]

\begin{theorem}[\cite{TangXu}]\label{thm:iso asy by tangxu}
 Given any constant $n\times n$ matrix $\Phi_0$ 
that satisfies the {boundary condition} for the system \eqref{eq:iso eq of one irr in z}
\begin{equation}
\label{eq: shring condition of one irr iso}
\left| \mathrm{Re} \left(
\lambda^{(k-1)}_{i} - 
\lambda^{(k-1)}_{j}
\right)\right | < 1,
\quad
\text{for every $1\leqslant i,j\leqslant k-1$},
\end{equation}
where $\{\lambda^{(k-1)}_{i}\}_{i=1,\ldots,k-1}$
are the eigenvalues of the
upper left $(k-1)\times (k-1)$ submatrix
of $\Phi_{0}$, there exists a unique solution 
$\Phi(\mathbf{z};\Phi_0)=\Phi_{n-1}(\mathbf{z})$ to the equation \eqref{eq:iso eq of one irr in z},
and a sequence of matrix functions $\Phi_{n-1},\ldots,\Phi_{1}$
such that 
\begin{align} 
\label{eq:limit in zk one irr iso}
\underset{z_k \rightarrow \infty}{\lim} 
\delta_{k } \Phi_k  = 
\delta_{k } \Phi_{k - 1},\quad
\underset{z_k \rightarrow \infty}{\lim} 
z_k^{\mathrm{ad} 
\delta_{k} \Phi_{k - 1}} \Phi_k  =  \Phi_{k - 1}\quad \text{for } k=2,\dots n-1; \ \text{ and } \ 
z_1^{\mathrm{ad} 
\delta \Phi_{ 1}} \Phi_1  =  \Phi_{0}.
\end{align}
The regularized limit $\Phi_0$ is called the \textbf{boundary value} of $\Phi(\mathbf{z};\Phi_0)$ (at the iterated limit point $z_{n-1},...,z_2\rightarrow \infty$).
Furthermore, the set of solutions $\Phi(\mathbf{z};\Phi_0)$, parameterized by $\Phi_0$ satisfying the boundary condition \eqref{eq: shring condition of one irr iso}, includes almost all the solutions.
 \end{theorem}
\begin{remark}
  The original statement of this theorem in \cite{TangXu} differs slightly from the formulation presented here. This difference arises from the different definition of the $z$-coordinate (where $z_k=\frac{u_k-u_{k-1}}{u_{k-1}-u_{k-2}}$ is defined in \cite{TangXu}). In \cite{TangXu}, the limiting behavior of $\Phi(\mathbf{u})$ is described as: 
  $$
  \label{limit1}
\underset{u_k \rightarrow \infty}{\lim} 
\delta_{k - 1} \Phi_k  = 
\delta_{k - 1} \Phi_{k - 1},\quad
\underset{u_k \rightarrow \infty}{\lim} 
\left(
\frac{u_k-u_{k-1}}{u_{k-1}-u_{k-2}}
\right)^{\mathrm{ad} 
\delta_{k - 1} \Phi_{k - 1}} \Phi_k  =  \Phi_{k - 1},\quad 1\leq k\leq n.
  $$
  Nevertheless, the present formulation is equivalent to the original version of the theorem, and the $\Phi_0$ obtained here is identical to the $\Phi_0$ in the original work.
\end{remark}

\subsection{Monodromy matrices of the linear ordinary differential equation}\label{sec2.2}
For fixed $u_1,...,u_n$, the linear system of partial differential equation becomes an ordinary differential equation \eqref{introisoStokeseq1}. 

\begin{definition}
 The \textbf{anti-Stokes directions} of the linear ODE \eqref{introisoStokeseq1} at $\bf{\xi=\infty}$ are 
    \begin{align*}
        aS(u) &:=\{-\mathrm{arg}(u_i-u_j)+2k\pi: k\in \mathbb{Z}, i\neq j\}.
    \end{align*}
Let us choose an initial anti-Stokes direction $\tau_0\in (\pi,\pi)$ and then
arrange the anti-Stokes directions  into a strictly monotonically increasing sequence 
\begin{align*}
        \cdots <\tau_{-1}<\tau_0<\tau_1<\cdots.
    \end{align*}
    For any direction $d\in (\tau_j,\tau_{j+1})$, the \textbf{Stokes sector} $\bf{Sect_d}$ is 
    \begin{align}\label{eq:stoke sector for single}
        Sect_{d}^{(\infty)}:=\left\{\phi\in\bar{\mathbb{C}} : \mathrm{arg}\phi\in\bigg(\tau_j-\frac{\pi}{2},\tau_{j+1}+\frac{\pi}{2}\bigg)\right\}
    \end{align} 
\end{definition}
Then the standard theory of resummation states that on each of these sectors there is a unique (therefore canonical) holomorphic solution with the prescribed asymptotics. See e.g., \cite{Balser}.
\begin{theorem}
     On each $\mathrm{Sect}^{(\infty)}_d$, there is a unique holomorphic fundamental solution $F_d(\xi)$ of \eqref{introisoStokeseq1} with the prescribed asymptotics:
    \begin{align}
    F_{d}(\xi)\cdot \mathrm{e}^{-\xi U}\xi^{-\delta \Phi} \sim \mathrm{Id}+O(\xi^{-1}),\  \text{ as } \xi\rightarrow  \infty \text{ within } \mathrm{Sect}^{(\infty)}_{d}.
\end{align}
\end{theorem}

    \begin{definition}
The \textbf{Stokes matrix} with respect to direction $d$ is determined by 
   \begin{align}
        S_{d}^{\pm}(U,\Phi):=F_{d\pm\pi}(\xi)^{-1}F_d(\xi),
    \end{align}
    \end{definition}

The Stokes matrices are understood as internal monodromy data of the linear system \eqref{introisoStokeseq1} at $\xi=\infty$. Now let us introduce the connection matrices from $\xi=0$ to $\xi=\infty$. First,
\begin{lemma}
    Under the assumption that the eigenvalues of matrix $\Phi$ do not differ by a positive integer,  the linear ODE \eqref{introisoStokeseq1} has a fundamental solution $F^{(0)}(\xi)$ with 
    \begin{align}
        F^{(0)}(\xi)\cdot \xi^{-\Phi}\sim \mathrm{Id}+O(\xi),\ \text{ as } \xi\rightarrow 0.
    \end{align}
\end{lemma}

\begin{definition} \label{def:conn for 1irr}
If $\Phi$ is non-resonant, i.e. the difference of any two eigenvalues of $\Phi$ is not non-zero integers, we define the \textbf{connection matrix} $C_d(U,\Phi)$ of \eqref{introisoStokeseq1}, with respect to the direction $d$, as 
    \begin{align}\label{def: defi of connec mat for 1 pole}
        C_d(U,\Phi)=F_d(\xi)^{-1}F^{(0)}(\xi).
    \end{align}
And we define the \textbf{monodromy matrix} $\nu_d(U,\Phi)$ of $F_d(\xi)$ as
\begin{align}
    \nu_d(U,\Phi):=F_d(\xi)^{-1}F_d(\xi \mathrm{e}^{2\pi\mathrm{i}}).
\end{align}   
        \end{definition}
\begin{remark}
   By the definition of $\mathrm{Sect}_d$ in \eqref{eq:stoke sector for single}, we have $F_{d_1}=F_{d_2}$ for $\tau_j<d_1<d_2<\tau_{j+1}$. It follows that the Stokes matrices, connection matrix and monodromy matrices are all invariant under perturbation of the direction $d$ that remains within the sector $(\tau_j,\tau_{j+1})$.
\end{remark}
The following proposition is standard. For example, the identity \eqref{monorel3} follows from the fact that a simple positive loop (i.e., in clockwise direction) around $0$ is a simple negative loop (i.e., in anticlockwise direction) around $\infty$.
\begin{lemma}\label{prop:monodromy factor for one irr sys}
We have the following monodromy relations
    \begin{align}\label{eq:trans in stokes mat with c}
       S_d^{\pm}(cU,\Phi)&=c^{\delta \Phi}S_{d+\mathrm{argc}}^{\pm}(U,\Phi)c^{-\delta \Phi},\\
         S_{d\mp\pi}^{\pm}(U,\Phi)&=S_{d}^{\mp}(U,\Phi)^{-1},\\ \label{eq:trans in connection mat with c}
        C_d(cU,\Phi)&=c^{\delta\Phi}C_{d+\mathrm{arg}c}(U,\Phi)c^{-\Phi},\\ \label{monorel3}
        \nu_d(U,\Phi)&=S_d^{-1}(U,\Phi)^{-1}\cdot \mathrm{e}^{2\pi\mathrm{i}\delta\Phi}\cdot S_d^{+}(U,\Phi)
        =C_d(U,\Phi)\cdot\mathrm{e^{2\pi\mathrm{i}\Phi}}\cdot C_d(U,\Phi)^{-1}.
    \end{align}
\end{lemma}

For given $d\notin{aS}(u)$, let us introduce the $n\times n$ permutation matrix $P$ with entries $P_{ij} = \delta_{\sigma(i)j}$, where $\sigma$ is the permutation of $\{1, . . . , n\}$ corresponding to the
dominance ordering of $\{e^{u_1z},...,e^{u_nz}\}$ along the direction $d+\frac{\pi}{2}$. That is
$\sigma(i) < \sigma(j)$ if and only if $e^{(u_i-u_j)z}\rightarrow 0$ as $z\rightarrow \infty$ along $d+\frac{\pi}{2}$. Then the Stokes matrices are triangular matrices up to the conjugation by the permutation matrix. That is
\begin{lemma}\label{lem:upper}
The matrices $P\cdot S_d^+(U,\Phi)\cdot P^{-1}$ and $P\cdot S_d^-(U,\Phi)\cdot P^{-1}$ are upper and lower triangular matrices respectively. Furthermore, they have $1$'s along the diagonal.
\end{lemma}

\subsection{The isomonodromy property}
Recall that the discussion in Section \ref{sec2.2} is for a fixed $U={\rm diag}(u_1,...,u_n)$. It follows from the work of Jimbo-Miwa-Ueno \cite{JMU1981I} that
\begin{prop}\label{prop:iso}
If $\Phi(\mathbf{u})$ satisfies the isomonodromy equations \eqref{eq:iso eq of one irr in u}, then the Stokes matrices $S_{d}^{\pm}(U,\Phi)$ are locally constant. In particular, for any direction $d$,  let 
\begin{equation}
    R_{u,d}:=\{U\in\mathbb{C}^n\setminus\Delta ~|~d\neq -\mathrm{arg}(u_i-u_j)+2k\pi, \text{ for all } i\neq j, k\in \mathbb{Z}\}
\end{equation}
be the subset of $\mathbb{C}^n\setminus\Delta$ consisting of all $u$ such that $d\notin aS(u)$, then the Stokes matrices $S_{d}^{\pm}(U,\Phi)$ are constant for all $u$ in each connected component of $R_{u,d}$.
\end{prop}

It should be noted that the  connection matrix, as defined in Definition \ref{def:conn for 1irr}, is not locally invariant. Specifically, its evolution is described by the following proposition:
\begin{prop}[\cite{JMU1981I}]
    Given a non-resonant solution $\Phi(z_0,...,z_{n-1})$ of the isomonodromy equations \eqref{eq:iso eq of one irr in z}, the connection matrix $C_d(U,\Phi)$, seen as function in the coordinates $(z_0, z_1, \dots, z_{n-1})$, satisfies the following system of equations for a matrix function $X(\mathbf{z})$:
    \begin{align}\label{eq: eq of one irr connection mat}
        \frac{\partial X(\mathbf{z})}{\partial z_k} = -X(\mathbf{z}) \cdot \mathrm{ad}_{U}^{-1} \mathrm{ad}_{\frac{\partial U}{{\partial z}_k}} \Phi, \quad k=0, \dots, n-1.
    \end{align}
\end{prop}

\subsection{The isomonodromy deformation and monodromy problem}
Given any solution $\Phi(\mathbf{u};\Phi_0)$ with the boundary value $\Phi_0$, by Proposition \ref{prop:iso}, the Stokes matrices  $S_{d}^{\pm}(U,\Phi(u;\Phi_0))$ are constant on $R_{u,d}\in\mathbb{C}^n\setminus\Delta $, thus only depend on the boundary value (integration constant) $\Phi_0$. The following theorem gives explicit expression of the Stokes matrices via $\Phi_0$. 

\begin{theorem}[\cite{TangXu}] \label{thm:monodromy factor of one irr sys}
   Let $R_{u,d}(J)$ be the connected component of $R_{u,d}$  labelled by a subset $J\subset \{1,2,...,n-1\}$ as follows: $U=\mathrm{diag}(u_1,...,u_n)\in R_{u,d}(J)$ if and only if
\begin{align}
    &\mathrm{Im}(u_{k+1}\mathrm{e}^{\mathrm{i}d})<\min_{1\leq j\leq k} \mathrm{Im}(u_j\mathrm{e}^{\mathrm{i}d}),\  \text{for}\ k\in J,\\
    &\mathrm{Im}(u_{k+1}\mathrm{e}^{\mathrm{i}d})>\min_{1\leq j\leq k} \mathrm{Im}(u_j\mathrm{e}^{\mathrm{i}d}),\  \text{for}\ k\notin J.
\end{align}
Then for $U\in R_{u,d}(J)$, we have
    \begin{align}\label{eq:catformula for sing irr}
            S_d^{-1}(U,\Phi)^{-1}\cdot \mathrm{e}^{2\pi\mathrm{i}\delta\Phi}\cdot S_d^{+}(U,\Phi)
=\mathrm{Ad}\left(\overrightarrow{\prod_{k=1}^{n-1}}C_{d+\mathrm{arg}(u_{k+1}-u_k)}(E_{k+1},\delta_{k+1}(\Phi_{0}))\right)\mathrm{e}^{2\pi\mathrm{i}\Phi_0}.
    \end{align}
Here each $C_{d+\mathrm{arg}(u_{k+1}-u_k)}(E_{k+1},\delta_{k+1}(\Phi_{0}))\in {\rm GL}_n$ is the connection matrix of the $n\times n$ linear system of ODEs
\begin{equation}\label{relequation}
\frac{dF}{dz}=\left(E_{k+1}+\frac{\delta_{k+1}(\Phi_0)}{z}\right)\cdot F,
\end{equation}
with respect to the direction $d+\mathrm{arg}(u_{k+1}-u_k)$.
\end{theorem}
\begin{remark}
    When $J=\{1,2,...,n-1\}$, $U\in R_{u,d}(J)$ means that $\mathrm{Im}(u_1\mathrm{e}^{\mathrm{i}d})>\cdots>\mathrm{Im}(u_n\mathrm{e}^{\mathrm{i}d})$, and if we further let $-\pi<d+\mathrm{arg}(u_{k+1}-u_{k})<0$ for $k=1,...,n-1$, we have 
    \begin{align*}
          S_d^{-1}(U,\Phi)^{-1}\cdot \mathrm{e}^{2\pi\mathrm{i}\delta\Phi}\cdot S_d^{+}(U,\Phi)
=\mathrm{Ad}\left(\overrightarrow{\prod_{k=1}^{n-1}}C_{-\frac{\pi}{2}}(E_{k+1},\delta_{k+1}(\Phi_{0}))\right)\mathrm{e}^{2\pi\mathrm{i}\Phi_0}.
    \end{align*}
\end{remark}
The above Theorem states a factorization of ( the monodromy data of) the system \eqref{introisoStokeseq1} into multiple systems \eqref{relequation} for $k=1,...,n-1$. Since the systems \eqref{introisoStokeseq1} for $k=1,...,n-1$ are rigid, one can get the explicit formula of the connection matrices $C_{d+\mathrm{arg}(u_{k+1}-u_k)}(E_{k+1},\delta_{k+1}(\Phi_{0}))$ and therefore the explicit formula of $S_d^{\pm}(U,\Phi(\mathbf{u};\Phi_0)$ in terms of $\Phi_0$. See \cite{TangXu, xu2019closure1} for more details.

As a consequence of the monodromy relation \eqref{monorel3}, and Theorem \ref{thm:monodromy factor of one irr sys}, we also have

\begin{cor}\label{cor:decom of CPhiC in one irr}
    Given $\Phi(\mathbf{u})=\Phi(\mathbf{u};\Phi_0)$, if $\Phi_0$ is non-resonant, then for all $U\in R_{u,d}(J)$  we have
    \begin{align}\label{cor:log}
        C_d(U,\Phi)\cdot\Phi\cdot C_d(U,\Phi)^{-1}=\mathrm{Ad}\left(\overrightarrow{\prod_{k=1}^{n-1}}C_{d+\mathrm{arg}(u_{k+1}-u_{k})}\left(E_{k+1},\delta_{k+1}(\Phi_{0})\right)\right)\Phi_0.
    \end{align}
\end{cor}

\section{System with two irregular singularities, its monodromy data and isomonodromy equation}\label{sec:intro for two irr sys}

In this section, we introduce the system  with two second-order poles, and define its monodromy data and isomonodromic deformation equations.

\subsection{Lax pair of the nonlinear PDEs \eqref{eq:iso eq of two irr A-u}-\eqref{eq:iso eq of two irr G-v}}
Let us consider the $n\times n$ linear system for function $F(\xi,u_1,...,u_n,v_1,...,v_n)\in {\rm GL}_n$ with $2n+1$ complex variables as follows:
\begin{subequations}
\begin{align}
    \label{eq: linr sys two irr in xi}
    \frac{\partial F}{\partial \xi}&=\left(U+\frac{A(\mathbf{u},\mathbf{v})}{\xi}+\frac{G(\mathbf{u},\mathbf{v})VG^{-1}(\mathbf{u},\mathbf{v})}{\xi^2}\right)F,\\ \label{eq:linr sys two irr in u}
     \frac{\partial F}{\partial u_k}&=(E_k\xi+\mathrm{ad}_U^{-1}\mathrm{ad}_{E_k}A(\mathbf{u},\mathbf{v}))F,\\ \label{eq:linr sys sys two irr in v}
    \frac{\partial F}{\partial v_k}&=\left(-\frac{G(\mathbf{u},\mathbf{v})E_kG^{-1}(\mathbf{u},\mathbf{v})}{\xi}\right)F,
    \end{align}
\end{subequations}
where $U=\mathrm{diag}\{u_1,u_2,...,u_n\},V=\mathrm{diag}\{v_1,v_2,...,v_n\}$, and $A(\mathbf{u},\mathbf{v})\in\frak{gl}_n$, $G(\mathbf{u},\mathbf{v})\in{\rm GL}_n$ are matrix valued solutions of the system \eqref{eq:iso eq of two irr A-u}-\eqref{eq:iso eq of two irr G-v} with variables  $(\mathbf{u},\mathbf{v})=(u_1,...,u_n,v_1,...,v_n)$. We assume $u_i-u_j\notin \mathbb{Z}, v_i-v_j\notin \mathbb{Z}$, for $i\neq j$.

\subsection{Monodromy data}
For fixed diagonal matrices $U,V$ under the assumption 
\begin{equation}\label{nrassumption}
    u_i-u_j\notin \mathbb{Z}, \hspace{3mm} v_i-v_j\notin \mathbb{Z}, \text{ for } i\neq j,
\end{equation} 
we consider the ordinary differential equation \eqref{eq: linr sys two irr in xi}.
\begin{definition}
    The \textbf{anti-Stokes directions} of \eqref{eq: linr sys two irr in xi} at $\bf{\xi=\infty}$ are \[
        aS(u):=\{-\mathrm{arg}(u_i-u_j)+2k\pi: k\in \mathbb{Z}, i\neq j\}.\]
If we arrange these directions into a strictly monotonically increasing sequence $
        \cdots <\tau_{-1}<\tau_0<\tau_1<\cdots$,
Then for any $d\in (\tau_j,\tau_{j+1})$, the \textbf{Stokes sector} $\bf{Sect_{d}^{(\infty)}}$ is 
    \begin{align*}
        Sect_{d}^{(\infty)}:=\left\{\phi\in\bar{\mathbb{C}} : \mathrm{arg}\phi\in\left(\tau_j-\frac{\pi}{2},\tau_{j+1}+\frac{\pi}{2}\right)\right\}.
    \end{align*}
\end{definition}    
    Similarly, 
\begin{definition}
The \textbf{anti-Stokes directions} of \eqref{eq: linr sys two irr in xi} at $\bf{\xi=0}$ are 
     \begin{align*}
        aS(v) &:=\{\mathrm{arg}(v_i-v_j)+2k\pi: k\in \mathbb{Z}, i\neq j\}.
    \end{align*}
If we arrange these directions into a strictly monotonically increasing sequence $
        \cdots <\theta_{-1}<\theta_0<\theta_1<\cdots,$
then for any $d\in (\theta_j,\theta_{j+1})$, the \textbf{Stokes sector} $\bf{Sect_{d}^{(0)}}$ is 
    \begin{align*}
        Sect_{d}^{(0)}:=\left\{\phi\in\bar{\mathbb{C}} : \mathrm{arg}\phi\in\left(\theta_j-\frac{\pi}{2},\theta_{j+1}+\frac{\pi}{2}\right)\right\}.
    \end{align*}
\end{definition}

Similar to the one irregular pole case, we have different holomorphic solutions with the prescribed asymptotics at different Stokes sectors.
\begin{prop}[See e.g., \cite{Balser},\cite{Balser-Jurkat-Lutz1981}]\label{prop:fundamental solution of two irr sys}
For fixed $u,v$ and for any $d\notin {\rm aSR}(u)$, there is a unique holomorphic fundamental solution $F^{(\infty)}_d(\xi)$ of the ordinary differential equation \eqref{eq: linr sys two irr in xi} over $\widetilde{\mathbb{C}}$ with the prescribed asymptotics:
    \begin{align}
    F^{(\infty)}_{d}(\xi)\cdot \mathrm{e}^{-\xi U}\xi^{-\delta A}\sim \mathrm{Id}+O(\xi^{-1}),\  \text{ as } \xi\rightarrow\infty \ \text{ within } \mathrm{Sect}^{(\infty)}_{d}.
\end{align}
At the same time, for any $d\notin {\rm aSR}(v)$, there is a unique holomorphic solution $F^{(0)}_d(\xi)$ with the prescribed asymptotics:
\begin{align}
    F^{(0)}_{d}(\xi)\cdot \mathrm{e}^{\frac{V}{\xi}}\xi^{-\delta (G^{-1}AG)}\sim G+O(\xi), \  \text{ as } \xi\rightarrow 0 \ \text{ within } \mathrm{Sect}^{(0)}_{d}.
\end{align}
\end{prop}

In the following, we introduce the monodromy data, that are various transition matrices between the above preferred fundamental solutions.
\begin{definition}\label{def:stokes for 2irr}
    The \textbf{Stokes matrices} of \eqref{eq: linr sys two irr in xi} at $\bf{\xi=\infty}$ (with respective to any given direction $d\notin {\rm aSR}(u)$) are 
    \begin{align}
        S_{d,\pm}^{(\infty)}(U,V,A,G):=F^{(\infty)}_{d\pm\pi}(\xi)^{-1}F^{(\infty)}_d(\xi).
    \end{align}
The \textbf{Stokes matrices} of \eqref{eq: linr sys two irr in xi} at $\bf{\xi=0}$ (with respective to any given direction $d\notin {\rm aSR}(v)$) are 
    \begin{align}
        S_{d,\pm}^{(0)}(U,V,A,G):=F^{(0)}_{-d\mp\pi}(\xi)^{-1}F^{(0)}_{-d}(\xi).
    \end{align}
\end{definition}
\begin{definition}\label{def:conn for 2irr}
For a chosen direction $d\notin {\rm aSR}(u)\cup {\rm aSR}(v)$, the associated \textbf{connection matrix} of the equation \eqref{eq: linr sys two irr in xi} is
    \begin{align}
        C_d(U,V,A,G):=F^{(\infty)}_d(\xi)^{-1}F^{(0)}_{-d}(\xi).
    \end{align}
\end{definition}
\begin{remark}
Through the change of variable $\eta=\frac{1}{\xi}$, the monodromy data at 0, by the definition introduced above, is transformed into the monodromy data at $\infty$ of the new system.
\end{remark}
\begin{definition}\label{def: mono mat}
For any chosen direction $d\notin {\rm aSR}(u)$, the associated \textbf{monodromy matrix} $\nu^{(\infty)}_d(A,G)$  of the equation \eqref{eq: linr sys two irr in xi} at $\xi=\infty$ is
\begin{align}
    \nu_d^{(\infty)}(U,V,A,G):=F^{(\infty)}_d(\xi)^{-1}F^{(\infty)}_d(\xi \mathrm{e}^{2\pi\mathrm{i}}).
\end{align}
Similarly, given any direction $d\notin {\rm aSR}(v)$, the associated \textbf{monodromy matrix} $\nu^{(0)}_d(A,G)$ at $\xi=0$ is
\begin{align}
    \nu_d^{(0)}(U,V,A,G):=F^{(0)}_{d}(\xi)^{-1}F^{(0)}_d(\xi \mathrm{e}^{-2\pi\mathrm{i}}).
\end{align}
\end{definition}
The following triangular property of the Stokes matrices and the monodromy relation are standard. See e.g. \cite{Balser-Jurkat-Lutz1981},\cite{JMU1981I}. Similar to Lemma \ref{lem:upper}, the two pairs of Stokes matrices are upper and lower triangular, up to the conjugation by two permutation matrices respectively. In particular, we have

\begin{lemma}
\label{lem:triang stokes mat} Suppose $d\notin aS(u)\cup aS(v)$.

$(a)$. If 
\begin{equation}\label{uorder}
    \mathrm{Im}(u_1e^{\mathrm{i}d})>...> \mathrm{Im}(u_ne^{\mathrm{i}d}),
\end{equation}
then the Stokes matrices $S_{d,+}^{(\infty)}(U,V,A,G)$ and $S_{d,-}^{(\infty)}(U,V,A,G)$ are upper and lower triangular matrices respectively, with $1$'s along the diagonal. 

$(b)$. If 
\begin{equation}\label{vorder}
    \mathrm{Im}(v_1e^{\mathrm{i}d})<...< \mathrm{Im}(v_ne^{\mathrm{i}d}),
\end{equation}
then the Stokes matrices $S_{d,+}^{(0)}(U,V,A,G)$ and $S_{d,-}^{(0)}(U,V,A,G)$ are upper and lower triangular matrices respectively, with $1$'s along the diagonal.
\end{lemma}

\begin{lemma}[\label{prop:identity of monodromy factor for two irr sys}\cite{JMU1981I}]
Suppose $d\notin aS(u)\cup aS(v)$, we have the following monodromy relations between monodromy matrices, Stokes matrices and connection matrix
\begin{align}\label{eq:LU fac of nu infty}
      \nu^{(\infty)}_d(U,V,A,G)&=S_{d,-}^{(\infty)}(U,V,A,G)^{-1}\cdot  \mathrm{e}^{2\pi\mathrm{i}\delta A}\cdot S_{d,+}^{(\infty)}(U,V,A,G),\\ \label{eq:LU fac of nu 0}
         \nu^{(0)}_d(U,V,A,G)&=S_{-d,-}^{(0)}(U,V,A,G)^{-1}\cdot \mathrm{e}^{-2\pi\mathrm{i}\delta (G^{-1}AG)}\cdot S_{-d,+}^{(0)}(U,V,A,G),\\ \label{eq:connection relation for 0 and infty}
         \nu^{(\infty)}_d(U,V,A,G)&=C_d(U,V,A,G)\cdot\nu^{(0)}_{-d}(U,V,A,G)^{-1}\cdot C_d(U,V,A,G)^{-1}.
  \end{align}
\end{lemma}
Note that by Lemma \ref{lem:triang stokes mat}, under the assumptions \eqref{uorder} and \eqref{vorder}, the identities \eqref{eq:LU fac of nu infty} and \eqref{eq:LU fac of nu 0} give the LU decomposition of the matrices $  \nu^{(\infty)}_d(U,V,A,G)$ and $  \nu^{(0)}_d(U,V,A,G)$ respectively.

\subsection{The isomonodromy property}\label{sec:iso property for 2 irr sys}
\begin{prop}[\cite{JMU1981I}]\label{lem:monodromy indep by JMU}
    When $A(\mathbf{u},\mathbf{v}),G(\mathbf{u},\mathbf{v})$ satisfy the isomonodromy equations \eqref{eq:iso eq of two irr A-u}-\eqref{eq:iso eq of two irr G-v}, the Stokes matrices  $S_{d,\pm}^{(\infty)}(U,V,A,G)$, $S_{d,\pm}^{(0)}(U,V,A,G)$ and the connection matrix  $ C_d(U,V,A,G)$ of \eqref{eq: linr sys two irr in xi} are locally constant (independent of $u$ and $v$).
\end{prop}
    In particular, for any direction $d$, let 
\begin{align}
    R_{u,d}&:=\{U\in\mathbb{C}^n\setminus\Delta ~|~d\neq -\mathrm{arg}(u_i-u_j)+2k\pi, \text{ for all } i\neq j, k\in \mathbb{Z}\},\\
    R_{v,d}&:=\{V\in\mathbb{C}^n\setminus\Delta ~|~d\neq \mathrm{arg}(v_i-v_j)+2k\pi, \text{ for all } i\neq j, k\in \mathbb{Z}\}.
\end{align}
Then the Stokes matrices $S_{d,\pm}^{(\infty)}(U,V,A,G)$ are constant for all $u$ in each connected component of $R_{u,d}$, $S_{d,\pm}^{(0)}(U,V,A,G)$ are constant for all $V$ in each connected component of $R_{v,d}$, the connection matrix $C_d(U,V,A,G)$ is  constant over the intersection of the above two connected regions.

\section{Construction of solutions of the isomonodromy equations with the prescribed asymptotics as $t\rightarrow 0$}

Recall that we have introduced a new coordinate system \eqref{coor1}-\eqref{coor3}.
That is,
\begin{align*}
    &u_1=z_0,u_2=z_1+z_0,...,u_n=z_1z_2\cdots z_{n-1}+z_0,\\
    &v_n=\frac{t}{z_1z_2\cdots z_{n-1}}+w_0,...,v_2=\frac{t}{z_1z_2\cdots z_{n-1}w_{n-1}\cdots w_2}+w_0,\ v_1=w_0.
\end{align*}
For convenience, we also denote $w_1:=v_2-v_1=\frac{t}{z_1\cdots z_{n-1}w_{n-1}\cdots w_2}$.

In this coordinate system, the isomonodromy equations become the nonlinear PDEs for matrix valued functions $A(\mathbf{z},t,\mathbf{w})\in\frak{gl}_n$ and $G(\mathbf{z},t,\mathbf{w})\in {\rm GL}_n$, with variables $\mathbf{z} = (z_0,z_1,\ldots,z_{n-1})$ and $\mathbf{w} = (w_{n-1},\ldots,w_2,w_0)$:
\begin{subequations}
\begin{align}\label{iso for z begin}
    \frac{\partial A}{\partial z_0} &= 0,\ \frac{\partial A}{\partial z_k} = \bigl[\mathrm{ad}_U^{-1}\mathrm{ad}_{\frac{\partial U}{\partial z_k}}A, A\bigr] + \biggl[\frac{\partial U}{\partial z_k}, GVG^{-1}\biggr] - \frac{1}{z_k}[U, GVG^{-1}],\ k=1,\ldots,n-1, \\ \label{A for z_n+1}
    \frac{\partial A}{\partial t} &= \biggl[U, \frac{1}{t}GVG^{-1}\biggr] = \biggl[U, \frac{1}{t}(GVG^{-1}-w_0\mathrm{Id})\biggr], \\
    \frac{\partial A}{\partial w_{0}} &= 0,\ \frac{\partial A}{\partial w_{k}} = \biggl[U, G\frac{\partial V}{\partial w_{k}}G^{-1}\biggr],\ k=2,\ldots,n-1.\\
       \frac{\partial G}{\partial z_0} & = 0,\ \frac{\partial G}{\partial z_k} = \bigl(\mathrm{ad}_U^{-1}\mathrm{ad}_{\frac{\partial U}{\partial z_k}}A\bigr) \cdot G - \frac{1}{z_k}AG+\frac{1}{z_k}G\cdot\delta(G^{-1}AG),\ k=1,\ldots,n,\\ \label{G for z_n+1}
    \frac{\partial G}{\partial t} &= \frac{1}{t}AG-\frac{1}{t}G\cdot\delta(G^{-1}AG),\\ \label{iso for z end}
    \frac{\partial G}{\partial w_{0}} &= 0,
    \frac{\partial G}{\partial w_{k}} =G \cdot \bigl(\mathrm{ad}_{\widetilde{V}}^{-1}\mathrm{ad}_{\frac{\partial \widetilde{V}}{\partial w_{k}}}(G^{-1}AG)\bigr)-\frac{1}{w_k}AG+\frac{1}{w_k}G\cdot\delta(G^{-1}AG),\ k=2,\ldots,n.
    \end{align}
    \end{subequations}
Here the $n\times n$ diagonal matrices
   \begin{align}
        \frac{\partial U}{\partial z_k}&= \frac{1}{z_k}(u_{k+1}E_{k+1}+...+u_nE_n-z_0(E_{k+1}+...+E_n)),\\
        \frac{\partial V}{\partial w_k}&=-\frac{1}{w_k}(v_1E_1+...+v_kE_k-w_0(E_1+...+E_k)),
        \end{align}
        and
        \begin{equation}\label{eq:def of tilV}
        \widetilde{V}=\mathrm{diag}(\widetilde{v}_1,...,\widetilde{v}_n), \hspace{2mm} \text{with entries} \ \ \widetilde{v}_k=\frac{v_k-v_1}{v_2-v_1}.
   \end{equation}

In the rest of this section, we will separate and study the nonlinear ordinary differential system \eqref{A for z_n+1} and \eqref{G for z_n+1} with respect to the parameter $t$. To be more precise, we will construct a solution $A(\mathbf{z}, t,\mathbf{w} ), G(\mathbf{z}, t,\mathbf{w} )$ of the above system, from certain given solution $\widehat{A}(\mathbf{z}, \mathbf{w}), \widehat{G}(\mathbf{z}, \mathbf{w})$ of the simpler completely integrable system 
    \begin{subequations}
\begin{align}\label{eq iso for hatA z}
\frac{\partial \widehat{A}}{\partial z_0} &=  0, 
\frac{\partial \widehat{A}}{\partial z_k} = \left[\mathrm{ad}_U^{-1} \mathrm{ad}_{\frac{\partial U}{\partial z_k}} \widehat{A}, \widehat{A}\right], \quad 
\frac{\partial \widehat{A}}{\partial w_k} = 0, \quad k=1,2,\dots,n-1, \\
 \label{eq for hatG z}
\frac{\partial \widehat{G}}{\partial z_0} &=  0, 
\frac{\partial \widehat{G}}{\partial z_k} = \left(\mathrm{ad}_U^{-1} \mathrm{ad}_{\frac{\partial U}{\partial z_k}} \widehat{A}\right) \widehat{G} - \frac{1}{z_k}\widehat{A}\widehat{G}, \quad 
k=1,\dots,n-1, \\ \label{eq for hatG w}
\frac{\partial \widehat{G}}{\partial w_0} &= 0, 
\frac{\partial \widehat{G}}{\partial w_k} = \widehat{G} \left(\mathrm{ad}_{\widetilde{V}}^{-1} \mathrm{ad}_{\frac{\partial \widetilde{V}}{\partial w_k}} (\widehat{G}^{-1}\widehat{A}\widehat{G})\right) - \frac{1}{w_k}\widehat{A}\widehat{G}, \quad 
k=2,\dots,n.
\end{align}
\end{subequations}

\begin{theorem}\label{thm:asy of iso A from hatA}
Given any multi-valued meromorphic solution $\widehat{A}(\mathbf{z},\mathbf{w} )\in \frak{gl}_n, \widehat{G}(\mathbf{z},\mathbf{w})\in{\rm GL}_n$ of the system \eqref{eq iso for hatA z}-\eqref{eq for hatG w}, if the eigenvalues $\mu_1,...,\mu_n$ of $\widehat{A}$ satisfy 
\begin{align}\label{eq:shrinking in t}
    \max_{1 \leq j,k \leq n} |\mathrm{Re}(\mu_j - \mu_k)| < 1,
\end{align}
then there exists a unique solution of isomonodromy equations \eqref{iso for z begin}-\eqref{iso for z end} $A(\mathbf{z}, t,\mathbf{w}),G(\mathbf{z}, t,\mathbf{w})$, such that
\begin{align}\label{eq:lim of t A,G}
    \lim_{t\rightarrow 0} A(\mathbf{z},t,\mathbf{w})=\widehat{A}(\mathbf{z}),\quad \lim_{t\rightarrow 0} t^{-\widehat{A}}G(\mathbf{z},t,\mathbf{w})\cdot w_1^{\delta(G^{-1}AG)}=\widehat{G}(\mathbf{z},\mathbf{w}).
\end{align}
\end{theorem}
We remark that the eigenvalues $\mu_1,...,\mu_n$ of $\widehat{A}$ are constant due to the equation \eqref{eq iso for hatA z}.
The proof of this theorem will be presented in the remainder of this section. First we fix $\mathbf{z},\mathbf{w}$, and just consider the system \eqref{A for z_n+1} and \eqref{G for z_n+1} with respect to $t$:
\begin{align}\label{eq:iso for only t A}
     \frac{\partial A}{\partial t} &= \biggl[U, \frac{1}{t}GVG^{-1}\biggr] = \biggl[U, \frac{1}{t}(GVG^{-1}-w_0\mathrm{Id})\biggr],\\ \label{eq:iso for only t G}
    \frac{\partial G}{\partial t} &= \frac{1}{t}AG-\frac{1}{t}G\cdot\delta(G^{-1}AG).
\end{align}
In Proposition \ref{thm:asym in t}, we prove the existence of the solution of \eqref{eq:iso for only t A} and \eqref{eq:iso for only t G} with the behaviour \eqref{eq:lim of t A,G} for fixed $\mathbf{z}$ and $\mathbf{w}$. Then we finish the proof by showing that when varying $\mathbf{z}$ and $\mathbf{w}$, these solutions also satisfy the whole isomonodromy equations \eqref{iso for z begin}-\eqref{iso for z end}.

\begin{prop}\label{thm:asym in t}
Consider $\widehat{A},\widehat{G}$ as in Theorem \ref{thm:asy of iso A from hatA}. Fix $\mathbf{z}=(z_0,...,z_{n-1}),\mathbf{w}=(w_{n-1},...,w_2,w_0),$ and 
choose real constants $\sigma_1,\sigma_2,K > 0$ such that:
\begin{align}\label{eq:eigen cond in asy t}
&\max_{1 \leq j,k \leq n} |\mathrm{Re}(\mu_j - \mu_k)| < \sigma_2<\sigma_1 < 1,\\ \label{eq:bound in asy t}
&\max\left\{\left|\frac{U}{u_n-u_1}\right|,|\widehat{A}|, |\widehat{G}|, \left| \widehat{G}\frac{(u_n-u_1)(V - w_0 \mathrm{Id})}{t}\widehat{G}^{-1}\right|\right\} < K.
\end{align}
Then for any angle $\phi$, there exists an $\epsilon > 0$, such that 
the system \eqref{eq:iso for only t A} and \eqref{eq:iso for only t G} admits a unique solution $ A(\mathbf{z}, t,\mathbf{w}) $, $ G(\mathbf{z}, t,\mathbf{w}) $  in $\mathcal{S}_{\epsilon,\phi} := \{ t \in \mathbb{C} \mid 0 < |t| < \epsilon, |\arg t| < \phi \}$, satisfying:
\begin{align}\label{lim for A}
|A - \widehat{A}| &\leq K|t|^{1-\sigma_1}, \quad |t^{-\widehat{A}}(A - \widehat{A})t^{\widehat{A}}| \leq K|t|^{1-\sigma_1}, \\ \label{lim for G}
|t^{-\widehat{A}}G w_1^{\delta(G^{-1}AG)} - \widehat{G}| &\leq K|t|^{1-\sigma_1}.
\end{align}

\end{prop} 

\begin{proof}
 Set $\widetilde{U}=\frac{U}{u_n-u_1},\ \widetilde{B}=t^{-\widehat{A}}\cdot \frac{(u_n-u_1)(G(t)VG^{-1}(t)-w_0\mathrm{Id})}{t}\cdot t^{\widehat{A}},\ \widetilde{G}= t^{-\widehat{A}}\cdot G\cdot w_1^{\delta(G^{-1}AG)}$, and rewrite \eqref{eq:iso for only t A} and \eqref{eq:iso for only t G} as
 \begin{subequations}
\begin{align}\label{eq:Ahat at iter}
       \frac{\partial A}{\partial t} &= \left[\widetilde{U},t^{\widehat{A}}\widetilde{B}t^{-\widehat{A}}\right] ,\\
       \label{eq:Gtil at iter}
       \frac{\partial \widetilde{G}}{\partial t}&= \frac{1}{t}t^{-\widehat{A}}(A-\widehat{A})t^{\widehat{A}}\cdot \widetilde{G}.
\end{align}
 \end{subequations}
Note that the equation for the auxiliary matrix function $\widetilde{B}(t)$ is
 \begin{align}
     \label{eq:Btil at iter}
       \frac{\partial\widetilde{B}(t)}{\partial t}= \frac{1}{t}\left[t^{-\widehat{A}}(A-\widehat{A})t^{\widehat{A}},\widetilde{B}\right] .
\end{align}
Let $A^{(0)}(t) = \widehat{A},\ \widetilde{G}^{(0)}(t)=\widehat{G},\  \widetilde{B}^{(0)}(t) = \widehat{B}=\widehat{G}\frac{(u_n-u_1)(V - w_0 \mathrm{Id})}{t}\widehat{G}^{-1}, $ and define $A^{(k)}(t), \widetilde{G}^{(k)}(t), \widetilde{B}^{(k)}(t),\  (k=0,1,2,...)$ recursively by
\begin{align}\label{picard iter for A}
    A^{(k)}(t)&=\widehat{A}+\int_{0}^{t}\left[\widetilde{U},s^{\widehat{A}}\widetilde{B}^{(k-1)}s^{-\widehat{A}}\right]ds,\\ \label{picard iter for G}
    \widetilde{G}^{(k)}(t)&= \widehat{G}+\int_0^{t}\frac{1}{s}s^{-\widehat{A}}(A^{(k-1)}(s)-\widehat{A})s^{\widehat{A}}\cdot \widetilde{G}^{(k-1)}ds,\\ \label{picard iter for B}
     \widetilde{B}^{(k)}(t)&=\widehat{B}+\int_0^{t}\frac{1}{s}\left[s^{-\widehat{A}}(A^{(k-1)}(s)-\widehat{A})s^{\widehat{A}},\widetilde{B}^{(k-1)}\right].
\end{align}
Here the path of integration is $\{s=re^{i\theta}|0<r<|t|,\theta=\mathrm{arg}\ t \}$.

Let $\delta_0$ be a constant such that $0<\delta_0<1$. For a constant $C_K$ dependent of $K$, and a small enough $\epsilon$ such that $0<|t|<\epsilon,$ we claim the following
\begin{subequations}
\begin{align}\label{eq:hatpicard-A1}
    &|A^{(k)}(t)-\widehat{A}|\leq K|t|^{1-\sigma_1},\\ \label{eq:hatpicard-A2}
    &|t^{-\widehat{A}}(A^{(k)}(t)-\widehat{A})t^{\widehat{A}}|\leq \frac{1}{\sigma_2}4K^2C^2_K|t|^{1-\sigma_2},\\ \label{eq:hatpicard-B1}
    &|\widetilde{B}^{(k)}(t)-\widehat{B}|\leq K|t|^{1-\sigma_1},\\ \label{eq:hatpicard-G1}
    &|\widetilde{G}^{(k)}(t)-\widehat{G}|\leq K|t|^{1-\sigma_1},\\ \label{eq:hatpicard-A3}
    &|A^{(k)}(t)-A^{(k-1)}(t)|\leq K\delta_0^{k-1}|t|^{1-\sigma_1},\\ \label{eq:hatpicard-A4}
    &|t^{-\widehat{A}}(A^{(k)}(t)-A^{(k-1)})t^{\widehat{A}}|\leq K\delta_0^{k-1}|t|^{1-\sigma_2},\\ \label{eq:hatpicard-B2}
    &|\widetilde{B}^{(k)}(t)-\widetilde{B}^{(k-1)}|\leq K\delta_0^{k-1}|t|^{1-\sigma_1},\\ \label{eq:hatpicard-G2}
    &|\widetilde{G}^{(k)}(t)-\widetilde{G}^{(k-1)}|\leq K\delta_0^{k-1}|t|^{1-\sigma_1}.
\end{align}
\end{subequations} 
To continue, we need the following lemma established in \cite{Jimbo-Miwa-Sato}:
\begin{lemma}[\cite{Jimbo-Miwa-Sato}]\label{lem:esti by jimbo}
    Let $X(t)$ and $Y(t)$ be $n\times n$ matrix functions which satisfy $\left|X(t)\right|\leq C_1, \left|Y(t)\right|\leq C_2$ for $t\in \mathcal{S}_{\epsilon,\phi}$, and 
    let $f(t)$ be a holomorphic function in $\{t:\left|t\right|<1\}$, such that $|f(t)|\leq C_f$. Let $M$ be a matrix such that $|M|<K$ and have eigenvalues $\mu_i,\ i=1,...,n$, such that $ \max_{1 \leq j,k \leq n} |\mathrm{Re}(\mu_j - \mu_k)| < \sigma_2<\sigma_2<1$ for some constant  $\sigma_2,\sigma_2$.
Then there exists a constant $C_K$ (only dependent on $K$), and a constant $\delta$ independent of $C_1,C_2$ nor $M$,  such that for any $\epsilon$ satisfying $0<\epsilon<\delta$, the following are valid for $t\in \mathcal{S}_{\epsilon,\phi}$
\begin{align}\label{eq:esti by jimbo 1}
   & \left|t^{M}X(t)t^{-M}\right|\leq C_1C_K\left|t\right|^{-\sigma_2},\quad \left|t^{M}Y(t)t^{-M}\right|\leq C_2C_K\left|t\right|^{-\sigma_2},\\ \label{eq:esti by jimbo 2}
    &\left|\int_{0}^{t}t^{M}X(s)s^{-M}Y(s)s^{M}f(s)t^{-M}\mathrm{d}s \right|\leq C_1C_2C^2_KC_f\left|t\right|^{1-\sigma_2},\\ \label{eq:esti by jimbo 3}
    &\left|\int_{0}^{t}t^{M}s^{-M}Y(s)s^{M}X(s)f(s) t^{-M}\mathrm{d}s\right|\leq C_1C_2C^2_KC_f\left|t\right|^{1-\sigma_2}.
\end{align}
\end{lemma}
\begin{proof}[Proof of Lemma~\ref{lem:esti by jimbo}]
    Consider the Jordan-Chevalley decomposition $D+N$ of $M$, where $D=\mu_1P_1+...+\mu_nP_n$ is the semisimple part of $M$, and $N$ is the nilpotent part of $M$. Then
    \begin{align*}
        \left|t^{M}X(t)t^{-M}\right|=\left|\sum_{1\leq i,j\leq n}t^{\mu_i-\mu_j}P_it^{N}X(t)t^{-N}P_j\right|, \quad   \left|t^{M}Y(t)t^{-M}\right|=\left|\sum_{1\leq i,j\leq n}t^{\mu_i-\mu_j}P_it^{N}Y(t)t^{-N}P_j\right|.
    \end{align*}
    Since $D$ and $N$ are both the polynomials of $M$ with coefficients bounded by finite sum and product of elements of $M$, the inequalities \eqref{eq:esti by jimbo 1} hold.

    For \eqref{eq:esti by jimbo 2}, consider 
    \begin{align*}
        \left|\int_{0}^{t}t^{M}X(s)s^{-M}Y(s)s^{M}f(s)t^{-M}\mathrm{d}s \right|&=\left|\int_0^{t}t^{M}X(s)t^{-M}\left(\frac{s}{t}\right)^{-M}Y(s)\left(\frac{s}{t}\right)^{M}f(s)\mathrm{d}s\right|\\
        &\leq\int_0^{1}|t|\left|t^{M}X(s)t^{-M}\right|\left|\widetilde{s}^{-M}Y(s)\widetilde{s}^{M}\right||f(s)|\mathrm{d}\widetilde{s}\leq C_1C_2C_K^2C_f|t|^{1-\sigma_2}.
    \end{align*}
    
    The proof of \eqref{eq:esti by jimbo 3} is similar.
\end{proof}

Now return to the proof of the proposition. Assume \eqref{eq:hatpicard-A1}-\eqref{eq:hatpicard-G1} hold for $k-1$. In Lemma \ref{lem:esti by jimbo}, taking $M=\widehat{A}$, we have 
\begin{align*}
    |A^{(k)}(t)-\widehat{A}| \leq 2\int_0^{|t|} \left|\widetilde{U}\right|\left|s^{\widehat{A}}\widetilde{B}^{(k-1)}s^{-\widehat{A}}\right|\mathrm{d}|s|
    \leq \frac{1}{\sigma_2}4K^2C_K |t|^{1-\sigma_2}\leq K|t|^{1-\sigma_1}.
\end{align*}
Here in the second inequality we choose $\epsilon<1$  to make $|\widetilde{B}^{(k-1)}|\leq 2K$, and choose $\frac{1}{\sigma_2}4K^2C^2_K\epsilon^{\sigma_1-\sigma_2}<1$ in the third inequality.
Similarly,
\begin{align*}
    |t^{-\widehat{A}}(A^{(k)}(t)-\widehat{A})t^{\widehat{A}}|&= \left| t^{-\widehat{A}}\int_{0}^{t}\left[\widetilde{U},s^{\widehat{A}}\widetilde{B}^{(k-1)}s^{-\widehat{A}}\right]\mathrm{d}s\   t^{\widehat{A}}\right|\\
    &\leq  \int_{0}^{t} \left|t^{-\widehat{A}}\widetilde{U}s^{\widehat{A}}\widetilde{B}^{(k-1)}s^{-\widehat{A}}t^{\widehat{A}}\right|\mathrm{d}s\  + \int_{0}^{t} \left|t^{-\widehat{A}}s^{\widehat{A}}\widetilde{B}^{(k-1)}s^{-\widehat{A}}\widetilde{U}t^{\widehat{A}}\right|\mathrm{d}s\\
    &\leq \frac{1}{\sigma_2}4K^2C^2_K t^{1-\sigma_2},
 \end{align*}
\begin{align*}
    |\widetilde{B}^{(k)}(t)-\widehat{B}|\leq 2\int_0^{|t|} \frac{1}{|s|}\left|s^{-\widehat{A}}(A^{(k-1)}(s)-\widehat{A})s^{\widehat{A}}\right||\widetilde{B}^{(k-1)}|\mathrm{d|s|} \leq \frac{1}{\sigma_2}4K^2C^2_K|t|^{1-\sigma_2}\leq K|t|^{1-\sigma_1},
\end{align*}
\begin{align*}
    |\widetilde{G}^{(k)}(t)-\widehat{G}|\leq \int_0^{|t|}\frac{1}{|s|}\left|s^{-\widehat{A}}(A^{(k-1)}(s)-\widehat{A})s^{\widehat{A}}\right|\cdot |\widetilde{G}^{(k-1)}|\mathrm{d}|s|\leq \frac{1}{\sigma_2}2K^2C^2_K|t|^{1-\sigma_{2}}\leq K|t|^{1-\sigma_1}.
\end{align*}

For \eqref{eq:hatpicard-A3}-\eqref{eq:hatpicard-G2}, assume they hold for $k-1$, then also by Lemma \ref{lem:esti by jimbo}, we have
\begin{align*}
    \left|A^{(k+1)}(t)-A^{(k)}(t)\right|&\leq\int_0^t\left|\widetilde{U}s^{\widehat{A}}(\widetilde{B}^{(k)}-\widetilde{B}^{(k-1)})s^{-\widehat{A}}\right|\mathrm{d}s+\int_0^t\left|s^{\widehat{A}}(\widetilde{B}^{(k)}-\widetilde{B}^{(k-1)})s^{-\widehat{A}}\widetilde{U}\right|\mathrm{d}s\\
    &\leq\int_0^t\left|\widetilde{U}s^{\widehat{A}}(K\delta_0^{k-1}|s|^{1-\sigma_1})s^{-\widehat{A}}\right|\mathrm{d}s+\int_0^t\left|s^{\widehat{A}}(K\delta_0^{k-1}|s|^{1-\sigma_1})s^{-\widehat{A}}\widetilde{U}\right|\mathrm{d}s\\
    &\leq 2K^2C_K\delta_0^{k-1}|t|^{2-\sigma_1-\sigma_2}\leq K\delta_0^k|t|^{1-\sigma_1},
    \end{align*}
    \begin{align*}
         \left |t^{-\widehat{A}}(A^{(k+1)}(t)-A^{(k)}(t))t^{\widehat{A}}\right|&\leq \int_0^t\left|t^{-\widehat{A}}\widetilde{U}s^{\widehat{A}}(\widetilde{B}^{(k)}-\widetilde{B}^{(k-1)})s^{-\widehat{A}}t^{\widehat{A}}\right|\mathrm{d}s+\int_0^t\left|t^{-\widehat{A}}s^{\widehat{A}}(\widetilde{B}^{(k)}-\widetilde{B}^{(k-1) })s^{-\widehat{A}}\widetilde{U}t^{-\widehat{A}}\right|\mathrm{d}s\\
         &\leq 2K^2C^2_K\delta_0^{k-1}\delta_0^k|t|^{2-\sigma_1-\sigma_2}\leq K\delta_0^k|t|^{1-\sigma_2}.
    \end{align*}
    for $\epsilon$ such that $2KC^2_K\epsilon^{1-\sigma_1}<\delta_0$. Simlarly, we have
    \begin{align*}
        \left|\widetilde{B}^{(k+1)}(t)-\widetilde{B}^{(k)}(t)\right|&= \left|\int_0^{t}\frac{1}{s}\left[s^{-\widehat{A}}(A^{(k)}-A^{(k-1)})s^{\widehat{A}},\widetilde{B}^{(k)}\right]+\frac{1}{s}\left[s^{-\widehat{A}}(A^{(k-1)}-\widehat{A})s^{\widehat{A}},\widetilde{B}^{(k)}-\widetilde{B}^{(k-1)}\right]\mathrm{d}s\right|\\
        &\leq  8K^2C^2_K\delta_0^{k-1}|t|^{1-\sigma_2}\leq K\delta_0^{k}|t|^{1-\sigma_1},                                                                                                      
    \end{align*}
    \begin{align*}
        |\widetilde{G}^{(k+1)}(t)-\widetilde{G}^{(k)}|&=\left|\int_0^t \frac{1}{s}s^{-\widehat{A}}(A^{(k)}-A^{(k-1)})s^{\widehat{A}}\widetilde{G}^{(k)}+\frac{1}{s}s^{-\widehat{A}}(A^{(k)}-\widehat{A})s^{\widehat{A}}(\widetilde{G}^{(k)}-\widetilde{G}^{(k-1)})\mathrm{d}s \right|\\
        &\leq 4K^2C^2_K\delta_0^{k-1}|t|^{1-\sigma_2}\leq K\delta_0^{k}|t|^{1-\sigma_1},
    \end{align*}
    for $\epsilon$ such that $8KC^2_K\epsilon^{\sigma_1-\sigma_2}\leq \delta_0$. 

    Thus we prove the existence of the solutions of \eqref{eq:Ahat at iter}-\eqref{eq:Btil at iter} in a neighborhood of $p_0$. Set
    \begin{align}\label{eq:A cons by inter}
        A(\mathbf{z},t,\mathbf{w}) &= \lim_{k\rightarrow \infty} A^{(k)}(\mathbf{z},t,\mathbf{w}),\\ \label{eq:G cons by inter}
        G(\mathbf{z},t,\mathbf{w})w_1^{\delta(G^{-1}AG)} &= t^{\widehat{A}}\cdot\lim_{k\rightarrow \infty} \widetilde{G}^{(k)}(\mathbf{z},t,\mathbf{w}),\\ \nonumber
        B(\mathbf{z},t,\mathbf{w}) &= t\cdot t^{\widehat{A}}\lim_{k\rightarrow \infty} \widetilde{B}^{(k)}(\mathbf{z},t,\mathbf{w})t^{-\widehat{A}}+w_0\mathrm{Id}.
    \end{align}
    For the uniqueness of $A,\widetilde{B},\widetilde{G}$, suppose $(A,\widetilde{B},\widetilde{G})$ and $(A',\widetilde{B}',\widetilde{G}')$ are two solutions of \eqref{eq:Ahat at iter}-\eqref{eq:Btil at iter}, with the same asymptotic leading terms $\widehat{A},\widehat{G},\widehat{B}=\widehat{G}\frac{(u_n-u_1)(V - w_0 \mathrm{Id})}{t}\widehat{G}^{-1}$. Then we have the following integral equations
    \begin{align*}
        A-A'&=\int_0^{t}\left[\widetilde{U},s^{\widehat{A}}(\widetilde{B}-\widetilde{B}')s^{-\widehat{A}}\right]\ \mathrm{d}s ,\\
        \widetilde{B}-\widetilde{B}'&=\int_0^{t}\frac{1}{s}\left[s^{-\widehat{A}}\left(A(s)-A'(s)\right)s^{\widehat{A}},\widetilde{B}\right]+\frac{1}{s}\left[s^{-\widehat{A}}(A'(s)-\widehat{A})s^{\widehat{A}},\widetilde{B}-\widetilde{B}'\right]\ \mathrm{d}s,\\
        \widetilde{G}-\widetilde{G}'&=\int_0^{t}\frac{1}{s}s^{-\widehat{A}}\left(A(s)-A'(s)\right)s^{\widehat{A}}\widetilde{G}+\frac{1}{s}s^{-\widehat{A}}(A'(s)-\widehat{A})s^{\widehat{A}}(\widetilde{G}-\widetilde{G}')\ \mathrm{d}s.
    \end{align*}
    By Lemma \ref{lem:esti by jimbo}, for a small enough $\epsilon$ such that $0<|t|\leq\epsilon$, we have
    \begin{align*}
        \max_{0<|t|\leq\epsilon}|A-A'|&\leq \frac{1}{\sigma_1}2KC_K\max_{0<|t|\leq\epsilon}|\widetilde{B}-\widetilde{B}'|\cdot|t|^{1-\sigma_1},\\
        \max_{0<|t|\leq\epsilon}|t^{-\widehat{A}}(A-A')t^{\widehat{A}}|&\leq\frac{1}{\sigma_1}2KC^2_K\max_{0<|t|\leq\epsilon}|\widetilde{B}-\widetilde{B}'|\cdot|t|^{1-\sigma_1}.
        \end{align*}
        Substitute above two inequalities to the integral equations of $\widetilde{B}-\widetilde{B}'$ and $\widetilde{G}-\widetilde{G}' $, we have
        \begin{align*}
        \max_{0<|t|\leq\epsilon}|\widetilde{B}-\widetilde{B}'|&\leq \frac{1}{\sigma_1^2}4K^2C^2_K\max_{0<|t|\leq\epsilon}|\widetilde{B}-\widetilde{B}'|\cdot|t|^{1-\sigma_1}+\frac{1}{\sigma_1}2KC_K^2\left(\max_{0<|t|\leq\epsilon}|\widetilde{B}-\widetilde{B}'|\right)^2\cdot|t|^{1-\sigma_1}\\
          \max_{0<|t|\leq\epsilon}|\widetilde{G}-\widetilde{G}'|&\leq \frac{1}{\sigma_1^2}2K^2C^2_K\max_{0<|t|\leq\epsilon}|\widetilde{B}-\widetilde{B}'|\cdot|t|^{1-\sigma_1}+\frac{1}{\sigma_1}2KC_K^2\max_{0<|t|\leq\epsilon}|\widetilde{B}-\widetilde{B}'|\cdot \max_{0<|t|\leq\epsilon}|\widetilde{G}-\widetilde{G}'|\cdot|t|^{1-\sigma_1}
    \end{align*}
    Thus $(A,\widetilde{B},\widetilde{G})=(A',\widetilde{B}',\widetilde{G}')$ for $0\leq |t|\leq \epsilon$ for a small $\epsilon$.

Finally, since $GVG^{-1}$ satisfies the same equations as $B$'s, and has the same asymptotics, by the uniqueness $B=GVG^{-1}$. Therefore, $A,G$ constructed in \eqref{eq:A cons by inter} and \eqref{eq:G cons by inter} are the solution of \eqref{eq:iso for only t A} and \eqref{eq:iso for only t G}.

\end{proof} 

Now we use the compatibility conditions of \eqref{iso for z begin}-\eqref{iso for z end} to complete the analysis of asymptotics of $A,G$.

\begin{proof}[Proof of Theorem~\ref{thm:asy of iso A from hatA}]
Following Proposition \ref{thm:asym in t}, we remain to prove that when varying $(\mathbf{z},\mathbf{w})$, the functions $A(\mathbf{z}, t,\mathbf{w}),G(\mathbf{z}, t,\mathbf{w})$ constructed above also satisfy the isomonodromy equations with respect to $\mathbf{z}$ and $\mathbf{w}$. 

Notice that when $\widehat{A}$ satisfies \eqref{eq iso for hatA z}, the Jordan form of $\widehat{A}$ is unchanged. Thus for any $ p_0 = (z^{(0)}_0, \dots,w_0^{(0)})$, we can choose a neighborhood $D(p_0)$, where the conditions \eqref{eq:eigen cond in asy t} and \eqref{eq:bound in asy t} hold uniformly, thereby allowing us to use Lemma \ref{lem:esti by jimbo} and follow the same procedure to prove the Proposition \ref{thm:asym in t}. As a result, the limit \eqref{eq:lim of t A,G} is locally uniformly with respect to $\mathbf{z},\mathbf{w}$, and when $\widehat{A},\widehat{G}$ are holomorphic in $D(p_0)$, $A,G$ are also holomorphic in $(\mathbf{z},\mathbf{w})$.

 Let
    \begin{align}
    \widetilde{G}&=t^{-\widehat{A}}Gw_1^{\delta(G^{-1}AG)},\\
    L_k&=\bigl[\mathrm{ad}_U^{-1}\mathrm{ad}_{\frac{\partial U}{\partial z_k}}A, A\bigr] + t\biggl[\frac{\partial U}{\partial z_k}, t^{\widehat{A}}\widetilde{G}\frac{V}{t}\widetilde{G}^{-1}t^{-\widehat{A}}\biggr] - \frac{t}{z_k}\left[U,t^{\widehat{A}}\widetilde{G}\frac{V}{t}\widetilde{G}^{-1}t^{-\widehat{A}}\right],\\ \label{eq:expr of Jk}
    J_k&=\bigl(\mathrm{ad}_U^{-1}\mathrm{ad}_{\frac{\partial U}{\partial z_k}}A\bigr) \cdot \widetilde{G} - \frac{1}{z_k}t^{-\widehat{A}}At^{\widehat{A}}\widetilde{G}+t^{-\widehat{A}}\left(\mathrm{ad}_U^{-1}\mathrm{ad}_{\frac{\partial U}{\partial z_k}}(A-\widehat{A})\right)t^{\widehat{A}}\widetilde{G}.
    \end{align}
    For the equations in $z_k$, we need to prove
    \begin{align*}
        \frac{\partial A}{\partial z_k} = L_k \ \text{ and } \
       \frac{\partial \widetilde{G}}{\partial z_k}=J_k.
    \end{align*}
    To derive it, we proceed as follows:
    first we claim that $\left(\frac{\partial A}{\partial z_k}, \frac{\partial \widetilde{G}}{\partial z_k}\right)$ and $(L_k, J_k)$ satisfy the same equations in $t$:
\begin{align*}
    \frac{\partial X}{\partial t} &= \biggl[\frac{\partial U}{\partial z_k}, \frac{1}{t}t^{\widehat{A}}\widetilde{G}V\widetilde{G}^{-1}t^{-\widehat{A}}\biggr] \nonumber \\
    &\quad + \frac{1}{t}\left[U, \biggl[t^{\widehat{A}}Y\widetilde{G}^{-1}t^{-\widehat{A}} + \mathrm{ad}_U^{-1}\mathrm{ad}_{\frac{\partial U}{\partial z_k}}\widehat{A} - t^{\widehat{A}}\bigl(\mathrm{ad}_U^{-1}\mathrm{ad}_{\frac{\partial U}{\partial z_k}}\widehat{A}\bigr)t^{-\widehat{A}}, t^{\widehat{A}}\widetilde{G}V\widetilde{G}^{-1}t^{-\widehat{A}}\biggr] - \frac{1}{z_k}t^{\widehat{A}}\widetilde{G}V\widetilde{G}^{-1}t^{-\widehat{A}}\right], \\
    \frac{\partial Y}{\partial t} &= \frac{1}{t}\left[\mathrm{ad}_U^{-1}\mathrm{ad}_{\frac{\partial U}{\partial z_k}}\widehat{A} , t^{-\widehat{A}}(A-\widehat{A})t^{\widehat{A}}\right]\widetilde{G} -\frac{1}{t}t^{-\widehat{A}}\left[\mathrm{ad}_U^{-1}\mathrm{ad}_{\frac{\partial U}{\partial z_k}}\widehat{A} ,A\right]t^{\widehat{A}}\widetilde{G}+\frac{1}{t}t^{-\widehat{A}}Xt^{\widehat{A}}\widetilde{G}+\frac{1}{t}t^{-\widehat{A}}(A-\widehat{A})t^{\widehat{A}}Y.
\end{align*}    
For $(L_k, J_k)$, it is a direct calculation and for $\left(\frac{\partial A}{\partial z_k}, \frac{\partial \widetilde{G}}{\partial z_k}\right)$, notice 
    \begin{align*}
       \frac{\partial}{\partial t}\left( \frac{\partial A}{\partial z_k}\right)&=\frac{\partial}{\partial z_k}\left( \frac{\partial A}{\partial t}\right)=\frac{\partial}{\partial z_k}\left( \left[U,\frac{1}{t}t^{\widehat{A}}\widetilde{G}V\widetilde{G}^{-1}t^{-\widehat{A}} \right]\right),\\
        \frac{\partial}{\partial t}\left( \frac{\partial \widetilde{G}}{\partial z_k}\right)&=\frac{\partial}{\partial z_k}\left( \frac{\partial \widetilde{G}}{\partial t}\right)=\frac{\partial}{\partial z_k}\left( \frac{1}{t}t^{-\widehat{A}}(A-\widehat{A})t^{\widehat{A}}\widetilde{G}\right).
    \end{align*}
    Thus let $P=\frac{\partial A}{\partial z_k}-L_k,\ Q=\frac{\partial \widetilde{G}}{\partial z_k}-J_k$, we have 
    \begin{subequations}
    \begin{align}\label{eq:eq for P in vari of iso ode}
        \frac{\partial P}{\partial t}&=   \biggl[U, t^{\widehat{A}}\left[Q\widetilde{G}^{-1} ,\widetilde{G}\frac{V}{t}\widetilde{G}^{-1}\right]t^{-\widehat{A}}\biggr],\\ \label{eq:eq for Q in vari of iso ode}
             \frac{\partial Q}{\partial t}&= \frac{1}{t}\left(t^{-\widehat{A}}Pt^{\widehat{A}}\widetilde{G}+t^{-\widehat{A}}(A-\widehat{A})t^{\widehat{A}}Q\right).
    \end{align}
        \end{subequations}
    
    Secondly, we consider the asymptotics of $P$ and $Q$, as $t\rightarrow0$:
    since in \eqref{lim for A}-\eqref{lim for G}, $K$ is a constant in $D(p_0)$, so
    \begin{align*}
        \lim_{t\rightarrow0}\frac{\partial A}{\partial z_k}=\frac{\partial \widehat{A}}{\partial z_k},\quad
        \lim_{t\rightarrow0}\frac{\partial \widetilde{G}}{\partial z_k}=\frac{\partial \widehat{G}}{\partial z_k}.
    \end{align*}
    And using Proposition \ref{thm:asym in t}, we see that $L_k= \left[\mathrm{ad}_U^{-1} \mathrm{ad}_{\frac{\partial U}{\partial z_k}} \widehat{A}, \widehat{A}\right]+ O(t^{1-\sigma_1})$, as $t\rightarrow 0$, thus 
    \begin{align*}
        \lim_{t\rightarrow0}P=0.
    \end{align*}
    To show $ \lim_{t\rightarrow0}Q=0$, compare the expression of $J_k$ \eqref{eq:expr of Jk}, we only need to prove 
    \begin{align}\label{eq:t-adad difAt}
        \lim_{t\rightarrow0}t^{-\widehat{A}}\left(\mathrm{ad}_U^{-1}\mathrm{ad}_{\frac{\partial U}{\partial z_k}}(A-\widehat{A})\right)t^{\widehat{A}}=0.
    \end{align}
    In fact 
    \begin{align*}
        t^{-\widehat{A}}\left(\mathrm{ad}_U^{-1}\mathrm{ad}_{\frac{\partial U}{\partial z_k}}(A-\widehat{A})\right)t^{\widehat{A}}&= t^{-\widehat{A}}\mathrm{ad}_U^{-1}\mathrm{ad}_{\frac{\partial U}{\partial z_k}}\left(\int_0^{t}\mathrm{ad}_{\widetilde{U}}(s^{\widehat{A}}\widetilde{B}s^{-\widehat{A}})\mathrm{d}s\right)t^{\widehat{A}}\\
        &=t^{-\widehat{A}}\left(\int_0^{t}\mathrm{ad}_U^{-1}\mathrm{ad}_{\frac{\partial U}{\partial z_k}}\mathrm{ad}_{\widetilde{U}}(s^{\widehat{A}}\widetilde{B}s^{-\widehat{A}})\mathrm{d}s\right)t^{\widehat{A}}\\
        &=t^{-\widehat{A}}\left(\int_0^{t}\left[{\frac{1}{u_n-u_1}\frac{\partial U}{\partial z_k}}, \ s^{\widehat{A}}\widetilde{B}s^{-\widehat{A}}\right]\mathrm{d}s\right)t^{\widehat{A}}.
    \end{align*}
    then by Lemma \ref{lem:esti by jimbo}, \eqref{eq:t-adad difAt} is proved.

    Thirdly, we consider the right side of equations \eqref{eq:eq for P in vari of iso ode} and \eqref{eq:eq for Q in vari of iso ode}. By Theorem \ref{thm:asym in t}, for  a small enough $\epsilon$ such that $0<|t|<\epsilon$, there exists a constant $0<\sigma<1$ and a constant $C$, such that
    \begin{align*}
       \left| \biggl[U, \ t^{\widehat{A}}\left[Q\widetilde{G}^{-1} , \ \widetilde{G}\frac{V}{t}\widetilde{G}^{-1}\right]t^{-\widehat{A}}\biggr]\right|\leq C|t|^{-\sigma}.
    \end{align*}
    And since
    \begin{align*}
      |  t^{-\widehat{A}}Pt^{\widehat{A}}|&=\left|\frac{\partial }{\partial z_k}(t^{-\widehat{A}}At^{\widehat{A}}-\widehat{A})-\left[\mathrm{ad}_U^{-1} \mathrm{ad}_{\frac{\partial U}{\partial z_k}} \widehat{A}, t^{-\widehat{A}}At^{\widehat{A}}-\widehat{A}\right]-\left[t^{-\widehat{A}}\left(\mathrm{ad}_U^{-1}\mathrm{ad}_{\frac{\partial U}{\partial z_k}}(A-\widehat{A})\right)t^{\widehat{A}},t^{-\widehat{A}}At^{\widehat{A}}\right]+O(|t|^{1-\sigma_1})\right|\\
      &\leq C|t|^{1-\sigma},
    \end{align*}
     we also have 
    \begin{align*}
        \left| \frac{1}{t}\left(t^{-\widehat{A}}Pt^{\widehat{A}}\widetilde{G}+t^{-\widehat{A}}(A-\widehat{A})t^{\widehat{A}}Q\right)\right|\leq C'|t|^{-\sigma},
    \end{align*}
    for some constant $C'$.
    Thus, by Picard iteration we see that $P=Q=0$.

    The equations in $w_k$ of $A,G$ can be obtained in the same way, set
    \begin{align*}
        L_k'&= \biggl[U, t^{\widehat{A}}\widetilde{G}\frac{\partial V}{\partial w_{k}}\widetilde{G}^{-1}t^{-\widehat{A}}\biggr],\\
        J_k'&= \widetilde{G} \cdot \left(\mathrm{ad}_{\widetilde{V}}^{-1}\mathrm{ad}_{\frac{\partial \widetilde{V}}{\partial w_{k}}}\left(w_1^{-\delta(G^{-1}AG)}G^{-1}AGw_1^{\delta(G^{-1}AG)}\right)\right)-\frac{1}{w_k}t^{-\widehat{A}}At^{\widehat{A}}\widetilde{G}.
    \end{align*}
    And let $P'=\frac{\partial A}{\partial w_k}-L_k',\ Q'= \frac{\partial \widetilde{G}}{\partial w_k}-J_k'$, we can derive
    \begin{align}
        \frac{\partial P'}{\partial t}&=\left[U,\frac{1}{t}t^{\widehat{A}}(Q'V\widetilde{G}^{-1}-\widetilde{G}V\widetilde{G}^{-1}Q'\widetilde{G}^{-1})t^{-\widehat{A}}\right],\\
        \frac{\partial Q'}{\partial t}&=\left[\frac{1}{t}t^{-\hat{A}}P't^{\hat{A}}\widetilde{G}+\frac{1}{t}t^{-\hat{A}}(A-\widehat{A})t^{\hat{A}}Q'\right].
    \end{align}
    We can verify that $\lim_{t\rightarrow0}P'=\lim_{t\rightarrow0}Q'=0$ using Proposition \ref{thm:asym in t}, and use Picard iteration to show $P'=Q'=0$.
\end{proof}

\begin{cor}\label{cor:def of tilA}
Let $A(\mathbf{z},t,\mathbf{w})$, $G(\mathbf{z},t,\mathbf{w})$ be the solution of system \eqref{iso for z begin}–\eqref{iso for z end}, constructed from a given solution $\widehat{A}(\mathbf{z},\mathbf{w}),\widehat{G}(\mathbf{z},\mathbf{w})$ of \eqref{eq iso for hatA z}-\eqref{eq for hatG w} as in Theorem \ref{thm:asy of iso A from hatA}, then
    \begin{align}
        \lim_{t\rightarrow 0} w_1^{-\delta(G^{-1}AG)}(G^{-1}AG)w_1^{\delta(G^{-1}AG)}=\widehat{G}^{-1}\widehat{A}\widehat{G}.
    \end{align}
 \end{cor}
 \begin{proof}
 It follows that
      \begin{align*}
        \lim_{t\rightarrow 0} \left( w_1^{-\delta(G^{-1}AG)}G^{-1}AG w_1^{\delta(G^{-1}AG)}\right)=\lim_{t\rightarrow 0} \left( w_1^{-\delta(G^{-1}AG)}G^{-1}t^{\widehat{A}}\right)\left(t^{-\widehat{A}}At^{\widehat{A}}\right)\left(t^{-\widehat{A}}G w_1^{\delta(G^{-1}AG)}\right)=\widehat{G}^{-1}\widehat{A}\widehat{G}.
    \end{align*}
 \end{proof}
It follows from a straightforward computation that
\begin{cor}
The matrix valued function 
\begin{equation}
\widetilde{A}(\mathbf{z},\mathbf{w}):=    -\widehat{G}^{-1}\widehat{A}\widehat{G}
\end{equation} satisfies the following equations
    \begin{align}\label{eq iso for tilA w}
        \frac{\partial \widetilde{A}}{\partial z_k}&=0,\ k=0,1,...,n-1,\\
        \frac{\partial\widetilde{A}}{\partial w_0}&=0,\   \frac{\partial\widetilde{A}}{\partial w_k}= \left[ \mathrm{ad}_{\widetilde{V}}^{-1} \mathrm{ad}_{\frac{\partial \widetilde{V}}{\partial w_k}} \widetilde{A},\widetilde{A}\right],\ k=2,...,n.
    \end{align}
\end{cor}

 \section{Monodromy of the linear system with 2 irregular singularities and Riemann-Hilbert correspondence}\label{sec:monodecom}
 
In this section, we complete the proof the Theorem \ref{mainthm}, and Theorem \ref{thm: introcatformula}. In the previous section, for certain given $(\widehat{A},\widehat{G})$, we constructed a solution $(A, G)$ to the isomonodromy equations exhibiting the corresponding asymptotic behavior. In Section \ref{sec:decom of sys}, we proceed to prove that for such a solution, the system \eqref{introeq} can be decomposed into two one second-order pole systems with residue matrices $\widehat{A}$ and $\widetilde{A}$, respectively. Furthermore, the monodromy matrices of the original system can be explicitly expressed in terms of those of the two decomposed systems.
In Section \ref{sec: asy of iso}, we construct the solution $(A, G)$  with constant boundary values $(\widehat{A}_0,G_0)$ satisfying the boundary condition. We refer to such $(A, G)$ as a shrinking solution. Subsequently, we express the monodromy matrices in terms of $(\widehat{A}_0,G_0)$.
In Section \ref{sec:almost every is shrinking}, we provide a specific criterion on the monodromy matrix, showing that the set of shrinking solutions determined by this criterion is open and dense.

\subsection{Decomposition of the system \eqref{eq: linr sys two irr in xi}-\eqref{eq:linr sys sys two irr in v} as $t\rightarrow 0$}\label{sec:decom of sys}

In this section we will study the behaviors of the solutions and monodromy data of the system \eqref{eq: linr sys two irr in xi}-\eqref{eq:linr sys sys two irr in v} as $t\rightarrow 0$.
 Consider the system \eqref{eq: linr sys two irr in xi}-\eqref{eq:linr sys sys two irr in v} in the coordinate $(\xi,z_0,z_1...,z_{n-1},t, w_{n-1},\cdots w_{2},w_0)$
\begin{subequations}
\begin{align}\label{eq:irr sys in z normal}
\frac{\partial F}{\partial \xi}&=\left(U+\frac{A}{\xi}+\frac{B}{\xi^2}\right)F,\quad B=GVG^{-1},\\ \label{eq:irr sys in t normal}
   \frac{\partial F}{\partial t}&=-\frac{B-w_0\cdot\mathrm{Id}}{t\xi}\cdot F,\\
    \frac{\partial F}{\partial z_k}&=\left(\frac{\partial U}{\partial z_k}\xi+\mathrm{ad}_U^{-1}\mathrm{ad}_{\frac{\partial U}{\partial z_k}}A+\frac{B-w_0\cdot \mathrm{Id}}{z_k\xi}\right)F,\ k=0,1,2,...,n-1,\\ \label{eq:irr sys in w normal}
    \frac{\partial F}{\partial w_{k}}&=-\frac{G\frac{\partial V}{\partial w_{k}}G^{-1}}{\xi}\cdot F,\ k=0,2,3,...,n-1.
\end{align}
\end{subequations}

First we study the decomposition of the linear system as $t\rightarrow 0$. In the following theorem, we will construct solutions of the system \eqref{eq:irr sys in z normal}-\eqref{eq:irr sys in t normal} from the solutions of the two limiting systems.

\begin{theorem}\label{thm:decom of two irr F}
   Fix $z_0,z_1...,z_{n-1},w_{n-1},\cdots w_{2},w_0$. Suppose $A(\mathbf{z},t,\mathbf{w}),G(\mathbf{z},t,\mathbf{w})$ in  system \eqref{eq:irr sys in z normal} -\eqref{eq:irr sys in t normal} are constructed from $\widehat{A}(\mathbf{z},\mathbf{w}),\widehat{G}(\mathbf{z},\mathbf{w})$ as in Theorem \ref{thm:asy of iso A from hatA}. The fundamental solutions $F_d^{(\infty)}(\xi,t)$ and $F_d^{(0)}(\xi,t)$ of the linear system \eqref{eq:irr sys in z normal} -\eqref{eq:irr sys in t normal} at $\xi=\infty$ and $\xi=0$ respectively (see Proposition \ref{prop:fundamental solution of two irr sys}), have the following factorizations:
    \begin{align}\label{eq:decom in infty in thm}
    &        \mathrm{(1)}\quad    F^{(\infty)}_{d}(\xi,t)= L(\xi,t)\cdot\mathrm{e}^{-\frac{w_0}{\xi}}Y^{(\infty)}_d(\xi).\\ \label{eq:decom in 0 in thm}
   & \mathrm{(2)}\quad
        F^{(0)}_{d}(\xi,t)=\mathrm{e}^{-\frac{w_0}{\xi}}Y^{(0)}(\xi)\cdot H\left(\xi,\frac{t}{\xi}\right)\cdot \left(\frac{t}{\xi}\right)^{\widehat{A}}\cdot \widehat{G}\cdot K^{(\infty)}_{-d+\mathrm{arg}(w_1)}\left(\frac{w_1}{\xi}\right).
    \end{align}
Here recall $w_1=\frac{t}{z_1\cdots z_{n-1}w_{n-1}\cdots w_2}=v_2-v_1$ and the factors are defined as follows:
\begin{itemize}
    \item 
 $Y^{(\infty)}_d(x)$ is the  solution of the following equation (with one irregular sigularity):
\begin{align}\label{eq:decom in hatA x}
\frac{\partial Y}{\partial x}&=\left( U +
\frac{\widehat{A}(\mathbf{z})}{x}\right)\cdot Y(x),
\end{align}
with the prescribed asymptotics
\begin{align}
    Y^{(\infty)}_d(x)x^{-\delta \widehat{A}}e^{-Ux}\sim \mathrm{Id}+O(x^{-1}), \ \ \text{as}\ x\rightarrow\infty\ \text{within}\ \mathrm{Sect}_{d}^{(\infty)};
\end{align}
    \item $Y^{(0)}(x)$ is the  solution of equation \eqref{eq:decom in hatA x} with the prescribed asymptotics
    \begin{align}\label{eq:asym sing irr Ahat in x}
        Y^{(0)}(x)x^{-\widehat{A}}\sim \mathrm{Id}+O(x),\ \ \text{as}\ x\rightarrow 0;
    \end{align}
   \item 
$L(\xi,t)$ is the solution of \eqref{eq:irr sys in t normal}, defined in $\mathcal{S}_{\epsilon,\phi} = \{ t \in \mathbb{C} \mid 0 < |t| < \epsilon, |\arg t| < \phi \}$ for a small enough $\epsilon$ and any finite $\phi$, satisfying
\begin{align}
    \lim_{t\rightarrow 0} L(\xi,t)=\lim_{\xi\rightarrow \infty} L(\xi,t)=\mathrm{Id};
\end{align}
\item    $K^{(\infty)}_d\left(\frac{w_1}{\xi}\right)$ is the solution of the following equation in the variable { $\gamma=\frac{w_1}{\xi}$,}
\begin{align}\label{eq:decom in tilA y}
\frac{\partial K}{\partial \gamma}&=\left( -\widetilde{V} +
\frac{\widetilde{A}(\mathbf{w})}{\gamma}\right)\cdot K(\gamma),\quad \text{where}\ \widetilde{V}:=\frac{1}{w_1}(V-w_0\mathrm{Id}),\ \widetilde{A}:=-\widehat{G}^{-1}\widehat{A}\widehat{G},
\end{align}
with the prescribed asymptotics
    \begin{align}\label{eq:asym of sing K w}
        K^{(\infty)}_d\left(\frac{w_1}{\xi}\right)\cdot e^{\frac{w_1}{\xi}\widetilde{V}}\cdot \left(\frac{w_1}{\xi}\right)^{-\delta\widetilde{A}}\sim \mathrm{Id}+O\left(\frac{w_1}{\xi}\right)^{-1},\ \ \text{as}\ \frac{w_1}{\xi}\rightarrow\infty\ \text{within}\ \mathrm{Sect}_{d}^{(\infty)};
    \end{align}

    \item 
Moreover, $H(x,y)$ is the solution for equation \eqref{eq:equiv equations of H}, defined in $\mathcal{S}'_{\epsilon,\phi}=\{(x,y)|0<|x|<\epsilon,0<|xy|<\epsilon,|\mathrm{arg}(x)|+|\mathrm{arg}(y)|<\phi\}$ for a small enough $\epsilon$ and a finite $\phi$, such that 
 \begin{align}\label{eq:lim of H1}
      \lim_{x\rightarrow0}H(x,y)=\mathrm{Id},\quad
      \lim_{x\rightarrow 0} \frac{\partial H}{\partial y}=0.
 \end{align}
\end{itemize}

\end{theorem}
\begin{proof}
$\mathrm{(1)}:$ Set
\begin{align}
    L(\xi,t) = \mathrm{Id}+\sum_{k=1}^{\infty}\int_0^t \mathrm{d}t_1\int_0^{t_1}\mathrm{d}t_2\cdots\int_0^{t_{k-1}}\mathrm{d}t_k\left(\frac{-B(t_1)+w_0\cdot\mathrm{Id}}{t_1\xi}\cdots\frac{-B(t_k)+w_0\cdot\mathrm{Id}}{t_k\xi}\right).
\end{align}
The integration is taken along the line segment joining $0$ and $t$ in $\mathcal{S}_{\epsilon,\phi}$. By Proposition \ref{thm:asym in t}, for a small enough $\epsilon$, there exists a constant $C$, and $0<\sigma<1$, such that
\begin{align*}
   \left| \frac{-B(t_1)+w_0\cdot\mathrm{Id}}{t_1\xi}\cdots\frac{-B(t_k)+w_0\cdot\mathrm{Id}}{t_k\xi}\right|\leq |t_1|^{-\sigma}\cdots|t_k|^{-\sigma}\frac{C^k}{\xi^k}.
\end{align*}
Thus $L(\xi,t)$ converges uniformly with respect to $\xi$ on every compact subset of $\{\xi\in\mathbb{C}:|\xi|>0\}$. By  direct calculation, $L(\xi,t)$ is the solution of \eqref{eq:irr sys in t normal}, and $  \lim_{t\rightarrow 0} L(\xi,t)=\lim_{\xi\rightarrow \infty} L(\xi,t)=\mathrm{Id}$.

Then denote the coefficients of \eqref{eq:irr sys in z normal} and \eqref{eq:irr sys in t normal} by $A_1,A_2$ respectively.
By compatibility of these two equations, we have $\left(\frac{\partial}{\partial t}-A_2\right)\left(\frac{\partial}{\partial \xi}-A_1\right)L=0$. Since $\left(\frac{\partial}{\partial t}-A_2\right)L=0$,  $\left(\frac{\partial}{\partial \xi}-A_1\right)L=LR(\xi)$ for a function $R(\xi)$ independent of $t$, and 
\begin{align}\label{eq:eq for Rxi}
    R(\xi)=L^{-1}\frac{\partial L}{\partial \xi}-L^{-1}\left(U+\frac{A}{\xi}+\frac{B}{\xi^2}\right)L.
\end{align}
 Let $t\rightarrow 0$ in \eqref{eq:eq for Rxi}. By Theorem \ref{thm:asy of iso A from hatA}, we have $R(\xi)=-\left(U+\frac{\widehat{A}}{\xi}+\frac{w_0\cdot\mathrm{Id}}{\xi^2}\right)$. Thus $L(\xi,t)\cdot \mathrm{e}^{-\frac{w_0}{\xi}} Y^{(\infty)}_d(\xi)$ is the solution of \eqref{eq:irr sys in z normal} and \eqref{eq:irr sys in t normal}.

Finally,  compare asymptotic behavior of $F^{(\infty)}_{d}(\xi,t)$ and  $L(\xi,t)\cdot\mathrm{e}^{-\frac{w_0}{\xi}} Y^{(\infty)}_d(\xi)$, as $\xi\rightarrow \infty$ within $\mathrm{Sect}^{(\infty)}_d$, we can see
\begin{align*}
    F^{(\infty)}_{d}(\xi,t)= L(\xi,t)\cdot\mathrm{e}^{-\frac{w_0}{\xi}}Y^{(\infty)}_d(\xi).
\end{align*}

This also shows that $\lim_{t\rightarrow 0} F_d^{(\infty)}(\xi,t)=\mathrm{e}^{-\frac{w_0}{\xi}}Y_d^{(\infty)}(\xi)$.

$\mathrm{(2)}:$
    Let 
    \begin{align*}
        x=\xi,\ y=\frac{t}{\xi}.
    \end{align*}
    Then system \eqref{eq:irr sys in z normal} and \eqref{eq:irr sys in t normal} becomes to 
    \begin{subequations}
    \begin{align}\label{eq:irr in sing x change vari}
        \frac{\partial F}{\partial x}&=\left(U+\frac{A}{x}+\frac{w_0\cdot \mathrm{Id}}{x^2}\right)F,\\ \label{eq:irr in w change vari}
   \frac{\partial  F}{\partial y}&=-\frac{B-w_0\cdot\mathrm{Id}}{xy} F.
    \end{align}
    \end{subequations}
    Similar to $(1)$, we  decompose $F_d^{(0)}(\xi,t)$ in the following three steps:

        $\mathrm{Step \ 1:}$ Let $Y^{(0)}(x)=\Psi(x)\cdot x^{\widehat{A}}$, where $\Psi(x)$ is a holomorphic function around $x=0$, 
 \begin{align*}
    \Psi(x)=\mathrm{Id}+\sum_{k=1}^{\infty}{\Psi}_kx^k,\quad |x|<\rho.
 \end{align*}
Suppose $H(x,y)=(\mathrm{e}^{-\frac{w_0}{x}}Y^{(0)}(x))^{-1}Q(x,y)$, then $Q$ is the solution of \eqref{eq:irr in sing x change vari} if and only if $H(x,y
)$ satisfies the following equation:
   \begin{align}\label{eq:equiv equations of H}
       \frac{\partial H}{\partial x}= \left(x^{-\widehat{A}}\Psi^{-1}(x)\frac{(A(xy)-\widehat{A})}{x}\Psi(x)x^{\widehat{A}}\right)H.
   \end{align}

For simplicity, denote the coefficient of \eqref{eq:equiv equations of H}: $x^{-\widehat{A}}\Psi^{-1}(x)\frac{(A(xy)-\widehat{A})}{x}\Psi(x)x^{\widehat{A}}$ by $P(x,y)$. Similar to $(1)$, set
\begin{align}
     H(x,y) = \mathrm{Id}+\sum_{k=1}^{\infty}\int_0^x \mathrm{d}x_1\int_0^{x_1}\mathrm{d}x_2\cdots\int_0^{x_{k-1}}\mathrm{d}x_k P(x_1,y)\cdots P(x_{k},y).
\end{align}
The integration is taken along the line segment joining $0$ and $x$ for any fixed $y$ in $\mathcal{S}'_{\epsilon,\phi}=\{(x,y)|0<|x|<\epsilon,0<|xy|<\epsilon,|\mathrm{arg}(x)|+|\mathrm{arg}(y)|<\phi\}$. Note that for a small enough $\epsilon$,
\begin{align}\nonumber
    \left|x^{-\widehat{A}}{(A(xy)-\widehat{A})}x^{\widehat{A}}\right|&=\left| \int_0^{xy}x^{-\widehat{A}}\left[\frac{U}{u_n-u_1},\tau^{\widehat{A}}\cdot{\widetilde{B}(\tau)}\cdot\tau^{-\widehat{A}}\right]x^{\widehat{A}}\mathrm{d}\tau  \right|
    =\left| \int_0^{xy}\left[x^{-\widehat{A}}\frac{U}{u_n-u_1}x^{\widehat{A}},(\frac{\tau}{x})^{\widehat{A}}{\widetilde{B}}(\frac{\tau}{x})^{-\widehat{A}}\right]\mathrm{d}\tau \right|\\ \label{eq:eqt for Axy}
    & \overset{\tilde{\tau}=\tau/x}{= }\ |x|\left| \int_0^{y}\left[x^{-\widehat{A}}\frac{U}{u_n-u_1}x^{\widehat{A}},\widetilde{\tau}^{\widehat{A}}{\widetilde{B}}\widetilde{\tau}^{-\widehat{A}}\right]\mathrm{d}\widetilde{\tau} \right|
    \leq C'_1|xy|^{1-\sigma_1}.
\end{align}
Here $\widetilde{B}(t)=t^{-\widehat{A}}\cdot \frac{(u_n-u_1)\left(G(t)VG(t)^{-1}-v_1\cdot \mathrm{Id}\right)}{t}\cdot t^{\widehat{A}}$, and the first equality is derived from Proposition \ref{thm:asym in t}, the last inequality is derived from  Lemma \ref{lem:esti by jimbo}.
Also note that
\begin{align}\label{eq:eqt for psi}
    x^{-\widehat{A}}\Psi(x)x^{\widehat{A}}=\mathrm{Id}+x\cdot x^{-\widehat{A}}\left(\sum_{k=0}^{\infty}\Psi_{k+1}\cdot x^{k}\right) x^{\widehat{A}}=\mathrm{Id}+O(x^{1-\sigma_1}).
\end{align}
Then by \eqref{eq:eqt for Axy} and \eqref{eq:eqt for psi},
\begin{align*}
  \left|  P(x,y)\right|&\leq\left|\left(x^{-\widehat{A}}\Psi(x)x^{\widehat{A}}\right)^{-1}\right|\left|\left(x^{-\widehat{A}}\frac{(A(xy)-\widehat{A})}{x}x^{\widehat{A}}\right)\right|\left|\left(x^{-\widehat{A}}\Psi(x)x^{\widehat{A}}\right)\right|\\ 
  &\leq C_1^{''}|x|^{-\sigma_1} |y|^{1-\sigma_1}.
\end{align*}
Thus we can see $H(x,y)$ converges uniformly with respect to any bounded $y$, and 
\begin{align*}
    \lim_{x\rightarrow0}H(x,y)=\mathrm{Id},\quad (x,y)\in \mathcal{S}'_{\epsilon,\phi}.
\end{align*}
Meanwhile, by Proposition \ref{thm:asym in t} 
\begin{align*}
    &\left|\left[x^{-\widehat{A}}Ux^{\widehat{A}},x^{-\widehat{A}}\frac{G(xy)VG(xy)^{-1}-v_1\cdot \mathrm{Id}}{xy}x^{\widehat{A}}\right]\right|\leq 2\left| x^{-\widehat{A}}Ux^{\widehat{A}}\right|\cdot\left|x^{-\widehat{A}}\frac{G(xy)VG(xy)^{-1}-v_1\cdot \mathrm{Id}}{xy}x^{\widehat{A}}\right|\\
    &=2\left| x^{-\widehat{A}}Ux^{\widehat{A}}\right|\cdot\left|y^{\widehat{A}}\cdot(xy)^{-\widehat{A}}\frac{G(xy)VG(xy)^{-1}-v_1\cdot \mathrm{Id}}{xy}(xy)^{\widehat{A}}\cdot y^{-\widehat{A}}\right|\leq C_2' |xy|^{-\sigma_1}.
\end{align*}
Then
\begin{align*}
    \left|\frac{\partial  P(x,y)}{\partial y}\right|\leq\left|\left(x^{-\widehat{A}}\Psi(x)x^{\widehat{A}}\right)^{-1}\right|\left|\left[x^{-\widehat{A}}Ux^{\widehat{A}},x^{-\widehat{A}}\frac{G(xy)VG(xy)^{-1}-w_0\cdot \mathrm{Id}}{xy}x^{\widehat{A}}\right]\right|\left|\left(x^{-\widehat{A}}\Psi(x)x^{\widehat{A}}\right)\right| \leq C''_2|xy|^{-\sigma_1}.
\end{align*}
So $\frac{\partial H}{\partial y}$ converges uniformly with respect to $y$  on every compact subset of $\{y\in\mathbb{C}:|y|>0\}$, and 
\begin{align*}
    \lim_{x\rightarrow 0} \frac{\partial H}{\partial y}=0,\quad (x,y)\in \mathcal{S}'_{\epsilon,\phi}.
\end{align*}

    $\mathrm{Step \ 2:}$  Denote the coefficients of \eqref{eq:irr in sing x change vari} and \eqref{eq:irr in w change vari} by $A_1$ and $A_2$ respectively. By the compatibility of these two equations, we have $\left(\frac{\partial}{\partial x}-A_1\right)\left(\frac{\partial}{\partial y}-A_2\right)Q=0$, where $Q(x,y)=\mathrm{e}^{-\frac{w_0}{x}}Y^{(0)}(x)H(x,y)$. Since $\left(\frac{\partial}{\partial x}-A_1\right)Q=0$, so $\left(\frac{\partial}{\partial y}-A_2\right)Q=QR(y)$ for a matrix function $R(y)$ independent of $x$, and 
    \begin{align}\nonumber
        R(y)& = Q^{-1}\frac{\partial Q}{\partial y} + Q^{-1}\frac{B-w_0\cdot\mathrm{Id}}{xy}Q\\ \label{eq:eq for Ry}
        &= H^{-1}\frac{\partial H}{\partial y}+ \left(x^{-\widehat{A}}Y^{(0)}(x)H(x,y)\right)^{-1}\left(x^{-\widehat{A}}\frac{B-w_0\cdot\mathrm{Id}}{xy}x^{\widehat{A}}\right)\left(x^{-\widehat{A}}Y^{(0)}(x)H(x,y)\right).
    \end{align}
  We compute $R(y)$ as $x\rightarrow 0$ in \eqref{eq:eq for Ry}. Using  \eqref{eq:lim of H1} which we have proved in Step 1, and Theorem \ref{thm:asym in t}, we have
  \begin{align*}
      R(y) = \lim_{x\rightarrow 0} x^{-\widehat{A}}\frac{B-w_0\cdot\mathrm{Id}}{xy}x^{\widehat{A}} = y^{\widehat{A}}\widehat{G}\frac{V-w_0\cdot\mathrm{Id}}{t}\widehat{G}^{-1}y^{-\widehat{A}}.
  \end{align*}
    Then we can verify that 
    \begin{align*}
        \frac{\partial}{\partial y}\left(y^{\widehat{A}}\widehat{G}K^{(\infty)}_{-d+\mathrm{arg}(w_1)}\left(\frac{w_1}{t}y\right)\right)=-R(y)\left(y^{\widehat{A}}\widehat{G}K^{(\infty)}_{-d+\mathrm{arg}(w_1)}\left(\frac{w_1}{t}y\right)\right).
    \end{align*}
    Here note that $\frac{w_1}{t}$ is a fixed constant.
Therefore, 
\begin{align*}
    \Xi(x,y):=\mathrm{e}^{-\frac{w_0}{x}}Y^{(0)}(x)H(x,y)y^{\widehat{A}}\widehat{G}K^{(\infty)}_{-d+\mathrm{arg}(w_1)}\left(\frac{w_1}{t}y\right)
\end{align*}
 is the solution of \eqref{eq:irr in sing x change vari} and \eqref{eq:irr in w change vari}. 

    $\mathrm{Step \ 3:}$ Now setting $x=\frac{t}{y}$ and $y$ as variables, it follows that 
      there is a constant matrix $C$, such as 
  \begin{align*}
    \Xi(\frac{t}{y},y)=F^{(0)}_d(\frac{t}{y},t)\cdot C. 
  \end{align*}
    By Step 1 and Step 2 of this Theorem, when we fix $y$, and let $t\rightarrow0$, we have
  \begin{align}\label{eq:asy in dec1}\nonumber
      t^{-\widehat{A}}\cdot\mathrm{e}^{\frac{w_0y}{t}} \Xi\left(\frac{t}{y},y\right)&=\left(t^{-\widehat{A}}\Psi(\frac{t}{y})t^{\widehat{A}}\right)\cdot y^{-\widehat{A}}H\left(\frac{t}{y},y\right)y^{\widehat{A}}\cdot \widehat{G}K^{(\infty)}_{-d+\mathrm{arg}(w_1)}\left(\frac{w_1}{t}y\right)\\
     &\rightarrow\widehat{G}K^{(\infty)}_{-d+\mathrm{arg}(w_1)}\left(\frac{y}{z_1\cdots z_{n-1}w_{n-1}\cdots w_2}\right).
  \end{align}
    
   On the other hand, suppose that $\widetilde{F}(\xi)$ satisfies the following equation:
     \begin{align}\label{eq for tilF for xi}
      \frac{\partial\widetilde{F}}{\partial\xi}=\left(-\widetilde{V}-\frac{w_1^{\delta\widetilde{A}}G^{-1}AGw_1^{-\delta\widetilde{A}}}{\xi}-\frac{w_1^{\delta\widetilde{A}}G^{-1}w_1UGw_1^{-\delta\widetilde{A}}}{\xi^2}\right)\widetilde{F}.
  \end{align}
   We can verify that $w_1^{\delta\widetilde{A}}G^{-1}\cdot \mathrm{e}^{\frac{w_0\xi}{w_1}}F\left(\frac{w_1}{\xi},t\right)$ also satisfies the above equation \eqref{eq for tilF for xi}, where $F\left(\frac{w_1}{\xi},t\right)$ is the solution of \eqref{eq:irr sys in z normal} and \eqref{eq:irr sys in t normal} with $\frac{w_1}{\xi}$ and $t$ as variables. Then comparing the asymptotic behavior between $F^{(0)}(\xi,t)$ and $\widetilde{F}^{(\infty)}(\frac{w_1}{\xi},t)$, as $\xi\rightarrow 0$ within $\mathrm{Sect}_d^{(0)}$, we have
  \begin{align*}
      \mathrm{e}^{\frac{w_0}{\xi}}F^{(0)}_{d}(\xi,t)=Gw_1^{-\delta\widetilde{A}}\widetilde{F}^{(\infty)}_{-d+\mathrm{arg}(w_1)}\left(\frac{w_1}{\xi},t\right).
  \end{align*}
  Thus by \eqref{eq:decom in infty in thm} of this Theorem which we have proved in $(1)$,  when we fix $y$, and let $t\rightarrow0$,
  \begin{align}\label{eq:asy in dec2}\nonumber
    t^{-\widehat{A}}\cdot \mathrm{e}^{\frac{w_0y}{t}} F^{(0)}_{d}(\frac{t}{y},t)
    &=t^{-\widehat{A}}Gw_1^{-\delta\widetilde{A}}\widetilde{F}^{(\infty)}_{-d+\mathrm{arg}(w_1)}\left(\frac{w_1}{t}y,t\right)\\
    &\rightarrow\widehat{G}K^{(\infty)}_{-d+\mathrm{arg}(w_1)}\left(\frac{y}{z_1\cdots z_{n-1}w_{n-1}\cdots w_2}\right).
  \end{align}
   Therefore, comparing \eqref{eq:asy in dec1} and \eqref{eq:asy in dec2}, we know the constant $C=\mathrm{Id}$, and the identity \eqref{eq:decom in 0 in thm} holds.
\end{proof}
\begin{cor}\label{cor:F decom in infty} Under the same conditions as in the Theorem \ref{thm:decom of two irr F}, we have:
\begin{align}\label{eq:decom of infty in cor}
    F^{(\infty)}_d(\xi,t)\cdot C_d(U,\widehat{A})=\mathrm{e}^{-\frac{w_0}{\xi}}Y^{(0)}(\xi)H\left(\xi,\frac{t}{\xi}\right)\left(\frac{t}{\xi}\right)^{\widehat{A}}\widehat{G}K^{(0)}\left(\frac{w_1}{\xi}\right)\cdot\left(z_1\cdots z_{n-1}w_{n-1}\cdots w_2\right)^{\widetilde{A}}\widehat{G}^{-1}.
\end{align}
Here recall that $C_d(U,\widehat{A})$ is the connection matrix for system \eqref{eq:decom in hatA x}, $Y^{(0)}(x),H(x,y)$ are defined in Theorem \ref{thm:decom of two irr F}. And $K^{(0)}(\frac{w_1}{\xi})$ is the solution of \eqref{eq:decom in tilA y} with the prescribed asymptotics
\begin{align*}
    K^{(0)}\left(\frac{w_1}{\xi}\right)\cdot\left(\frac{w_1}{\xi}\right)^{-\widetilde{A}}\sim \mathrm{Id}+O\left(\frac{w_1}{\xi}\right),\ \frac{w_1}{\xi}\rightarrow 0.
\end{align*}
\end{cor}
\begin{proof}
    By Theorem \ref{thm:decom of two irr F}, both sides in \eqref{eq:decom of infty in cor} are solutions of \eqref{eq:irr sys in z normal} and \eqref{eq:irr sys in t normal}, so we only need two compare the asymptotic behavior of two sides at $t\rightarrow0$.

    By \eqref{eq:decom in infty in thm}, $\lim_{t\rightarrow 0}\left( F^{(\infty)}_d(\xi,t)\cdot C_d(U,\widehat{A})\right)=\mathrm{e}^{-\frac{w_0}{\xi}}Y_d^{(\infty)}(\xi)\cdot C_d(U,\widehat{A})=\mathrm{e}^{-\frac{w_0}{\xi}}Y^{(0)}(\xi)$. By \eqref{eq:lim of H1}, $\lim_{t\rightarrow0} H(\xi,\frac{t}{\xi})=\mathrm{Id}$, and 
    \begin{align*}
        &\left(\frac{t}{\xi}\right)^{\widehat{A}}\widehat{G}K^{(0)}\left(\frac{w_1}{\xi}\right)\cdot\left(z_1\cdots z_{n-1}w_{n-1}\cdots w_2\right)^{\widetilde{A}}\widehat{G}^{-1}\\
        =& \widehat{G}\cdot\left(\frac{t}{\xi}\right)^{-\widetilde{A}}\left(\mathrm{Id}+\sum_{i=1}^{\infty}K_{0,i}\left(\frac{w_1}{\xi}\right)^{i}\right)\left(\frac{t}{\xi}\right)^{\widetilde{A}}\cdot\widehat{G}^{-1}\rightarrow \mathrm{Id}, \quad \text{ as } t\rightarrow 0.
    \end{align*}
    Thus the limit of the right side of \eqref{eq:decom of infty in cor} when $t\rightarrow0$ is also $\mathrm{e}^{-\frac{w_0}{\xi}}Y^{(0)}(\xi)$, the equality holds.
\end{proof}

Now we can use the decomposition of solutions to compute the monodromy matrices.

\begin{theorem}\label{prop:decomp of monodromy}
 Suppose $A(\mathbf{z},t,\mathbf{w}),G(\mathbf{z},t,\mathbf{w})$ in  system \eqref{eq:irr sys in z normal}-\eqref{eq:irr sys in w normal} can be constructed from $\widehat{A}(\mathbf{z},\mathbf{w}),\widehat{G}(\mathbf{z},\mathbf{w})$ as in Theorem \ref{thm:asy of iso A from hatA}, then we have
 
$\mathrm{(1)}$\  The Stokes matrices at $\xi=\infty$ of linear system \eqref{eq:irr sys in z normal} are equal to the ones of the system \eqref{eq:decom in hatA x}, i.e.
\begin{align}
    S_{d,\pm}^{(\infty)}(U,V,A,G)=S_{d}^{\pm}(U,\widehat{A}).
\end{align}

$\mathrm{(2)}$\ The Stokes matrices at $\xi=0$ of system \eqref{eq:irr sys in z normal} are equal to the ones of the system \eqref{eq:decom in tilA y}, i.e.
\begin{align}
    S_{d,\pm}^{(0)}(U,V,A,G)=S_{d+\mathrm{arg}(w_1)}^{\pm}(-\widetilde{V},\widetilde{A}).
\end{align}

$\mathrm{(3)}$ \ The connection matrix $C_d(U,V,A,G)$ of system \eqref{eq:irr sys in z normal} can be represented by a combination of connection matrices of \eqref{eq:decom in hatA x} and \eqref{eq:decom in tilA y} :
\begin{align}
    C_d(U,V,A,G) &=C_d(U,\widehat{A})\cdot  (z_1\cdots z_{n-1}w_{n-1}\cdots w_2)^{\widehat{A}}\widehat{G}\cdot C_{d+\mathrm{arg}(w_1)}(-\widetilde{V},\widetilde{A})^{-1}.
\end{align}

\end{theorem}
\begin{proof}
    $\mathrm{(1)}$:\ By \eqref{eq:decom in infty in thm} in Theorem \ref{thm:decom of two irr F}, 
    \begin{align*}
         S_{d,\pm}^{(\infty)}(U,V,A,G)&=F^{(\infty)}_{d\pm\pi}(\xi,t)^{-1}F^{(\infty)}_d(\xi,t) =Y^{(\infty)}_{d\pm\pi}(\xi)^{-1}Y^{(\infty)}_d(\xi)=S_{d}^{\pm}(U,\widehat{A}).
    \end{align*}
$\mathrm{(2)}$:\ By \eqref{eq:decom in 0 in thm} in Theorem \ref{thm:decom of two irr F}, 
\begin{align*}
    S_{d,\pm}^{(0)}(U,V,A,G)&=F^{(0)}_{-d\mp\pi}(\xi,t)^{-1}F^{(0)}_{-d}(\xi,t)=K^{(\infty)}_{d\pm\pi+\mathrm{arg}(w_1)}\left(\frac{w_1}{\xi}\right)^{-1}K^{(\infty)}_{d+\mathrm{arg}(w_1)}\left(\frac{w_1}{\xi}\right)=S_{d+\mathrm{arg}(w_1)}^{\pm}(-\widetilde{V},\widetilde{A}).
\end{align*}
   $\mathrm{(3)}$:\ Compare \eqref{eq:decom in 0 in thm} in Theorem \ref{thm:decom of two irr F} and \eqref{eq:decom of infty in cor} in Corollary \ref{eq:decom of infty in cor} , we have
   \begin{align*}
       C_d(U,V,A,G)&=(F_d^{(\infty)})^{-1}F_{-d}^{(0)}\\
       &=C_d(U,\widehat{A})\widehat{G}\cdot(z_1\cdots z_{n-1}w_{n-1}\cdots w_2)^{-\widetilde{A}}K^{(0)}\left(\frac{w_1}{\xi}\right)^{-1}K^{(\infty)}_{d+\mathrm{arg}(w_1)}\left(\frac{w_1}{\xi}\right)\\
       &=C_d(U,\widehat{A})\widehat{G}\cdot(z_1\cdots z_{n-1}w_{n-1}\cdots w_2)^{-\widetilde{A}}C_{d+\mathrm{arg}(w_1)}(-\widetilde{V},\widetilde{A})^{-1}.
   \end{align*}

\end{proof}

\subsection{The asymptotics of shrinking solutions of isomonodromy equations}\label{sec: asy of iso}

In the preceding subsection, we consistently assumed that the solutions $A(\mathbf{z},t, \mathbf{w})$ and $G(\mathbf{z},t, \mathbf{w})$ of the isomonodromy equations were constructed from prescribed asymptotic data $\widehat{A}(\mathbf{z}, \mathbf{w})$ and $\widehat{G}(\mathbf{z}, \mathbf{w})$. In this subsection, we extend this analysis by further investigating the asymptotic behavior of $\widehat{A}$ and $\widehat{G}$.
\begin{prop}\label{prop:from A0,G0 to hatA, hatG}
Given constant matrices $\widehat{A}_0, \widetilde{A}_0 \in \frak{gl}_n$ and $G_0 \in {\rm GL}_n$, such that $\widetilde{A}_0 = -G_0^{-1} \widehat{A}_0 G_0$, and the eigenvalues $\widehat{\lambda}_i^{(k)}$ and $\widetilde{\lambda}_i^{(k)}$ of their upper-left $k \times k$ submatrices satisfy
   \begin{align}\label{eq:eigen condition of Ahat and Atil}
        \left|\mathrm{Re}\left(\widehat{\lambda}^{(k)}_{i}-\widehat{\lambda}^{(k)}_{j}\right)\right|<1,\quad \left|\mathrm{Re}\left(\widetilde{\lambda}^{(k)}_{i}-\widetilde{\lambda}^{(k)}_{j}\right)\right|<1,\   k=1,2,...,n;\ 1\leq i,j\leq k.
\end{align}
Let $\widehat{A}\left(\mathbf{z},\widehat{A}_0\right)$ and $\widetilde{A}\left(\mathbf{w};\widetilde{A}_0\right)$ be the solutions of \eqref{eq iso for hatA z} and \eqref{eq iso for tilA w} respectively, as provided by Theorem~\ref{thm:iso asy by tangxu} with boundary value $\widehat{A}_0$ and $\widetilde{A}_0$. Then the system  \eqref{eq for hatG z} and \eqref{eq for hatG w} has a unique solution $\widehat{G}(\mathbf{z},\mathbf{w})$ such that
$
\widetilde{A}\left(\mathbf{w},\widetilde{A}_0\right) = -\widehat{G}\left(\mathbf{z},\mathbf{w}\right)^{-1}\widehat{A}\left(\mathbf{z},\widehat{A}_0\right)\widehat{G}\left(\mathbf{z},\mathbf{w}\right),
$
and as
$
z_{n-1}\rightarrow\infty,\ldots,z_2\rightarrow\infty;\ w_{n-1}\rightarrow\infty,\ldots,w_2 \rightarrow \infty
$  successively, 
   \begin{align}\label{eq:asym leading of hatG}
   \lim_{w_2\rightarrow\infty}\cdots\lim_{w_{n-1}\rightarrow\infty}\lim_{z_{2}\rightarrow\infty}\cdots\lim_{z_{n-1}\rightarrow\infty} \overrightarrow{\prod_{k=1}^{n-1}}\left(z_{k}^{\delta_{k}\widehat{A}_{k-1}}z_{k}^{-\widehat{A}_{k}}\right)\left(z_1\cdots z_{n-1}w_{n-1}\cdots w_2\right)^{\widehat{A}}\widehat{G} \overleftarrow{\prod_{k=2}^{n-1}}\left(w_{k}^{\widetilde{A}_{k}}w_{k}^{-\delta_{k}\widetilde{A}_{k-1}}\right)= G_0.
   \end{align}
\end{prop}
\begin{proof}
Since the equations \eqref{eq iso for hatA z} in the variables $\mathbf{z} = (z_1, \dots, z_{n-1})$ coincide with the isomonodromy equations \eqref{eq:iso eq of one irr in z}, Theorem~\ref{thm:iso asy by tangxu} guarantees the existence of a unique solution $\widehat{A}_{n-1} = \widehat{A}\left(\mathbf{z}, \widehat{A}_0\right)$, characterized by the sequence of asymptotic data $\widehat{A}_{n-2}, \dots, \widehat{A}_{1}, \widehat{A}_{0}$. Similarly, for $\widetilde{A}$ satisfying \eqref{eq iso for tilA w}, the same theorem ensures that a unique solution $\widetilde{A}\left(\mathbf{w}, \widetilde{A}_0\right) = \widetilde{A}_{n-1}$ exists with the corresponding asymptotic data $\widetilde{A}_{n-2}, \dots, \widetilde{A}_{1}=\widetilde{A}_0,\  \widetilde{A}_{0}$.

To construct $\widehat{G}$, we require the following lemma concerning the connection matrix $C_d(U,\Phi)$ of the system with a single irregular singularity:

\begin{lemma}[{\cite[Section 4.2]{TangXu}}]\label{prop:eq of one irr connection mat}
Let $\Phi_{n-1}=\Phi(\mathbf{z};\Phi_0)$ be a non-resonant solution of the isomonodromy equations \eqref{eq:iso eq of one irr in z}, with a sequence of asymptotics $\Phi_{n-2},...,\Phi_{0}$ as defined in Theorem \ref{thm:iso asy by tangxu}. Let $X(\mathbf{z},\Phi_{n-1})$ be a meromorphic multi-valued solution of system  \eqref{eq: eq of one irr connection mat} where $\Phi=\Phi_{n-1}$.
    Then, as $z_{n-1} \rightarrow \infty, \dots, z_{2} \rightarrow \infty$ successively, there is a constant $X_0$, such that
    \begin{align}\label{eq:leading term of connection mat in one irr}
       \lim_{z_{2}\rightarrow\infty}\cdots\lim_{z_{n-1}\rightarrow\infty} X (\mathbf{z},\Phi_{n-1})\cdot \overleftarrow{\prod_{k=1}^{n-1}} (z_{k}^{\Phi_{k}} z_{k}^{-\delta_{k}\Phi_{k-1}})=X_0.
    \end{align}
    In particular, when $X(\mathbf{z})=C_d(U,\Phi_{n-1})$ for $U$ in a given connected region $R_{u,d}(J)$, the asymptotic constant $X_0$ is 
    \begin{align*}
        X_0=\overrightarrow{\prod_{k=1}^{n-1}} C_{d+\mathrm{arg}(u_{k+1}-u_k)}(E_{k+1}, \delta_{k+1}(\Phi_{0})), \quad \text{for}\ U\in R_{u,d}(J).
    \end{align*}
\end{lemma}

Returning to the proof of the proposition.
Let $\widehat{\mathfrak{C}}(\mathbf{z},\widehat{A}_{n-1})$ be the solutions of \eqref{eq: eq of one irr connection mat} corresponding to $\Phi=\widehat{A}(\mathbf{z},\widehat{A}_0)$, and $\widetilde{\mathfrak{C}}(\mathbf{w},\widetilde{A}_{n-1})$ be the solutions of \eqref{eq: eq of one irr connection mat} corresponding to $U=\widetilde{V},\Phi=\widetilde{A}(\mathbf{w},\widetilde{A}_0)$, such that both asymptotic constants in \eqref{eq:leading term of connection mat in one irr} are $\mathrm{Id}$. By Lemma \ref{prop:eq of one irr connection mat},
in any region $R_{u,d}(J)$, 
    \[\widehat{\mathfrak{C}}(\mathbf{z},\widehat{A}_{n-1})=\left( \overrightarrow{\prod_{k=1}^{n-1}} C_{d+\mathrm{arg}(u_{k+1}-u_k)}(E_{k+1}, \delta_{k+1}(\widehat{A}_{0})) \right)^{-1}{C}_d(U, \widehat{A}_{n-1}),\]
and in any region $R_{\widetilde{v},d}(J)$,
  \[\widetilde{\mathfrak{C}}(\mathbf{w},\widetilde{A}_{n-1})=\left( \overrightarrow{\prod_{k=1}^{n-1}} C_{d+\mathrm{arg}(\widetilde{v}_{k+1}-\widetilde{v}_k)}(E_{k+1}, \delta_{k+1}(\widetilde{A}_{0})) \right)^{-1}{C}_d(\widetilde{V}, \widetilde{A}_{n-1}),\quad \text{where}\  \widetilde{v}_k=\frac{v_k-w_0}{w_1}.
  \]
  Thus by Corollary \ref{cor:decom of CPhiC in one irr}, $\mathrm{Ad}\left(\widehat{\mathfrak{C}}(\mathbf{z},\widehat{A}_{n-1})\right)\widehat{A}=\widehat{A}_0,\ \mathrm{Ad}\left(\widetilde{\mathfrak{C}}(\mathbf{w},\widetilde{A}_{n-1})\right)\widetilde{A}=\widetilde{A}_0 $.

Now set
\begin{align}\label{eq:def of hatG in RH prob}
    \widehat{G} =\mathrm{e}^{\pi\mathrm{i}\widehat{A}} (z_1 \cdots z_{n-1} w_{n-1} \cdots w_2)^{-\widehat{A}} \cdot \widehat{\mathfrak{C}}(\mathbf{z},\widehat{A}_{n-1})^{-1}\cdot G_0\mathrm{e}^{\pi\mathrm{i}\widetilde{A}_0}\cdot \widetilde{\mathfrak{C}}(\mathbf{w},\widetilde{A}_{n-1}).
\end{align}
Note that
\begin{align*}
    \widehat{G}^{-1} \widehat{A} \widehat{G} &= \mathrm{Ad}\left(\widetilde{\mathfrak{C}}(\mathbf{w},\widetilde{A}_{n-1})^{-1}\cdot \mathrm{e}^{-\pi\mathrm{i}\widetilde{A}_0}G_0^{-1}\cdot\widehat{\mathfrak{C}}(\mathbf{z},\widehat{A}_{n-1})\right) \widehat{A} \\
    &= -\mathrm{Ad}\left(\widetilde{\mathfrak{C}}(\mathbf{w},\widetilde{}{A}_{n-1})^{-1}\right) \widetilde{A}_0 \\
    &= -\widetilde{A}.
\end{align*}
With the relation $\widetilde{A} = -\widehat{G}^{-1} \widehat{A} \widehat{G}$ established, the equations and asymptotic behavior of $\widehat{G}$ can be computed directly from the corresponding properties of $\widehat{\mathfrak{C}}(\mathbf{z},\widehat{A}_{n-1})$ and $\widetilde{\mathfrak{C}}(\mathbf{w},\widetilde{A}_{n-1})$. In particular, we can use \eqref{eq:def of hatG in RH prob} to compute the equations for $\widehat{G}$ in variables $w_0,w_2,...,w_{n-1}$, and use the following which is equivalent to \eqref{eq:def of hatG in RH prob} to compute the equations  in variables $z_0,z_1,...,z_{n-1}$:
\begin{align*}
    \widehat{G} =   \widehat{\mathfrak{C}}(\mathbf{z},\widehat{A}_{n-1})^{-1}\cdot G_0\mathrm{e}^{\pi\mathrm{i}\widetilde{A}_0}\cdot \widetilde{\mathfrak{C}}(\mathbf{w},\widetilde{A}_{n-1})\cdot(z_1 \cdots z_{n-1} w_{n-1} \cdots w_2)^{\widetilde{A}}\mathrm{e}^{-\pi\mathrm{i}\widetilde{A}}.
\end{align*}

To prove the uniqueness of $\widehat{G}$, it suffices to show that any $\widehat{G}$ satisfying \eqref{eq:eigen condition of Ahat and Atil} and \eqref{eq:asym leading of hatG} must be given by \eqref{eq:def of hatG in RH prob}. Based on Theorem~\ref{thm:asy of iso A from hatA}, there exist isomonodromy solutions $(A, G)$ to the equations \eqref{iso for z begin}--\eqref{iso for z end} with leading asymptotics $\widehat{A}$ and $\widehat{G}$, respectively. Theorem~\ref{prop:decomp of monodromy} then shows that
\begin{align}\nonumber
    C_d(U, V, A, G) &= C_d(U, \widehat{A}) \cdot (z_1 \cdots z_{n-1} w_{n-1} \cdots w_2)^{\widehat{A}} \widehat{G} \cdot C_{d+\mathrm{arg}(w_1)}(-\widetilde{V}, \widetilde{A})^{-1}\\ \label{eq:Cd(u,a,v) in proof}
    &= C_d(U, \widehat{A}) \cdot (z_1 \cdots z_{n-1} w_{n-1} \cdots w_2)^{\widehat{A}} \widehat{G} \cdot \mathrm{e}^{\pi\mathrm{i}\widetilde{A}}C_{d+\pi+\mathrm{arg}(w_1)}(\widetilde{V}, \widetilde{A})^{-1} \mathrm{e}^{-\pi\mathrm{i}\delta\widetilde{A}}.
\end{align}
Combined with the asymptotics \eqref{eq:leading term of connection mat in one irr}, \eqref{eq:asym leading of hatG} and \eqref{eq:limit in zk one irr iso}, in a connected region $R_{u,d}(J_1)\times R_{v,d}(J_2)$, the constant $C_d(U,V,A,G)$ is 
\begin{align} \nonumber
    C_d =& C_d(U, \widehat{A}) \overleftarrow{\prod_{k=1}^{n-1}} \left(z_{k}^{\widehat{A}_{k}} z_{k}^{-\delta_{k} \widehat{A}_{k-1}}\right)  \\ \nonumber
    & \cdot \overrightarrow{\prod_{k=1}^{n-1}} \left(z_{k}^{\delta_{k} \widehat{A}_{k-1}} z_{k}^{-\widehat{A}_{k}}\right) (z_1 \cdots z_{n-1} w_{n-1} \cdots w_2)^{\widehat{A}} \widehat{G} \overleftarrow{\prod_{k=2}^{n-1}} \left(w_{k}^{\widetilde{A}_{k}} w_{k}^{-\delta_{k} \widetilde{A}_{k-1}}\right) \\ \nonumber
    & \cdot \overrightarrow{\prod_{k=2}^{n-1}} \left(w_{k}^{\delta_{k} \widetilde{A}_{k-1}} w_{k}^{-\widetilde{A}_{k}}\right) \mathrm{e}^{\pi\mathrm{i}\widetilde{A}} \overleftarrow{\prod_{k=2}^{n-1}} \left(w_{k}^{\widetilde{A}_{k}} w_{k}^{-\delta_{k} \widetilde{A}_{k-1}}\right) \overrightarrow{\prod_{k=2}^{n-1}} \left(w_{k}^{\delta_{k} \widetilde{A}_{k-1}} w_{k}^{-\widetilde{A}_{k}}\right) C_{d+\pi+\mathrm{arg}(w_1)}(\widetilde{V}, \widetilde{A})^{-1} \mathrm{e}^{-\pi\mathrm{i}\delta\widetilde{A}} \\  
  \label{eq:exp for Cd in prop}  =& \left( \overrightarrow{\prod_{k = 1}^{n-1}} C_{d+\mathrm{arg}(u_{k+1}-u_k)}\left(E_{k+1}, \delta_{k+1}\widehat{A}_0\right) \right) \cdot G_0  \mathrm{e}^{\pi\mathrm{i}\widetilde{A}_0}\cdot\left( \overrightarrow{\prod_{k = 1}^{n-1}} C_{d+\pi+\mathrm{arg}(v_{k+1}-v_k)}\left(E_{k+1}, \delta_{k+1}\widetilde{A}_0\right) \right)^{-1}\mathrm{e}^{-\pi\mathrm{i}\delta\widetilde{A}}. 
\end{align}
 Substituting this back into \eqref{eq:Cd(u,a,v) in proof} recovers the desired expression for $\widehat{G}$.
\end{proof}

The matrices $(\widehat{A},\widehat{G})$ obtained from Proposition \ref{prop:from A0,G0 to hatA, hatG} satisfy the eigenvalue condition \eqref{eq:shrinking in t}  and the equations \eqref{eq iso for hatA z}-\eqref{eq for hatG w}. Therefore, by Theorem \ref{thm:asy of iso A from hatA}, they yield solutions to the isomonodromy equation that satisfy the prescribed asymptotic behavior. These solutions are precisely the ones parameterized by the boundary value $(\widehat{A}_0,G_0)$ in Theorem \ref{mainthm}, and we refer to them as shrinking solutions.

\begin{definition}\label{def:shrinking solution}
    Given $\widehat{A}_0,G_0$  satisfying the boundary conditions \eqref{eq:eigenvalue cond of hatA tilA}, the solution $A\left(\widehat{A}_0,G_0\right)$ and $G\left(\widehat{A}_0,G_0\right)$ of the isomonodromy equations constructed via Proposition \ref{prop:from A0,G0 to hatA, hatG} and Theorem \ref{thm:asy of iso A from hatA} are called the \textbf{shrinking solutions}. The set of all shrinking solutions is denoted by $\mathfrak{Sol}_{Shr}$.

\end{definition}
We also compute the concrete expressions of monodromy data.
\begin{cor}\label{cor: concrete monodromy data by A0 and G0}
   Let $A\left(\widehat{A}_0,G_0\right),G\left(\widehat{A}_0,G_0\right)$ be the shrinking solution of isomonodromy equations \eqref{iso for z begin}–\eqref{iso for z end}. Then when $U\in R_{u,d}$, such that $\mathrm{Im}(u_1e^{\mathrm{i}d})>...> \mathrm{Im}(u_ne^{\mathrm{i}d})$, $-\pi<\left(d+\mathrm{arg}(u_{k+1}-u_{k})\right)<0$,
    \begin{align}
        \nu_d^{(\infty)}\left(U,V,A\left(\widehat{A}_0,G_0\right),G\left(\widehat{A}_0,G_0\right)\right)= \mathrm{Ad}\left(\overrightarrow{\prod_{k=1}^{n-1}}C_{-\frac{\pi}{2}}\left(E_{k+1},\delta_{k+1}\widehat{A}_0\right)\right)\mathrm{e}^{2\pi\mathrm{i}\widehat{A}_0}.
    \end{align}
    When $V\in R_{v,d}$, such that $\mathrm{Im}(v_1e^{\mathrm{i}d})<...< \mathrm{Im}(v_ne^{\mathrm{i}d}),\  0<\left(d+\mathrm{arg}(v_{k+1}-v_{k})\right)<\pi$,
    \begin{align}
        \nu_{-d}^{(0)}\left(U,V,A\left(\widehat{A}_0,G_0\right),G\left(\widehat{A}_0,G_0\right)\right)= \mathrm{Ad}\left(\overrightarrow{\prod_{k=1}^{n-1}}C_{\frac{\pi}{2}}\left(-E_{k+1},\delta_{k+1}\widetilde{A}_0\right)\right)\mathrm{e}^{2\pi\mathrm{i}\widetilde{A}_0}.
    \end{align}
    And when \begin{align}\nonumber
    &\mathrm{Im}(u_1e^{\mathrm{i}d})>...> \mathrm{Im}(u_ne^{\mathrm{i}d}),\quad \mathrm{Im}(v_1e^{\mathrm{i}d})<...< \mathrm{Im}(v_ne^{\mathrm{i}d}),\\ \label{eq:special case for uv}
    &-\pi<\left(d+\mathrm{arg}(u_{k+1}-u_{k})\right)<0,\quad 0<\left(d+\mathrm{arg}(v_{k+1}-v_{k})\right)<\pi,\ \text{for}\ \ k=1,..,n-1,
\end{align}
we have
    \begin{align}\label{eq:eq for connection mat}
    C_d\left(U,V,A\left(\widehat{A}_0,G_0\right),G\left(\widehat{A}_0,G_0\right)\right) 
        =\left( \overrightarrow{\prod_{k = 1}^{n-1}} C_{-\frac{\pi}{2}}\left(E_{k+1},\delta_{k+1}\widehat{A}_0\right) \right)\cdot {G}_0  \cdot \left( \overrightarrow{\prod_{k = 1}^{n-1}} C_{\frac{\pi}{2}}\left(-E_{k+1},\delta_{k+1}\widetilde{A}_0\right) \right)^{-1},
    \end{align}
\end{cor}
\begin{proof}
    By Theorem \ref{prop:decomp of monodromy} $(2)$, \eqref{eq:LU fac of nu 0} in Lemma \ref{prop:identity of monodromy factor for two irr sys}, and \eqref{eq:trans in stokes mat with c} in Lemma \ref{prop:monodromy factor for one irr sys}, we know 
    \begin{align*}
        \nu_d^{(\infty)}\left(U,V,A\left(\widehat{A}_0,G_0\right),G\left(\widehat{A}_0,G_0\right)\right)&=\nu_d\left(U,\widehat{A}\left(\widehat{A}_0\right)\right),\\
        \nu_{-d}^{(0)}\left(U,V,A\left(\widehat{A}_0,G_0\right),G\left(\widehat{A}_0,G_0\right)\right)&=\nu_{d+\mathrm{arg}(w_1)}\left(-\widetilde{V},\widetilde{A}\left(\widetilde{A}_0\right)\right),
    \end{align*}
    where $\widehat{A}\left(\widehat{A}_0\right)$, $\widetilde{A}\left(\widetilde{A}_0\right)$ represent the solutions  of \eqref{eq iso for hatA z} and \eqref{eq iso for tilA w}, constructed from $\widehat{A}_0$ and $\widetilde{A}_0$ respectively.
    Then by the Theorem \ref{thm:monodromy factor of one irr sys}, we obtain the expressions of $\nu_d^{(\infty)}$ and $\nu_{-d}^{(0)}$. 
     For the special conditions \eqref{eq:special case for uv}, when $\mathrm{Im}(u_1e^{\mathrm{i}d})>...> \mathrm{Im}(u_ne^{\mathrm{i}d}),\ 0<\left(d+\mathrm{arg}(u_{k}-u_{k+1})\right)<\pi$, all the directions $d+\mathrm{arg}(u_{k+1}-u_{k})\in (-\pi,0)$. And since there are no anti-Stokes directions between $-\pi$ and $0$ for $C_{d+\mathrm{arg}(u_{k+1}-u_{k})}(E_{k+1},\delta_{k+1}\widehat{A}_0)$, we can choose the direction $-\pi/2$ to represent $d+\mathrm{arg}(u_{k+1}-u_{k})$. Analogously, we choose the direction $\pi/2$ to represent $d+\mathrm{arg}(v_{k+1}-v_{k})$.

     In the proof of  Proposition \ref{prop:from A0,G0 to hatA, hatG}, we have obtained the expressions of $C_d$ in \eqref{eq:exp for Cd in prop}. Then using the following formula, and choosing the argument, we obtain the form in \eqref{eq:eq for connection mat}:
     \begin{align}\label{eq:simp for Cd}
         C_{d+\mathrm{arg}(v_{k+1}-v_k)}\left(-E_{k+1},\delta_{k+1}\widetilde{A}_0\right)=\mathrm{e}^{\pi\mathrm{i}\delta_k\widetilde{A}_0} C_{d+\pi+\mathrm{arg}(v_{k+1}-v_k)}\left(E_{k+1},\delta_{k+1}\widetilde{A}_0\right)\mathrm{e}^{-\pi\mathrm{i}\delta_{k+1}\widetilde{A}_0}.
     \end{align}
\end{proof}

\begin{proof}[Proof of Theorem~\ref{thm: introcatformula}]
The expressions of $C_{-\frac{\pi}{2}}\left(E_{k+1},\delta_{k+1}\widehat{A}_0\right)$ and $C_{\frac{\pi}{2}}\left(-E_{k+1},\delta_{k+1}\widetilde{A}_0\right)$ are known results, see \cite{Balser-Jurkat-Lutz1981} or \cite[Section 3.6]{xu2019closure1}, \cite[Appendix A]{TangXu}. Thus from \eqref{eq:eq for connection mat}, the computation of $C_d(U,V,A,G)$ reduces to a direct substitution. For the computation of Stokes matrices, by Lemma \ref{lem:triang stokes mat} and Lemma \ref{prop:identity of monodromy factor for two irr sys}, we just need to compute the LU decomposition of $\nu_d^{(\infty)}$ and $\nu_d^{(0)}$. For a detailed procedure of this decomposition, see  \cite[Section 3.6]{xu2019closure1}, \cite[Appendix A]{TangXu}.
\end{proof}

\subsection{Almost every solution of isomonodromy equations is shrinking solution}\label{sec:almost every is shrinking}

Recall that we have defined the monodromy matrices $\nu^{(\infty)}_d(U,V,A,G)$ and $\nu^{(0)}_d(U,V,A,G)$ for system \eqref{eq: linr sys two irr in xi} in Definition \ref{def: mono mat}. In this section, we present a sufficient condition on the monodromy data to identify shrinking solutions of the isomonodromy equations. Consequently, it allows us to demonstrate that almost every solution is shrinking solution.

\begin{definition}\label{def:generic sol}
    Given a diagonal matrix $\Lambda=\mathrm{diag}(\phi_{11},...,\phi_{nn})$, we define the set {$\mathcal{M}(\Lambda)\subset \mathbb {\rm Mat}_{n\times n}(\mathbb{C})$} as consisting of all matrices $M$, such that 
       for every $1\leq k\leq n$, there is a  set $\sigma_k=\{\lambda_j^{(k)}:{j=1,...,k}\}$, satisfying 
    \begin{align}\label{eq:generic cond 1 in def}
        &\sigma(M^{[k]})=\{e^{2\pi\mathrm{i}\lambda}:\lambda\in\sigma_k\},\quad
      \sum_{j=1}^k \lambda_{j}^{(k)}=\sum_{j=1}^k\phi_{jj},\\ \label{eq:generic cond 2 in def}
    &\left|\mathrm{Re}\left(\lambda_{j_1}^{(k)}-\lambda_{j_2}^{(k)}\right)\right|< 1,\ \text{ for every $1\leq j_1,j_2\leq k$},\\ \label{eq:generic cond 3 in def}
    & \lambda_{j_1}^{(k+1)}-\lambda_{j_2}^{(k)}\notin\mathbb{Z}\setminus\{0\},\ \text{ for every $1\leq j_1\leq k+1$,\ $1\leq j_2\leq k$}.
    \end{align}
    Here $\sigma(M^{[k]})$ denotes the spectrum of upper-left $k\times k$ submatrix of $M$. 
\end{definition}

\begin{definition}\label{def:log-conf cond}
The monodromy data of the system \eqref{introeq} 
 \[\left(\delta A,\delta(G^{-1}AG),\nu_d^{(\infty)}(U,V,A,G),\nu_{-d}^{(0)}(U,V,A,G), C_d(U,V,A,G)\right)\in\left(\mathrm{Diag}_n(\mathbb{C})\right)^2\times\left( \mathrm{Mat}_{n\times n}(\mathbb{C})\right)^3, \]
is called {strictly (upper-left) log-confined}, if for $U, V$, and a direction $d$, such that
\begin{align}\label{eq:u,v cond for monodromy mat}
     &\mathrm{Im}(u_1e^{\mathrm{i}d})>...> \mathrm{Im}(u_ne^{\mathrm{i}d}),\quad \mathrm{Im}(v_1e^{\mathrm{i}d})<...< \mathrm{Im}(v_ne^{\mathrm{i}d}),
\end{align} 
we have 
\[\nu_d^{(\infty)}\in \mathcal{M}(\delta A)  \ \text{ and } \ \nu_{-d}^{(0)}\in \mathcal{M}(\delta (G^{-1}AG)).\]
\end{definition}
\begin{remark}
As stated in Section \ref{sec:iso property for 2 irr sys}, the monodromy data is constant on each connected component of $R_{u,d}\times R_{v,d}$. The { strictly log-confined} test is for the monodromy data on the connected component corresponding to \eqref{eq:u,v cond for monodromy mat}. 
\end{remark}

In the remainder of this section, we will always assume that the monodromy data of the system \eqref{introeq} is taken under the assumption \eqref{eq:u,v cond for monodromy mat}. The following lemma shows that the {strictly (upper-left) log-confined} condition is generic.
\begin{lemma}[\cite{TangXu}]\label{Lem: RH corr 1}
    Consider the monodromy data for system \eqref{introeq}.
    Let $X=\nu_d^{(\infty)}(U,V,A,G)$ and $Y=\nu_{-d}^{(0)}(U,V,A,G)$. Then for every $1\leq k\leq n$, there is a  set $\sigma_k=\{\lambda_j^{(k)}:{j=1,...,k}\}$, such that 
    \begin{align}\label{eq:generic cond 1 in lem}
        &\sigma(X^{[k]})=\{\mathrm{e}^{2\pi\mathrm{i}\lambda}:\lambda\in\sigma_k\},\quad
      \sum_{j=1}^k \lambda_{j}^{(k)}=\mathrm{trace}(A^{[k]}),\\ \label{eq:generic cond 2 in lem}
    &\left|\mathrm{Re}\left(\lambda_{j_1}^{(k)}-\lambda_{j_2}^{(k)}\right)\right|\leq 1, \text{ for every $1\leq j_1,j_2\leq k$}.
    \end{align}
    Furthermore, if the sequence $\sigma_k$ is such that 
\begin{align}\label{eq:generic condition 1}
    \left|\mathrm{Re}\left(\lambda_{j_1}^{(k)}-\lambda_{j_2}^{(k)}\right)\right|< 1, \ \text{ for every $j_1,j_2$},
\end{align}
then the choice of set $\sigma_k$ is {unique}.

 Similarly, we have the {same result for $Y$ provided replacing $\mathrm{trace}(A^{[k]})$ by $-\mathrm{trace}((G^{-1}AG)^{[k]})$} in \eqref{eq:generic cond 1 in lem}.
\end{lemma}
\begin{proof}
   Suppose $\{\mu_j^{(k)}:{j=1,...,k}\}$ are the eigenvalues of $X^{[k]}$. By assumption \eqref{eq:u,v cond for monodromy mat}, and then Lemma \ref{lem:triang stokes mat} and  \eqref{eq:LU fac of nu infty} in Lemma \ref{prop:identity of monodromy factor for two irr sys}, we have 
    \begin{align*}
        e^{2\pi\mathrm{i}\left(\mathrm{trace}(A^{[k]})\right)}=\mathrm{det}(X^{[k]})=\prod_{j=1}^k\mu_j^{(k)}.
    \end{align*}
    Thus we can find $\lambda_j^{(k)},\ j=1,...,k$, such that $\mathrm{e}^{2\pi\mathrm{i}\lambda_j^{(k)}}=\mu_j^{(k)}$ and $\sum_{j=1}^{k}\lambda_j^{(k)}=\mathrm{trace}(A^{[k]})$. This is \eqref{eq:generic cond 1 in lem}.

     To satisfy the conditions in  \eqref{eq:generic cond 2 in lem}, we can apply a successive pairwise adjustment to the set $\{\lambda_j^{(k)}\}$, adjusting the pairs by tansformation  $(\lambda^{(k)}_i, \lambda^{(k)}_j)\rightarrow (\lambda^{(k)}_i+1,\lambda^{(k)}_j-1)$, until all  $\mathrm{Re}\lambda^{(k)}_j$ fall into a closed interval $I$ with $\mathrm{length}(I)=1$. This process preserves the identities \eqref{eq:generic cond 1 in lem}, and ends in finite steps.

     It follows from the adjustment procedure that the final sequence $\lambda_j^{(k)}$ satisfying Condition \eqref{eq:generic condition 1} is uniquely determined, as any further unit transfer would result in a non-admissible sequence (i.e., one where $\mathrm{Re}\lambda_j \notin I$).

     For $Y=\nu_{-d}^{(0)}(U,V,A,G)$, we also have 
        \begin{align*}
        e^{-2\pi\mathrm{i}\left(\mathrm{trace}((G^{-1}AG)^{[k]})\right)}=\mathrm{det}(Y^{[k]}),
    \end{align*}
     so the conclusion is established by the same method.

\end{proof}

We will characterize the solutions of the isomonodromy equations using their corresponding monodromy data. The feasibility of this approach is first established by the following lemma.
\begin{lemma}\label{lem:uniqueness in RH}
    If two systems
    \begin{align*}
         \frac{d F}{d \xi}=\left(U+A_1\cdot\xi^{-1}+G_1VG_1^{-1}\cdot\xi^{-2}\right) F,\quad \frac{d F}{d \xi}=\left(U+A_2\cdot\xi^{-1}+G_2VG_2^{-1}\cdot\xi^{-2}\right) F,
    \end{align*}
    have the same monodromy data, i.e. 
    \begin{align}
        &\nu^{(\infty)}_d(U,V,A_1,G_1)=\nu^{(\infty)}_d(U,V,A_2,G_2), \quad \nu^{(0)}_d(U,V,A_1,G_1)=\nu^{(0)}_d(U,V,A_2,G_2),\\
        &C_d(U,V,A_1,G_1)=C_d(U,V,A_2,G_2),\\
        &\delta A_1=\delta A_2,\quad \delta(G_1^{-1}A_1G_1)=\delta(G_2^{-1}A_2G_2),
    \end{align}
    then $A_1=A_2$, $G_1=G_2$.
\end{lemma}
\begin{proof}
    Consider the fundamental solutions $F^{(\infty)}_d(\xi,A_1,G_1)$, $F^{(0)}_{-d}(\xi,A_1,G_1)$ and $F^{(\infty)}_d(\xi,A_2,G_2)$, $F^{(0)}_{-d}(\xi,A_2,G_2)$ defined in Proposition~\ref{prop:fundamental solution of two irr sys} for the two systems, respectively. We have
    \begin{align*}
        F^{(\infty)}_d(\xi,A_1,G_1)\cdot F^{(\infty)}_d(\xi,A_2,G_2)^{-1}&=  (F^{(\infty)}_d(\xi,A_1,G_1)C_d(U,V,A_1,G_1))\cdot (F^{(\infty)}_d(\xi,A_2,G_2)C_d(U,V,A_2,G_2))^{-1}\\
        &= F^{(0)}_{-d}(\xi,A_1,G_1)\cdot F^{(0)}_{-d}(\xi,A_2,G_2)^{-1}.
    \end{align*}
   Note that both $F^{(\infty)}_d(\xi,A_1,G_1) F^{(\infty)}_d(\xi,A_2,G_2)^{-1}$ and $F^{(0)}_{-d}(\xi,A_1,G_1) F^{(0)}_{-d}(\xi,A_2,G_2)^{-1}$ are single-valued and tend to $\mathrm{Id}$ as $\xi$ approaches to their respective singularities ($\infty$ and $0$. So $F(\xi)=  F^{(\infty)}_d(\xi,A_1,G_1)\cdot F^{(\infty)}_d(\xi,A_2,G_2)^{-1}$ is holomorphic on $\mathbb{C}\cup\{\infty\}$,
and therefore $F(\xi)=F(\infty)=\mathrm{Id}$.
\end{proof}

We now show that the solution of the isomonodromy equations corresponding to the strictly log-confined monodromy data is a shrinking solution.
\begin{theorem}\label{thm:almost every solution shrinking}
   Let $A(\mathbf{u},\mathbf{v}),G(\mathbf{u},\mathbf{v})$ be the solution of the isomonodromy equations \eqref{eq:iso eq of two irr A-u}-\eqref{eq:iso eq of two irr G-v}, which serve as the coefficient matrices for the system \eqref{introeq}. Take the monodromy data of system \eqref{introeq}, if it is strictly log-confined, then $A(\mathbf{u},\mathbf{v}),G(\mathbf{u},\mathbf{v})$  is a shrinking solution.
\end{theorem}

The proof requires the following lemma, which can be viewed as the inverse formula of \eqref{eq:catformula for sing irr}:
\begin{lemma}[{\cite[Section 5.2]{TangXu}}] \label{lem:RH correspondence: from stokes mat to phi0}
  Let
\begin{align*}
        V\in \mathrm{GL}_n(\mathbb{C}),\quad \Lambda=\mathrm{diag}(\phi_{11},...,\phi_{nn}).
\end{align*}
If $V\in\mathcal{M}({\Lambda})$,
then there is a unique $\Phi_0\in\mathrm{gl}_n(\mathbb{C})$, such that for each $1\leq k\leq n$ the spectrum $\sigma(\Phi_0^{[k]})$ of its upper-left submatrix $\Phi_0^{[k]}$ is precisely  the set $\sigma_k$
  that ensures  $V\in\mathcal{M}({\Lambda})$ in Definition \ref{def:generic sol}, and moreover, 
\begin{align}
      \mathrm{Ad}\left(\overrightarrow{\prod_{k=1}^{n-1}}C_{d+\mathrm{arg}(u_{k+1}-u_{k})}(E_{k+1},\delta_{k+1}(\Phi_{0}))\right)\mathrm{e}^{2\pi\mathrm{i}\Phi_0}=V,
\end{align}
for some $u_1,...,u_n,\ d\notin aS(u)$, satisfying $\mathrm{Im}(u_1e^{\mathrm{i}d})>...> \mathrm{Im}(u_ne^{\mathrm{i}d})$.
\end{lemma}
\begin{proof}[Proof of Theorem~\ref{thm:almost every solution shrinking}]
  Let $X=\nu_d^{(\infty)}(U,V,A,G)$ and $Y=\mathrm{e}^{\pi\mathrm{i}\delta\widetilde{A}_0}\nu_{-d}^{(0)}(U,V,A,G)\mathrm{e}^{-\pi\mathrm{i}\delta\widetilde{A}_0}$. By  Lemma \ref{lem:RH correspondence: from stokes mat to phi0}, we can find $\widehat{A}_0$ and $\widetilde{A}_0$ satisfying the boundary conditions \eqref{eq:eigen condition of Ahat and Atil}, such that
   \begin{align}\nonumber
       &\delta\widehat{A}_0=\delta A,\ \sigma(X^{[k]})=\{\mathrm{e}^{2\pi\mathrm{i}\lambda}:\lambda\in \sigma(\widehat{A}_0^{[k]}) \};\quad \delta \widetilde{A}_0=-\delta(G^{-1}AG),\ \sigma(Y^{[k]})=\{\mathrm{e}^{2\pi\mathrm{i}\lambda}:\lambda\in \sigma(\widetilde{A}_0^{[k]}) \};\\ \nonumber
       &  \mathrm{Ad}\left(\overrightarrow{\prod_{k=1}^{n-1}}C_{d+\mathrm{arg}(u_{k+1}-u_{k})}\left(E_{k+1},\delta_{k+1}\widehat{A}_{0}\right)\right)\mathrm{e}^{2\pi\mathrm{i}\widehat{A}_0}=X,\\ \label{eq:catformulas for tilA}
       & \mathrm{Ad}\left(\overrightarrow{\prod_{k=1}^{n-1}}C_{d+\pi+\mathrm{arg}(v_{k+1}-v_{k})}\left(E_{k+1},\delta_{k+1}\widetilde{A}_{0}\right)\right)\mathrm{e}^{2\pi\mathrm{i}\widetilde{A}_0}=Y.
   \end{align}
  Using \eqref{eq:simp for Cd}, the equation \eqref{eq:catformulas for tilA} is equivalent to
  \begin{align*}
       \mathrm{Ad}\left(\overrightarrow{\prod_{k=1}^{n-1}}C_{d+\mathrm{arg}(v_{k+1}-v_{k})}\left(-E_{k+1},\delta_{k+1}\widetilde{A}_{0}\right)\right)\mathrm{e}^{2\pi\mathrm{i}\widetilde{A}_0}=\nu_{-d}^{(0)}(U,V,A,G).
  \end{align*}
   By Lemma \ref{prop:identity of monodromy factor for two irr sys},  $\nu_d^{(\infty)}$ and $\left(\nu_{-d}^{(0)}\right)^{-1}$ are similar matrix, therefore it follows by the uniqueness criterion in Lemma \ref{Lem: RH corr 1}  that $\sigma(\widehat{A}_0)=-\sigma(\widetilde{A}_0)$.

   Let \begin{align*}
        G_0=\left( \overrightarrow{\prod_{k = 1}^{n-1}} C_{d+\mathrm{arg}({u}_{k+1}-{u}_k)}\left(E_{k+1},\delta_{k+1}\widehat{A}_0\right) \right)^{-1} C_d(U,V,A,G)  \left( \overrightarrow{\prod_{k = 1}^{n-1}} C_{d+\mathrm{arg}({v}_{k+1}-{v}_{k})}\left(-E_{k+1},\delta_{k+1}\widetilde{A}_0\right) \right).
   \end{align*}
   Also using \eqref{eq:connection relation for 0 and infty} in  Lemma \ref{prop:identity of monodromy factor for two irr sys}, we can verify that $\mathrm{e}^{2\pi\mathrm{i}G_0^{-1}\widehat{A}_0G_0} =\mathrm{e}^{-2\pi\mathrm{i}\widetilde{A}_0}$. Combined with the spectrum condition $\sigma(\widehat{A}_0)=-\sigma(\widetilde{A}_0)$ and the boundary condition, this implies 
   \begin{align*}
       - G_0^{-1}\widehat{A}_0G_0=\widetilde{A}_0.
   \end{align*}

   Therefore, we can construct $A\left(\widehat{A}_0,G_0\right)$ and $G\left(\widehat{A}_0,G_0\right)$ as in Proposition~\ref{prop:from A0,G0 to hatA, hatG}. Finally, by Corollary~\ref{cor: concrete monodromy data by A0 and G0}, the monodromy data of the linear system with coefficients $A\left(\widehat{A}_0,G_0\right)$ and $G\left(\widehat{A}_0,G_0\right)$ coincides with that of the original system with coefficients $A$ and $G$. Thus, by Lemma~\ref{lem:uniqueness in RH}, we obtain $
A = A\left(\widehat{A}_0,G_0\right)$ and $G = G\left(\widehat{A}_0,G_0\right).$
\end{proof}

Finally, we  prove that almost every solution of the isomonodromy equations is a shrinking solution. 

\begin{prop}\label{prop:almost every sol is shr}
The set of strictly log-confined monodromy data is an open and dense subset of the entire space of monodromy data for system \eqref{introeq}. Consequently, almost all solutions of the isomonodromy equations are classified as shrinking solutions.
\end{prop}
\begin{proof}
    Lemma \ref{Lem: RH corr 1} states that if the monodromy data is not strictly log-confined, then it is only possible that for some sets $\sigma_k$ in Definition \ref{def:generic sol}, the following conditions occur:
    \begin{align} \label{eq:nongeneri cond1}
         &\left|\mathrm{Re}\left(\lambda_{j_1}^{(k)}-\lambda_{j_2}^{(k)}\right)\right|= 1,\ \text{ for some $1\leq j_1,j_2\leq k$},\\  \label{eq:nongeneri cond2}
    & \lambda_{j_1}^{(k+1)}-\lambda_{j_2}^{(k)}\in\mathbb{Z}\setminus\{0\},\ \text{ for some $1\leq j_1\leq k+1$,\ $1\leq j_2\leq k$}.
    \end{align}
    On the other hand, it follows from the  Lemma \ref{lem:RH correspondence: from stokes mat to phi0} and the construction of Theorem \ref{thm:almost every solution shrinking} that any tuple $$(h_1,h_2,V_1,V_2,C)\in\left(\mathrm{Diag}_n(\mathbb{C})\right)^2\times\left( \mathrm{Mat}_{n\times n}(\mathbb{C})\right)^3$$ such that $V_1=CV_2C^{-1}$, $V_1\in\mathcal{M}(h_1)$ and $V_2\in\mathcal{M}(h_2)$  can serve as the monodromy data for system \eqref{introeq}. This implies that the collection of monodromy data satisfying conditions \eqref{eq:nongeneri cond1} and \eqref{eq:nongeneri cond2} lies in a closed subset of lower dimension within the space of all monodromy data. Therefore,  the set of strictly log-confined monodromy data is open and dense.
\end{proof}
\begin{remark}
    For a shrinking solution $A(\widehat{A}_0,G_0),G(\widehat{A}_0,G_0)$, let $\sigma_k=\sigma(\widehat{A}_0^{[k]})$ be as in Definition \ref{def:generic sol}. Then the conditions \eqref{eq:generic cond 1 in def} and \eqref{eq:generic cond 2 in def} are satisfied. Consequencely, the solutions with {strictly log-confined} monodromy data is open and dense within $\mathfrak{Sol}_{Shr}$.
\end{remark}
\section{Applications in $tt^*$ equations}\label{sec:apply in tt}
In this section, we explore the application of our results to  $tt^*$ equations. In section \ref{sec:apply for sg piii}, we consider the $tt^{*}$-equations for matrix order $n=2$, which are equivalent to the sine-Gordon Painlevé III equation; by specializing our asymptotic results and formulas for the Stokes matrices to this case, we find consistency with known results in \cite{FIKN2006}. In section \ref{sec:apply for tt toda}, we consider a special class of $tt^{*}$-equations known as the $A_n$ type $tt^{*}$-Toda equations. A comparison of our asymptotic analysis with certain asymptotic results for the global smooth solutions detailed in \cite{guest2015isomonodromy1,guest2015isomonodromy2,guest2023toda} reveals that a subset of these solutions corresponds to our non-shrinking solutions. Finally, in section \ref{sec:apply for general tt}, we consider the $tt^{*}$-equations formulated in \cite{dubrovin1993tteq}. Applying our asymptotic results to this equations yields a description of the local behavior near $t=0$ for a family of solutions.

\subsection{Sine-Gorden Painlevé III equation}\label{sec:apply for sg piii}
We will use our asymptotic results and the monodromy formula to recover the corresponding results for a specific class of Painlevé III equations.
Consider the following compatible system
  \begin{align}\label{eq:sys for sg-PIII}
\frac{\partial Y}{\partial \xi} &= \left(\frac{\mathrm{\mathrm{i}}}{16}\sigma_3- \frac{\mathrm{i}xu_x}{4\xi}\sigma_1-  \frac{\mathrm{i}x^2}{\xi^{2}} G\cdot
\sigma_3\cdot G^{-1}
 \right) Y,  \\
\frac{\partial Y}{\partial x} &= \left(\frac{2\mathrm{i}x}{\xi} G\cdot\sigma_3\cdot G^{-1}\right)Y,
\end{align}
where $\sigma_1,\sigma_3$ are  Pauli matrices,
\begin{align*}
  \sigma_1=  \begin{pmatrix}
0 & 1 \\
1 & 0
\end{pmatrix},\quad     \sigma_3=  \begin{pmatrix}
1 & 0 \\
0 & -1
\end{pmatrix},
\end{align*}
$u,u_x$ are meromorphic functions on $\mathbb{C}^*$ with variable $x$, and 
\begin{align*}
    G=\begin{pmatrix}
\cos{\frac{u}{2}} & -\mathrm{i}\sin{\frac{u}{2}} \\
-\mathrm{i}\sin{\frac{u}{2}} & \cos{\frac{u}{2}}
\end{pmatrix}.
\end{align*}

Take the direction $d=0$, we can parametrize the Stokes matrices and connection matrix for system \eqref{eq:sys for sg-PIII} defined in \ref{def:stokes for 2irr} and \ref{def:conn for 2irr} as follows (see \cite[Chapter 13]{FIKN2006} for details):

\begin{align}\label{eq:Stokes mat for sgpiii}
   & S_{0,+}^{(\infty)}=\begin{pmatrix}
1 & -(p+q) \\
0 & 1
\end{pmatrix},\quad S_{0,-}^{(\infty)}=\begin{pmatrix}
1 & 0 \\
p+q & 1\end{pmatrix},\quad    S_{0,+}^{(0)}=\begin{pmatrix}
1 & p+q \\
0 & 1
\end{pmatrix},\quad S_{0,-}^{(0)}=\begin{pmatrix}
1 & 0 \\
-(p+q) & 1
\end{pmatrix},\\ \label{eq:conn mat for sgpiii}
& C_0=\frac{1}{\sqrt{1+pq}}\begin{pmatrix}
1 & q \\
-p & 1\end{pmatrix}.
\end{align}

The isomonodromy equation for \eqref{eq:sys for sg-PIII} can be reduced to the sine-Gorden Painlevé III equation:
\begin{align}\label{eq:sg-PIII}
    u_{xx}+\frac{1}{x}u_x+\sin{u}=0.
\end{align}
It is equivalent to the matrix order $n=2$ case in $tt^{*}$ equations as formulated in \cite{dubrovin1993tteq}.
Applying our Proposition \ref{thm:asym in t}, we can reproduce the asymptotic behaviors as well as boundary condition of its solutions as $x\rightarrow 0$ given in \cite{FIKN2006}:

\begin{cor}\label{prop:asy for sgpiii}
    For any given $r,s\in \mathbb{C}$ such that $|\mathrm{Im}(r)|<2$, there exists a unique solution  $u(x)$ of \eqref{eq:sg-PIII}, with the following behaviors near zero: 
    \begin{align}\label{eq:asy for piii}
    u(x)=r\mathrm{log}x+s,\ \text{as}\ x\rightarrow 0.
    \end{align}
\end{cor}
\begin{proof}
    For equation \eqref{eq:sys for sg-PIII}, the variable $t$ in the $(\mathbf{z},t,\mathbf{w})$  coordinate system \eqref{coor1}-\eqref{coor3} is given by $t=\frac{x^2}{4}$. Applying Proposition \ref{thm:asym in t}, for any prescribed 
 \begin{align}
        &\widehat{A}=\begin{pmatrix}
0 & -\frac{r\mathrm{i}}{4} \\
-\frac{r\mathrm{i}}{4} & 0
\end{pmatrix}=\frac{1}{2}\begin{pmatrix}
1 & 1 \\
-1 & 1
\end{pmatrix}\cdot \begin{pmatrix}
    \frac{r\mathrm{i}}{4} & 0\\
   0 &  - \frac{r\mathrm{i}}{4}
\end{pmatrix}\cdot \begin{pmatrix}
1 & -1 \\
1 & 1
\end{pmatrix}, \\
& \widehat{G}= \frac{1}{2}\begin{pmatrix}
1 & 1 \\
-1 & 1
\end{pmatrix}\cdot \begin{pmatrix}
\mathrm{e}^{\frac{\mathrm{i}s}{2}}\cdot 2^{\frac{\mathrm{i}r}{2}} & 0\\
0 & \mathrm{e}^{-\frac{\mathrm{i}s}{2}}\cdot 2^{-\frac{\mathrm{i}r}{2}}
\end{pmatrix}\cdot \begin{pmatrix}
1 & -1 \\
1 & 1
\end{pmatrix},
    \end{align}
satisfying $|\mathrm{Im}(r)|<2$, there exists a unique solution $(A(t), G(t))$ to \eqref{eq:iso for only t A} and \eqref{eq:iso for only t G} (where $U=\frac{\mathrm{i}}{16}\sigma_3$, $V=-\mathrm{i}x^2\sigma_3$, and $\delta(G^{-1}AG)=0$), such that 
\begin{align}\label{der of asy for piii}
    A\rightarrow\widehat{A},\quad t^{-\widehat{A}}G\rightarrow\widehat{G},\quad \text{as } t\rightarrow 0.
\end{align}
Following the Picard iteration \eqref{picard iter for A}--\eqref{picard iter for B} in Proposition \ref{thm:asym in t}, the solution $A(t),G(t)$ remains the form as above $\widehat{A},\widehat{G}$, thus can be represented in the form of \eqref{eq:sys for sg-PIII} in terms of a function $u(x)$ which solves the equation \eqref{eq:sg-PIII}. And then the asymptotics \eqref{der of asy for piii} translates to 
\begin{align*}
    xu_x\rightarrow r,\quad x^{-\frac{\mathrm{i}r}{2}}\cdot \mathrm{e}^{\frac{\mathrm{i}u}{2}}\rightarrow \mathrm{e}^{\frac{\mathrm{i}s}{2}}.
\end{align*}
Finally, we can choose a branch of the logarithm such that $u(x)$ satisfies the asymptotics \eqref{eq:asy for piii}, and then the solution $u(x)$ of \eqref{eq:sg-PIII} is uniquely determined by $A(t),G(t)$.
\end{proof}

Furthermore, by applying Theorem \ref{thm: introcatformula}, we can reproduce the explicit formulas in \cite{FIKN2006} for the monodromy matrices expressed in terms of the asymptotic parameters $r,s$.

\begin{cor}\label{cor:expr of stokes mat for piii}
    The monodromy parameters $p,q$ in \eqref{eq:Stokes mat for sgpiii}-\eqref{eq:conn mat for sgpiii} associated with the solution constructed in Corollary \ref{prop:asy for sgpiii} can be expressed by 
    \begin{align}\label{eq:exp of mono para for piii p,q}
            p = \frac{\alpha e^{-\frac{\pi r}{4}} -\beta e^{\frac{\pi r}{4}}}{\alpha + \beta}, \quad 
q = \frac{\beta e^{-\frac{\pi r}{4}} - \alpha e^{\frac{\pi r}{4}}}{\alpha + \beta},
        \end{align}
        \[
\text{where } \alpha = 2^{\frac{3 i r}{2}} e^{\frac{i s}{2}} \Gamma^{2}\left(\frac{1}{2} + \frac{i r}{4}\right), \quad 
\beta = 2^{-\frac{3 i r}{2}} e^{-\frac{i s}{2}} \Gamma^{2}\left(\frac{1}{2} - \frac{i r}{4}\right)
\].
\end{cor}
\begin{proof}
    By Corollary \ref{prop:asy for sgpiii}, 
    \begin{align*}
       \widetilde{A}=-\widehat{G}^{-1}\widehat{A}\widehat{G}=-\widehat{A}=\begin{pmatrix}
0 & \frac{r\mathrm{i}}{4} \\
\frac{r\mathrm{i}}{4} & 0
\end{pmatrix},\quad \delta(\widehat{A})=\delta(\widetilde{A})=0.
    \end{align*}
    Then by part (b) in Theorem \ref{mainthm}, $\widehat{A}_0=\widehat{A},\ G_0=\widehat{G},\  \widetilde{A}_0=\widetilde{A}$. Using the notation from Theorem \ref{thm: introcatformula}, we have $\widehat{\lambda}^{(1)}_1=\widetilde{\lambda}^{(1)}_1=0,\ \widehat{\lambda}^{(2)}_1=\widetilde{\lambda}^{(2)}_2= \frac{r\mathrm{i}}{4},\ \widehat{\lambda}^{(2)}_2=\widetilde{\lambda}^{(2)}_1= -\frac{r\mathrm{i}}{4}$, $
    (\widehat{A}_0)_{12}=(\widehat{A}_0)_{21}=-\frac{r\mathrm{i}}{4},\  (\widetilde{A}_0)_{12}=(\widetilde{A}_0)_{21}=\frac{r\mathrm{i}}{4}$ and $H_0=\mathrm{diag}(\mathrm{e}^{\frac{\mathrm{i}s}{2}}2^{\frac{r\mathrm{i}}{2}},\mathrm{e}^{-\frac{\mathrm{i}s}{2}}2^{-\frac{r\mathrm{i}}{2}})$. Substituting these parameters in Theorem \ref{thm: introcatformula}, and using the following identities for Gamma function, we can obtain the above expressions.
    \begin{align*}
        &\Gamma\left(1-\frac{r\mathrm{i}}{4}\right)\Gamma\left(1+\frac{r\mathrm{i}}{4}\right)=\frac{\pi \frac{r\mathrm{i}}{4}}{\sin{\pi \frac{r\mathrm{i}}{4}}}=\frac{\pi r}{2(\mathrm{e}^{\frac{\pi r}{4}}-\mathrm{e}^{\frac{-\pi r}{4}})},\\
        &\Gamma\left(1\pm\frac{r\mathrm{i}}{4}\right)\Gamma\left(\frac{1}{2}\pm\frac{r\mathrm{i}}{4}\right)=2^{\pm\frac{ r\mathrm{i}}{2}}\sqrt{\pi}\cdot\Gamma(1\pm\frac{r\mathrm{i}}{2}).
    \end{align*}
\end{proof}

From the expressions of monodromy parameters $p,q$ in \eqref{eq:exp of mono para for piii p,q}, the solutions obtained in Corollary \ref{prop:asy for sgpiii} constitute all possible solutions of the sine-Gordon Painlevé III equation, except for the case where  $p,q$ satisfy $p+q=\kappa\mathrm{i},\ \text{for}\ \kappa\in\mathbb{R},\ \text{and}\  |\kappa|\geq2$.

\subsection{$tt^{*}$-Toda equations and a family of non-shrinking solutions}\label{sec:apply for tt toda}
In this section, we compare our results with certain known asymptotic results for the $tt^{*}$-Toda equations. We begin with a brief review of the setup for the $tt^{*}$-Toda equations as presented in \cite{guest2015isomonodromy1,guest2015isomonodromy2,guest2023toda}, following the conventions therein.

The $tt^{*}$-Toda equations (of $A_n$ type) are 
\begin{align}\label{eq:tt toda}
    2(w_i)_{z\bar{z}}=-\mathrm{e}^{2(w_{i+1}-w_i)}+\mathrm{e}^{2(w_i-w_{i-1})},\ w_i:U\subset\mathbb{C}\rightarrow \mathbb{R}, \ i\in\mathbb{Z},
\end{align}
where the functions $w_i(z,\bar{z})$ satisfy the conditions $w_i=w_{i+n+1}$, $w_i=w_i(|z|)$ and $w_i+w_{n-i}=0$. This system can be reformulated as the isomonodromy  equations for a linear ODE system with two second-order poles as follows.

Suppose $w = \operatorname{diag}(w_{0}, \ldots, w_{n})$,  $\Omega = (\Omega_{kj})_{0\leq k,j\leq n}$ and $D$ be $(n+1)\times (n+1)$ matrices defined by
\begin{equation*}
    \Omega_{kj} = \mathrm{e}^{kj\frac{2\pi\mathrm{i}}{n+1}}, \quad 
    D = \operatorname{diag}\left(1, \mathrm{e}^{\frac{2\pi\mathrm{i}}{n+1}}, \mathrm{e}^{\frac{4\pi\mathrm{i}}{n+1}}, \dots, \mathrm{e}^{\frac{2n\pi\mathrm{i}}{n+1}}\right).
\end{equation*}
Let
\begin{align}
W =e^{-w}\Omega D \Omega^{-1}e^w= \left(
\begin{array}{cccc}
 & e^{w_{1}-w_{0}} & & \\
 & & \ddots & \\
 & & & e^{w_{n}-w_{n-1}} \\
e^{w_{0}-w_{n}} & & & 
\end{array}
\right).
\end{align}
Taking $x=|z|$, then the $tt^{*}$-Toda equations with radial solutions $w_i(x)$ can be written as
\begin{align}\label{eq:eq for tt toda matform}
    (xw_x)_x=2x[W^t,W],
\end{align}
where $W^t$ is the transpose of $W$. It is the compatibility condition of the following system:
\begin{align}
        \frac{\partial \Phi}{\partial \xi}&=\left(W^t-\frac{1}{\xi}xw_x-\frac{x^2}{\xi^2}W\right)\Phi,\\
     \frac{\partial \Phi}{\partial x}&=\left(\frac{2x}{\xi}W+w_x\right)\Phi.
\end{align}

Let $\Psi=\Omega\mathrm{e}^{-w}\cdot\Phi(\xi,x)$, then $\Psi$ satisfies
\begin{align}\label{eq:standard sys for tt* toda}
    \frac{\partial \Psi}{\partial\xi}=\left(D-\frac{1}{\xi}\Omega(xw_x)\Omega^{-1}-\frac{x^2}{\xi^2}GDG^{-1}\right)\Psi,
\end{align}
where $G=\Omega\mathrm{e}^{-2w}\Omega$.

Although $x=|z|$ is a real variable for $tt^{*}$-Toda equations, $w(x)$ can be analytically continued to a multi-valued meromorphic function on $\mathbb{C}$, by the Painlev\'e property of these isomonodromy equations. So applying Proposition \ref{thm:asym in t}, we have 
\begin{cor}
    For any given $m_i\in\mathbb{R},\ l_i\in \mathbb{R}_{+},\ i=0,...,n$, such that 
    \begin{equation}\label{shrtt*}
        |m_i-m_j|<1, \ m_i+m_{n-i}=0,\ l_i\cdot l_{n-i}=1,
    \end{equation}
    there exists a unique solution $w_i(x),\ i=0,1,...,n$, of \eqref{eq:eq for tt toda matform} with the following behaviors near zero:
    \begin{align}\label{eq:asy in tt toda}
        w_i=m_i\mathrm{log}x-\frac{1}{2}\mathrm{log}l_i + o(1), \ i=0,...,n,\ \text{as}\ x\rightarrow 0
    \end{align}
\end{cor}
\begin{proof}
   Since the sine-Gorden Painlev\'e III equation is the matrix order 2 case of  $tt^*$ -Toda equations of $A_n$ type, the proof here proceeds in parallel with that of Corollary \ref{prop:asy for sgpiii}. As established above, we identify the $tt^*$-Toda equations \eqref{eq:eq for tt toda matform} with the isomonodromy equations of system \eqref{eq:standard sys for tt* toda}. In the $(\mathbf{z},t,\mathbf{w})$ coordinate system \eqref{coor1}--\eqref{coor3}, the variable $t$ is given by $t=-(1-\mathrm{e}^{-\frac{2\pi\mathrm{i}}{n+1}})^2x^2$. By Proposition \ref{thm:asym in t}, for any prescribed 
\begin{align*}
    &\widehat{A}:=-\Omega\cdot\Lambda\cdot\Omega^{-1},\quad \Lambda:=\mathrm{diag}(m_0,\dots,m_n),\\
    &\widehat{G}:=\Omega\left(-(1-\mathrm{e}^{-\frac{2\pi\mathrm{i}}{n+1}})^2\right)^{\Lambda}\cdot L\cdot\Omega,\quad L:=\mathrm{diag}(l_0,\dots,l_n),
\end{align*}
satisfying \eqref{shrtt*}, there exists a unique solution $(A(t), G(t))$ to \eqref{eq:iso for only t A} and \eqref{eq:iso for only t G} (where $U=D$, $V=-x^2D$, and $\delta(G^{-1}AG)=0$), such that 
\begin{align}\label{der of asy for tteq}
    A\rightarrow\widehat{A},\quad t^{-\widehat{A}}G\rightarrow\widehat{G},\quad \text{as } t\rightarrow 0.
\end{align}
Furthermore, since both $\Omega^{-1}\widehat{A}\Omega$ and $\Omega^{-1}\widehat{G}\Omega^{-1}$ are diagonal, the Picard iteration \eqref{picard iter for A}--\eqref{picard iter for B} in Proposition \ref{thm:asym in t} ensures that $\Omega^{-1}{A}\Omega$ and $\Omega^{-1}{G}\Omega^{-1}$ also remain diagonal. 

Setting $\mathrm{e}^{-2w}=\Omega^{-1}G\Omega^{-1}$, the equation \eqref{eq:iso for only t G} for $G$ is now equivalent to $xw_x=-\Omega^{-1}A\Omega$, and the equation \eqref{eq:iso for only t A} for $A(t)$ reduce to equation \eqref{eq:eq for tt toda matform} for $w(x)$. Moreover, the asymptotics \eqref{der of asy for tteq} translate to:
\begin{align*}
    xw_x &\rightarrow \mathrm{diag}(m_0,\dots,m_n),\quad \text{as } x\rightarrow 0,\\
    \mathrm{e}^{2(-w+\Lambda\log x)} &\rightarrow \mathrm{diag}(l_0,\dots,l_n),\quad \text{as } x\rightarrow 0.
\end{align*}
Thus, we can choose a branch of the logarithm such that $w(x)$ satisfies the asymptotics \eqref{eq:asy in tt toda}, and then the solution $w(x)$ to equation \eqref{eq:eq for tt toda matform} is uniquely determined by the relation $\mathrm{e}^{-2w}=\Omega^{-1}G\Omega^{-1}$. Since $m_i,l_i$ are real, it follows that $w(x)$ is real-valued when restricted to the positive real axis. And condition \eqref{shrtt*} implies that $w_i(x)=w_{n-i}(x)$. Thus, $w(x)$ is indeed a solution to the $tt^*$-Toda equations.
\end{proof}
In the work of \cite{guest2015isomonodromy1,guest2015isomonodromy2,guest2023toda}, global smooth solutions $w(z,\bar{z})$ defined on $\mathbb{C}^*$ are considered. These solutions have prescribed asymptotic behavior at infinity ($w_i(x)\rightarrow 0 $ as $x\rightarrow\infty$) and thus can be parameterized  just by the parameters $m_i,\ i=0,...,n$, at the origin. The explicit expression of the
 $l_i$  in terms of the parameters $m_i$ for  global smooth solutions is given in \cite{guest2023toda}.

In \cite{guest2023toda}, the parameters range for all global smooth solutions is 
\[\{m\in\mathbb{R}^{n+1}: m_i-m_{i-1}\leq 1, \ m_i+m_{n-i}=0\}.\] This range is wider than the boundary condition \eqref{shrtt*}. (One can think of these solutions of $tt^*$-Toda equations as a $\frac{n}{2}$ parameters family of solutions of general rank $n$ isomonodromy equations with $2n^2$ parameters. Our boundary condition is open and dense in the space of $2n^2$ parameters, but is not dense when restricts to the slice of $2n$ parameters). Thus, some global smooth solutions of $tt^{*}$-Toda equations {yield solutions that} are not in the set $\mathfrak{Sol}_{Shr}$, which also means their monodromy data do not satisfy the strictly log-confined condition in our Definition \ref{def:log-conf cond}.  As shown in the proof of Proposition \ref{prop:almost every sol is shr}, only identities \eqref{eq:nongeneri cond1} and \eqref{eq:nongeneri cond2} can occur.

For matrices of order 4 and 5, we further numerically examine which of the two identities \eqref{eq:nongeneri cond1} and \eqref{eq:nongeneri cond2} specifically arise within the monodromy data corresponding to these non-shrinking solutions. The explicit formulas for the Stokes matrices in terms of asymptotic parameters for matrices of orders 4, 5, 6 are provided in 
\cite{guest2015isomonodromy1}. After accounting for the braid group action, we obtain the monodromy matrices at the points $U,V$ and direction $d$ satisfying \eqref{eq:u,v cond for monodromy mat}, as required in Definition \ref{def:log-conf cond}. Numerical verification shows that the monodromy matrix at infinity $M=\nu_d^{(\infty)}$  always has some upper-left submatrices $M^{[k]}$ with paired negative real eigenvalues, leading to  $|\mathrm{Re}(\lambda^{(k)}_i-\lambda^{(k)}_j)|=1$ for some $1\leq i<j\leq k$. 

In the case of matrices of order 4, a typical numerical example is as follows. We take the eigenvalues of $\widehat{A}$ to be  $(m_0,m_1,m_2,m_3)=(0.4,-0.55,-0.4,0.55)$, which do not satisfy the boundary condition \eqref{eq:shrinking in t}. (This eigenvalues condition corresponds to the case 4a  in \cite[Theorem A,\ Theorem B]{guest2015isomonodromy1}, matching their parameters $\gamma=0.8,\ \delta=-1.1$. )  By the formula given in \cite[Corollary 4.7]{guest2015isomonodromy1}, the Stokes matrices are 
\begin{align*}
 \left(S^{(\infty)}_{\frac{\pi}{8},+}\right)^{-1} =  \begin{pmatrix}
1 & 0 & 2.04909\mathrm{i} &-0.33219+0.33219\mathrm{i} \\
-0.33219+0.33219\mathrm{i} &1&-0.34850-0.34850\mathrm{i}& 1.82839\mathrm{i}\\
0 & 0&1&0.33219+0.33219\mathrm{i}\\
0&0&0&1
\end{pmatrix},\\
S^{(\infty)}_{\frac{\pi}{8},-}= \begin{pmatrix}
1 & 0.33219+0.33219\mathrm{i} & 0 &0 \\
0 &1&0& 0\\
2.04909\mathrm{i} & -0.33219+0.33219\mathrm{i}&1&0\\
-0.34850-0.34850\mathrm{i}&1.82839\mathrm{i}&-0.33219+0.33219\mathrm{i}&1
\end{pmatrix}.
\end{align*}
Here $\left(S^{(\infty)}_{\frac{\pi}{8},+}\right)^{-1}, S^{(\infty)}_{\frac{\pi}{8},-}$ are $S_1,S_2$ in \cite{guest2015isomonodromy1} respectively. Since the aforementioned Stokes matrices are computed at $(u_1,u_2,u_3,u_4)=(1,\mathrm{i},-1,-\mathrm{i})$, we apply the following braid group action to transform them into the stokes matrices at $(u_1,u_2,u_3,u_4)=(\mathrm{i},1,-1,-\mathrm{i})$, thereby satisfying the condition \eqref{eq:u,v cond for monodromy mat}: let
\begin{align*}
 &\mathcal{B}_1 = \mathrm{Id_4}+s_1E_{21}=\mathrm{Id_4}+(-0.33219+0.33219\mathrm{i})E_{21},\\
 &\mathcal{B}_2=\mathrm{Id}_4-\bar{s}_1E_{12}=\mathrm{Id_4}+(0.33219+0.33219\mathrm{i})E_{12},
\end{align*}
where $s_1=2\mathrm{e}^{\frac{3\pi\mathrm{i}}{4}}\left(\cos\frac{\pi}{4}(1+2m_0)+\cos\frac{\pi}{4}(3+2m_1)\right)$ as given in \cite{guest2015isomonodromy1}, and $E_{ij}$ denotes the $4\times4$ matrix with 1 at the $(i,j)$-entry and 0 elsewhere. Thus, the Stokes matrices at $(\mathrm{i},1,-1,-\mathrm{i})$ are 
\begin{align*}
    S_1'=\mathcal{B}_1^{-1}\left(S^{(\infty)}_{\frac{\pi}{8},+}\right)^{-1}\mathcal{B}_2,\quad S_2'=\mathcal{B}_2^{-1}S^{(\infty)}_{\frac{\pi}{8},-}\mathcal{B}_1.
\end{align*}
Therefore the monodromy matrix $\nu_{\frac{\pi}{8}}^{(\infty)}$ at $(\mathrm{i},1,-1,-\mathrm{i})$ is 
\begin{align*}
   \nu_{\frac{\pi}{8}}^{(\infty)}=\left(S_1'S_2'\right)^{-1} =\begin{pmatrix}
        1&-0.33219-0.33219\mathrm{i}&-1.82839\mathrm{i}&-0.95587+0.95585\mathrm{i}\\
        0.33219-0.33219\mathrm{i}&0.77930&-0.93956-0.93956\mathrm{i}&-1.19333\mathrm{i}\\
        -2.04909\mathrm{i}&-0.34850+0.34850\mathrm{i}&-2.96752&1.01911+1.01911\mathrm{i}\\
        -0.33219-0.33219\mathrm{i}&-1.82839\mathrm{i}&-0.95587+0.95587\mathrm{i}&-2.33219
    \end{pmatrix}.
\end{align*}
The eigenvalues of the upper-left $3\times 3$ submatrix of $\nu_{\frac{\pi}{8}}^{(\infty)}$ is 
$(-1.53758,1,-0.65037)$. Thus the $\sigma_3$ in Definition \ref{def:generic sol} is $(-0.5,0,0.5)$, which does not satisfy the condition \eqref{eq:generic cond 2 in def}.

The two scatter figures \ref{fig:order4}, \ref{fig:order5} below summarize our numerical results. The axes $\gamma$ and $\delta$  correspond to $2m_0$ and $2m_1$  respectively (twice eigenvalues of $\widehat{A}$). The colored points indicate $(\gamma,\delta)$ values for which corresponding submatrix of $M$ exhibits paired negative real eigenvalues, thus the corresponding monodromy data are not strictly log-confined. Within the region where $|\gamma|<1, |\delta|<1$, no colored points appear. This indicates that the monodromy matrices corresponding to these $(\gamma,\delta)$ values is { strictly log-confined }, which aligns with our theoretical expectations.
\begin{figure}[H]
    \centering
    \includegraphics[scale=0.35]{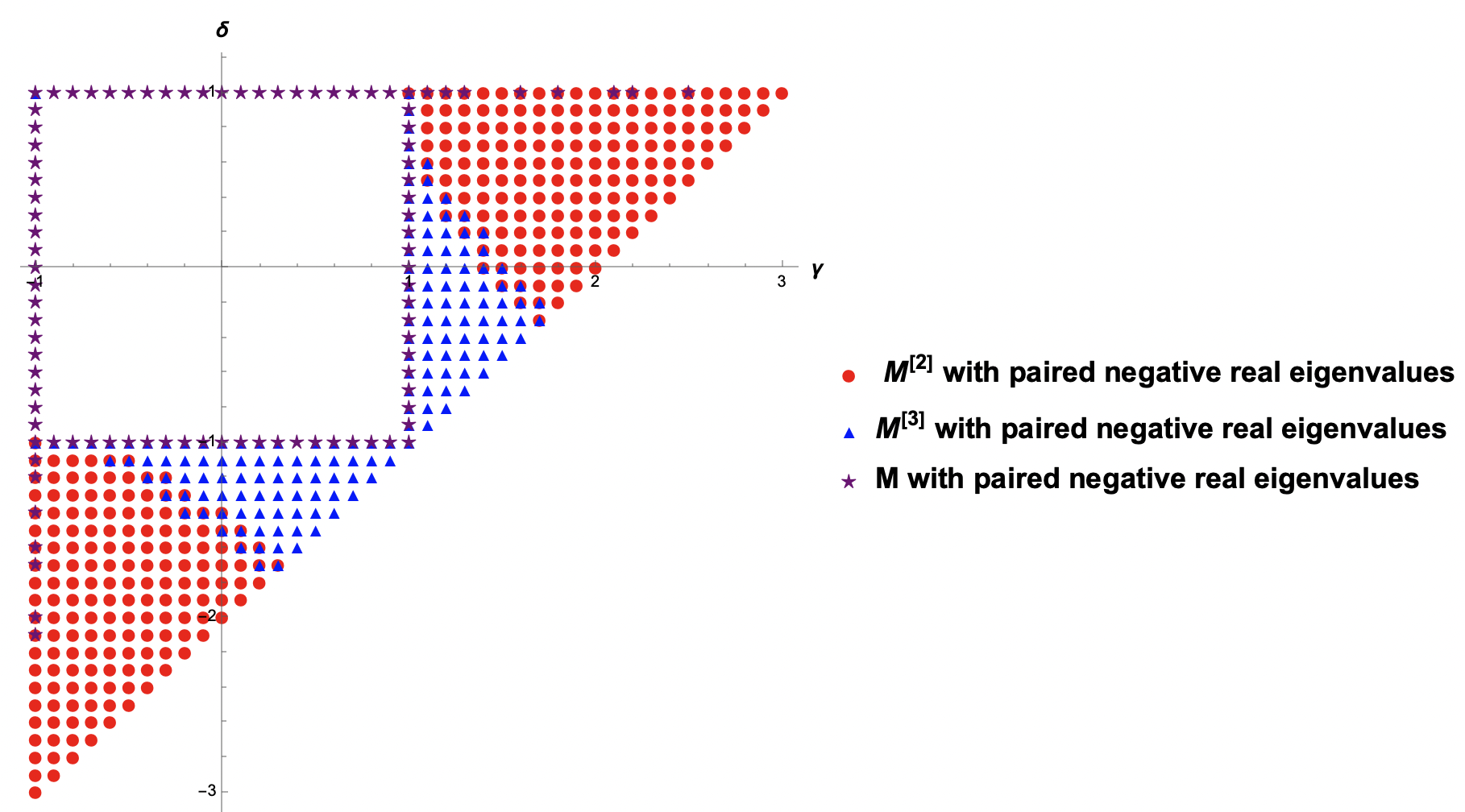}
    \caption{Discrimination of Monodromy matrix $M=\nu^{(\infty)}_{\frac{\pi}{8}}$ for order 4}
    \label{fig:order4}
\end{figure}

\vspace{-15pt} 

\begin{figure}[H]
    \centering
    \includegraphics[scale=0.35]{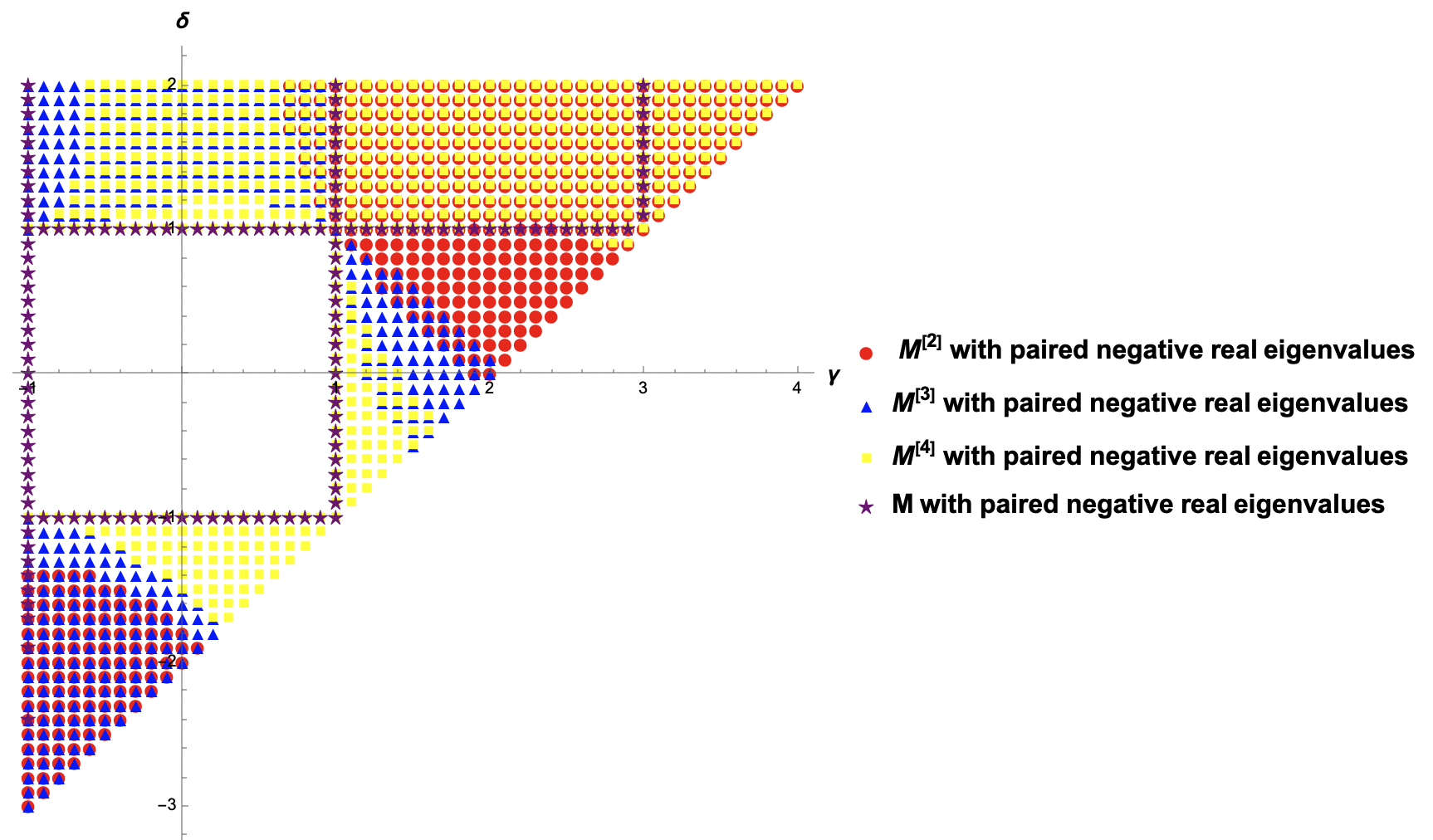}
    \caption{Discrimination of Monodromy matrix $M=\nu^{{(\infty)}}_{\frac{\pi}{10}}$ for order 5}
    \label{fig:order5}
\end{figure}

\subsection{General $tt^{*}$ equations}\label{sec:apply for general tt}

In this section, we use our asymptotic results to describe the local behavior of the solutions to the general $tt^{*}$ equations with similarity reduction as $t\rightarrow0$. The $tt^{*}$ equations originated from the study of $N=2$ supersymmetric quantum field theories in \cite{cecotti1991topological,cecotti1993classification}, where their solutions represent the metric for supersymmetric ground state. Dubrovin provided a mathematical interpretation of  $tt^{*}$ equations in \cite{dubrovin1993tteq}, formulating them as the compatibility conditions of a linear PDE system. And the solutions of $tt^{*}$ equations define a geometric structure (special geometry structure) on the Frobenius manifold.

Based on \cite{dubrovin1993tteq}, suppose $\mathbf{u}=(u_1,...,u_n)$ is canonical coordinate on a Frobenius manifold, $\bar{\mathbf{u}}=(\bar{u}_1,...,\bar{u}_n)$ is the conjugate coordinate of $\mathbf{u}$, let $q(\mathbf{u},\bar{\mathbf{u}})$ be a symmetric off-diagonal matrix function, $m(\mathbf{u},\bar{\mathbf{u}})$ be an orthogonal, Hermitian matrix function, such as 
\begin{align*}
    [U,q]=m[\overline{U},\overline{q}]\overline{m},\quad U=\mathrm{diag}(u_1,...,u_n).
\end{align*}
The $tt^{*}$ equations (with similarity reduction) are equations for  $q(\mathbf{u},\bar{\mathbf{u}}) ,\ m(\mathbf{u},\bar{\mathbf{u}})$ as following,
\begin{align}
    -[U,\frac{\partial }{\partial u_k}q]&=[\mathrm{ad}_{E_k}q,\mathrm{ad}_Uq]+[E_k,\ q+m\overline{U}m^{-1}],\\
  \frac{\partial}{\partial \bar{u}_k} q &= mE_km^{-1},\\
   \frac{\partial}{\partial {u}_k} m &=-[E_k,q]\cdot m,\\
    \frac{\partial}{\partial \bar{u}_k} m &=  m\cdot [E_k,\overline{q}].
\end{align}
These equations arise precisely as the compatibility conditions of the following linear PDE system: 
\begin{align}\label{eq:symmtric for tt}
   \frac{\partial}{\partial \lambda}\varphi&= (U-\lambda^{-1}[U,q]-\lambda^{-2}m\bar{U}m^{-1})\varphi,  \\ \label{eq:orig tt1}
    \frac{\partial}{\partial u_k} \varphi &= (\lambda E_{kk}-[E_{kk},q])\varphi,\\ \label{eq:orig tt2}
    \frac{\partial}{\partial \bar{u}_k}\varphi &= \lambda^{-1}mE_{kk}m^{-1}\varphi.
\end{align}
 Equation \eqref{eq:symmtric for tt} is precisely the special case to equation \eqref{introeq}. Therefore, the $tt^{*}$ equations are equivalent to the isomonodromy equations we are studying. Furthermore, suppose $u_1=0$ (for arbitrary $u_1$, it suffices to let  $\phi=\mathrm{e}^{-u_1\lambda-\frac{\bar{u}_1}{\lambda}}\varphi$ and consider $\phi$), and let $\xi=|u_n|\cdot\lambda$, \eqref{eq:symmtric for tt} becomes
 \begin{align}\label{eq:change sys for tteq}
      \frac{\partial}{\partial \xi}\varphi&= \left(\frac{1}{|u_n|}U-\lambda^{-1}[U,q]-\lambda^{-2}
      |u_n|m\bar{U}\bar{m}\right)\varphi.
 \end{align}
 The isomonodromy equations of \eqref{eq:change sys for tteq} are still the $tt^{*}$ equations, and our $(\mathbf{z},t,\mathbf{w})$ coordinates now transform as follows: 
 \begin{align*}
     t=|u_n|^2,\ z_{n-1}=\frac{u_n}{u_{n-1}},\cdots,z_{2}=\frac{u_3}{u_2},\ z_1=\frac{u_2}{|u_n|}, \ w_{n-1}=\bar{z}_{n-1},\cdots, w_2=\bar{z}_2,
 \end{align*}
thus $t$ represents scale for $\varphi$ and $z_1$ represents rotation angle for $\varphi$.

Similar to the $tt^{*}$-Toda case, by the Painlev\'e property, $q(\mathbf{z},\bar{\mathbf{z}}),\ m(\mathbf{z},\bar{\mathbf{z}})$ can be analytically continued to multi-valued meromorphic functions on $\mathbb{C}^{2n}\setminus\Delta_{\mathbf{u}, \mathbf{v}}$. Thus applying Theorem \ref{mainthm}, we have 
\begin{cor}\label{prop:asym of gener tt}
     Give an anti-symmetric matrix $\widehat{A}_0$, and an orthogonal, Hermitian matrix $m_0$, such that $\widehat{A}_0=-m_0\overline{\widehat{A}_0}\overline{m_0}$, and satisfy the eigenvalues conditions \eqref{eq:eigen condition of Ahat and Atil}.
Let $\widehat{A}\left(\mathbf{z},\widehat{A}_0\right)$ (denote by $\widehat{A}$ below)  be the solutions of \eqref{eq iso for hatA z} with boundary value $\widehat{A}_0$, as provided by Theorem~\ref{thm:iso asy by tangxu}; and let $\mathbf{z}=(z_1,...,z_{n-1})$,
\begin{align*}
   & \widehat{m}(\mathbf{z},\bar{\mathbf{z}}):=\left|z_1
   \right|^{2\widehat{A}}\cdot\widehat{\mathfrak{C}}^{-1}\cdot m_0\cdot \mathrm{e}^{\pi\mathrm{i}\widehat{A}_0}\cdot\overline{\widehat{\mathfrak{C}}}\cdot\mathrm{e}^{\pi\mathrm{i}\widehat{A}}.
\end{align*}
Here $\widehat{\mathfrak{C}}$ is  the solutions of \eqref{eq: eq of one irr connection mat} corresponding to $U=\mathrm{diag}(0,1,z_2,\cdots,z_2\cdots z_{n-1}),\ \Phi=\widehat{A}(\mathbf{z},\widehat{A}_0)$, such that the asymptotic constant in \eqref{eq:leading term of connection mat in one irr} is $\mathrm{Id}$. And $\overline{\widehat{\mathfrak{C}}}$ is the complex conjugate of $\widehat{\mathfrak{C}}$. Then there exists a unique solution $(q(t,\mathbf{z},\bar{\mathbf{z}}),m(t,\mathbf{z},\bar{\mathbf{z}}))$ of $tt^{*}$ equations satisfying the following asymptotic behavior near $t=0$:
\begin{align}
    A=-[U,q]\rightarrow \widehat{A},\quad
   t^{-\widehat{A}}m\rightarrow \widehat{m},\quad \text{as}\ t\rightarrow 0.
\end{align}

Furthermore, $(\widehat{A},\widehat{m})$ satisfy the following asymptotic behaviors as $z_{n-1}\rightarrow \infty, \cdots z_2\rightarrow \infty$ successively:
\begin{align*}
     &\lim_{z_{k}\rightarrow\infty} \delta_{k}\widehat{A}_k=\delta_k\widehat{A}_{k-1}, \ \lim_{z_k\rightarrow\infty} z_k^{\delta_k\widehat{A}_{k-1}}\cdot\widehat{A}_{k}\cdot z_k^{-\delta_k\widehat{A}_{k-1}}=\widehat{A}_{k-1},\ k=2,...,n-1,\ \widehat{A}_1=\widehat{A}_0;\\
      &\overrightarrow{\prod_{k=1}^{n-1}}\left(z_{k}^{\delta_{k}\widehat{A}_{k-1}}z_{k}^{-\widehat{A}_{k}}\right)\left(z_1\cdot |z_2|^2\cdots |z_{n-1}|^2\right)^{\widehat{A}}\widehat{m} \overleftarrow{\prod_{k=2}^{n-1}}\overline{\left(z_{k}^{\widehat{A}_{k}}z_{k}^{-\delta_{k}\widehat{A}_{k-1}}\right)}\rightarrow m_0.
\end{align*}

On the other hand, when the monodromy data of system \eqref{eq:change sys for tteq} or \eqref{eq:symmtric for tt} is strictly log-confined as in Definition \ref{def:log-conf cond}, the corresponding solutions of $tt^{*}$ equations must satisfy this asymptotic behaviors near $t=0$.
\end{cor}

 \begin{proof}
    Starting from the pair $(\widehat{A}_0, m_0)$ satisfying the conditions of the corollary, the relation $\widetilde{A}_0 = -m_0^{-1} \widehat{A}_0 m_0 = \overline{\widehat{A}_0}$ implies that the shrinking conditions for $\widehat{A}_0$ are equivalent to those for $\widetilde{A}_0$. Following Proposition \ref{prop:from A0,G0 to hatA, hatG}, we employ equation \eqref{eq:def of hatG in RH prob} to construct $\widehat{m}(\mathbf{z}, \mathbf{w})$ from $\widehat{A}(\mathbf{z}, \widehat{A}_0)$ and $\widetilde{A}(\mathbf{w}, \widetilde{A}_0)$. By Theorem \ref{thm:asy of iso A from hatA}, we obtain the multi-valued meromorphic solutions to the isomonodromy equations, denoted by $(A(\mathbf{z}, t, \mathbf{w}), m(\mathbf{z}, t, \mathbf{w}))$, on the domain $\mathbb{C}^{2n} \setminus \Delta_{\mathbf{u}, \mathbf{v}}$.

    Restricting these solutions to $V = \overline{U}$ and invoking the uniqueness result in Theorem \ref{thm:iso asy by tangxu}, we have $\widetilde{A}(\bar{z}, \widetilde{A}_0) = \overline{\widehat{A}}(z, \widehat{A}_0)$ and $\overline{C}_{-d}(U, \widehat{A}) = C_{d}(\overline{U}, \overline{\widehat{A}}) = C_d(\overline{U}, \widetilde{A})$. Consequently, $\widehat{m}(\mathbf{z}, \bar{\mathbf{z}})$ can be represented in the form specified in this corollary. Furthermore, the uniqueness property from Theorem \ref{thm:asy of iso A from hatA} yields the following symmetries:
    \begin{align*}
        -A^{T}(\mathbf{z}, t, \bar{\mathbf{z}}) &= A(\mathbf{z}, t, \bar{\mathbf{z}}), \quad m^{-T}(\mathbf{z}, t, \bar{\mathbf{z}}) = m(\mathbf{z}, t, \bar{\mathbf{z}}); \\
        -m(\mathbf{z}, t, \bar{\mathbf{z}}) \overline{A}(\mathbf{z}, t, \bar{\mathbf{z}}) \overline{m}^{-1}(\mathbf{z}, t, \bar{\mathbf{z}}) &= A(\mathbf{z}, t, \bar{\mathbf{z}}), \quad \overline{m}^{T}(\mathbf{z}, t, \bar{\mathbf{z}}) = m(\mathbf{z}, t, \bar{\mathbf{z}}).
    \end{align*}

    Conversely, suppose the monodromy data of system \eqref{eq:change sys for tteq} or \eqref{eq:symmtric for tt} is strictly log-confined. In this case, Theorem \ref{thm:almost every solution shrinking} ensures the existence of multi-valued meromorphic solutions $(A(\mathbf{z}, t, \mathbf{w}), m(\mathbf{z}, t, \mathbf{w}))$ on $\mathbb{C}^{2n} \setminus \Delta_{\mathbf{u}, \mathbf{v}}$. When these solutions are restricted to $U = \overline{V}$, the invariance of the monodromy data and Lemma \ref{lem:uniqueness in RH} imply that $(A(\mathbf{z}, t, \bar{\mathbf{z}}), m(\mathbf{z}, t, \bar{\mathbf{z}}))$ are precisely the solutions to the $tt^*$ equations (as given in system \eqref{eq:change sys for tteq} or \eqref{eq:symmtric for tt}) from which we initially started. Finally, from the algebraic relations between $A$ and $m$, we can deduce through a limiting process that $\widehat{A}_0$ is anti-symmetric, $m_0$ is orthogonal and Hermitian, and the identity $\widehat{A}_0 = -m_0 \overline{\widehat{A}_0} m_0^{-1}$ holds.
\end{proof}

\medskip

\bibliographystyle{plain}
\bibliography{2irrsys}

\end{document}